\documentclass[10pt,oneside]{article}
\usepackage{blindtext}
\usepackage{amsmath,amsfonts,amssymb,amsthm,eucal}
\usepackage{microtype}
\usepackage[colorlinks, citecolor=blue, linkcolor=blue, urlcolor=Maroon, filecolor=Maroon, linktocpage=true]{hyperref}
\usepackage[usenames,dvipsnames,svgnames,table]{xcolor}
\usepackage[normalem]{ulem}
\usepackage{lmodern}
\usepackage{enumerate}
\usepackage{bbm}
\usepackage{tikz}
\usepackage{tikz-cd}
\usepackage[font=footnotesize,labelfont=bf]{caption}
\usepackage{subcaption}
\usepackage{framed}
\usepackage[nottoc,notlot,notlof]{tocbibind}
\usepackage{tocloft}
\usepackage{morefloats}
\usepackage{float}
\usepackage[backgroundcolor=LightBlue,linecolor=LightBlue]{todonotes}

\newtheorem{Th}{Theorem}[section]
\newtheorem{Prop}[Th]{Proposition}

\newtheorem{Cor}[Th]{Corollary}
\newtheorem{Lem}[Th]{Lemma}
\newtheorem{Def}[Th]{Definition}

\theoremstyle{definition}
\newtheorem{rem}[Th]{Remark}
\newtheorem{Ex}[Th]{Example}
\newtheorem*{notation}{Notation}	
\usepackage[a4paper,
	hcentering,
	text={453pt,686pt},
]{geometry}
\allowdisplaybreaks
\catcode`@=11
\@addtoreset{equation}{section}
\catcode`@=12
\begin{document}
%
%
\vskip 1.0 true cm  
\begin{center}\noindent
\textbf{\Large Properties of the moduli set of complete\\ connected projective special real manifolds}\\[2em]
{\fontseries{m}\fontfamily{cmss}\selectfont \large David\ Lindemann}\\[1em] 
{\small
Department of Mathematics and Center for Mathematical Physics\\
University of Hamburg,
Bundesstra{\ss}e 55, D-20146 Hamburg, Germany\\  
\texttt{david.lindemann@uni-hamburg.de}
}
\end{center}
\vspace{1em}
\begin{abstract}  
\noindent
We construct a compact convex generating set $\mathcal{C}_n$ of the moduli set of closed connected projective special real manifolds of fixed dimension $n$. We show that a closed connected projective special real manifold corresponds to an inner point of $\mathcal{C}_n$ if and only if it has regular boundary behaviour. Our results can be used to describe deformations of 5d supergravity theories with complete scalar geometries.\\

\end{abstract}
\tableofcontents
\section{Introduction}
Projective special real (short: PSR) manifolds are centro-affine hypersurfaces that are contained in a positive level set of a hyperbolic homogeneous cubic polynomial and consist only of hyperbolic points of said polynomial. Their study is closely related to the theory of supergravity in five spacetime dimensions, where they are the scalar manifolds \cite{GST,DV}. By the so-called supergravity r-map and c-map their geometry is related to the study of projective special K\"ahler manifolds \cite{ACD,F} and quaternionic K\"ahler manifolds \cite{DV}.
Completeness properties of the r- and c-map have been studied in \cite{CHM} where it was shown that both the supergravity r-map (which associates to an $n$-dimensional PSR manifold a $2n+2$-dimensional projective special K\"ahler manifold) and the supergravity c-map (which associates to an $m$-dimensional projective special K\"ahler domain a $2m+4$-dimensional quaternionic K\"ahler manifold) preserve geodesic completeness. This yields an explicit way to obtain examples of complete non-compact quaternionic K\"ahler manifolds of negative scalar curvature by taking a complete connected PSR manifold and composing the supergravity r- and c-map (this is called the supergravity q-map \cite{DV}). Note that it has been proven in \cite{FS} that all manifold in the image of the supergravity c-map have negative scalar curvature, independent of completeness. In \cite{CNS}, hyperbolic centro-affine hypersurfaces have been studied and it was shown PSR manifolds are geodesically complete if and only if they are closed in their ambient space \cite[Thm.\,2.5]{CNS}.

The mentioned results motivate the study of the moduli set of all closed connected PSR manifolds of fixed dimension $n$. By the term ``moduli set'' we mean the set of isomorphism classes of closed PSR manifolds. Two closed connected PSR manifold are isomorphic if they are contained in the same $\mathrm{GL}(n+1)$-orbit, where the corresponding $\mathrm{GL}(n+1)$-action is induced by the action on the ambient space $\mathbb{R}^{n+1}$. We will call two isomorphic PSR manifolds equivalent. In dimension $1$ and $2$, all PSR manifolds without the restriction of being closed in the ambient space have been completely classified up to equivalence, cf. \cite{CHM} for curves and \cite{CDL} for surfaces, but in higher dimensions no complete classification has been found yet. However, there exist classification results when further restricting the geometry. PSR manifolds that are homogeneous spaces have been classified in \cite{DV}, and PSR manifolds related to reducible cubic polynomials have been classified in \cite{CDJL}.

Our main result in this work is the construction of a generating set of the moduli set of closed connected PSR manifolds in all dimensions:
\begin{Th}
\label{thm_Cn}
Let $\left(\begin{smallmatrix}x\\ y\end{smallmatrix}\right)=(x,y_1,\ldots,y_n)^T$ denote linear coordinates on $\mathbb{R}^{n+1}$, let $\langle\cdot,\cdot\rangle$ denote the standard Euclidean scalar product on $\mathbb{R}^n$ induced by the choice of the coordinates $(y_1,\ldots,y_n)^T$, and let $\|\cdot\|$ denote the corresponding norm.
For all $n\in\mathbb{N}$, the set of hyperbolic homogeneous cubic polynomials
\begin{equation*}
\mathcal{C}_n:=\left\{x^3-x\langle y,y\rangle + P_3(y)\ \left| \  \max\limits_{\|y\|=1}P_3(y)\leq \frac{2}{3\sqrt{3}}\right\}\right.
\end{equation*}
is a generating set for the moduli set of $n$-dimensional closed connected PSR manifolds under the action of $\mathrm{GL}(n+1)$, meaning that for every closed connected PSR manifold $\mathcal{H}$ of dimension $n$ there exists an element $\widetilde{h}\in\mathcal{C}_n$, such that the connected component $\widetilde{\mathcal{H}}\subset\{\widetilde{h}=1\}$ containing the point $\left(\begin{smallmatrix}
x\\ y
\end{smallmatrix}
\right)=\left(\begin{smallmatrix}
1\\ 0
\end{smallmatrix}
\right)\in\left\{\widetilde{h}=1\right\}\subset\mathbb{R}^{n+1}$ is equivalent to $\mathcal{H}$
and, conversely, each $h\in\mathcal{C}_n$ defines a closed connected PSR manifold which is given by the connected component of $\{h=1\}$ that contains the point $\left(\begin{smallmatrix}
x\\ y
\end{smallmatrix}
\right)=\left(\begin{smallmatrix}
1\\ 0
\end{smallmatrix}
\right)$.
The set $\mathcal{C}_n\subset\mathrm{Sym}^3\left(\mathbb{R}^{n+1}\right)^*$ is a uniformly bounded compact convex subset of the $\frac{n^3+3n^2+2n}{6}$-dimensional affine subspace
\begin{equation*}
\left\{x^3-x\langle y,y\rangle + P_3(y)\ \left|\ P_3\in\mathrm{Sym}^3\left(\mathbb{R}^n\right)^*\right\}\right.\subset \mathrm{Sym}^3\left(\mathbb{R}^{n+1}\right)^*.
\end{equation*}
The boundary of $\mathcal{C}_n$, that is $\partial \mathcal{C}_n$, is a continuous submanifold of $\mathrm{Sym}^3\left(\mathbb{R}^{n+1}\right)^*$.
Furthermore, $\widetilde{h}\in\partial\mathcal{C}_n$ if and only if the initial $\mathcal{H}$ does not have regular boundary behaviour.
\end{Th}
As we will describe in more detail in Remark \ref{rem_Cn_comparison}, it is in general not difficult to construct a bounded generating set for the moduli set of closed connected PSR manifolds. However, it was up until now not clear that we can find a bounded generating set of dimension less than the dimension of the vector space of cubic homogeneous polynomials in $n+1$ variables (corresponding to $n$-dimensional PSR manifolds), and the existence of such a generating set that is additionally closed and convex is also far from obvious. Furthermore, points in $\partial\mathcal{C}_n$ correspond precisely to closed connected PSR manifolds that do not have regular boundary behaviour. A PSR manifold $\mathcal{H}\subset\{h=1\}\subset\mathbb{R}^{n+1}$ having regular boundary behaviour means that
the negative Hessian of $h$, $-\partial^2h$, viewed as a bilinear form on $\mathbb{R}^{n+1}$ has only $1$-dimensional kernel along $\partial U \setminus\{0\}$, and $dh_p\ne 0$ for all $p\in \partial U\setminus\{0\}$, where $U:=\mathbb{R}_{>0}\cdot\mathcal{H}\subset\mathbb{R}^{n+1}$ denotes the cone spanned by $\mathcal{H}$. 
The main ingredient in order to show that points in $\partial\mathcal{C}_n$ correspond precisely to closed connected PSR manifolds with non-regular boundary behaviour is Theorem \ref{thm_irregularity_implications_CCPSR}, where we prove that for a closed connected PSR manifold $\mathcal{H}\subset\{h=1\}$ it suffices to show that $dh_p\ne 0$ for all $p\in \partial U\setminus\{0\}$ in order to show that $\mathcal{H}$ has regular boundary behaviour.
Note at this point that we consider the moduli set of PSR manifolds as a set, not as a topological space and, hence, we do not use the term ``moduli space''. Choosing and describing a suitable topology on the moduli set of PSR manifolds, which would then justify the term moduli space, is an interesting task for future studies.

While the generating set $\mathcal{C}_n$ is not a 1:1 description of the moduli set, its compactness and convexity properties imply many new properties of the moduli set. Compactness shows that for fixed dimension, the sectional curvatures and scalar curvature of complete PSR manifolds are bounded from above and below with bounds depending only on the dimension, see Corollary \ref{cor_general_CCPSR_S_bounds}. To prove said result we will develop curvature formulas for PSR manifolds (cf. Proposition \ref{prop_scal_GPSR} for a larger class of manifolds outlined below) and use a technical result for a standard form of PSR manifolds, see Proposition \ref{prop_std_form_h}. Convexity enables us to explicitly describe a curve in the class of closed connected PSR manifolds of dimension $n$ that connects any two given closed connected PSR manifolds of dimension $n$, see Corollary \ref{cor_curve_between_CCPSR} and Example \ref{ex_curves_CCPSRs}. The compactness property of $\mathcal{C}_n$ can also be used to find an alternative proof that closed PSR manifolds are complete, see Proposition \ref{prop_alternative_CCPSR_completeness_proof_numero2}. Furthermore, these properties also carry over to the supergravity r- and q-map and thereby in particular yield an explicit way to deform two quaternionic K\"ahler manifolds in the image of the supergravity q-map (restricted to closed connected PSR manifolds) into each other, which is by Theorem \ref{thm_Cn} always possible independent of their initial choice. This in particular means that we have developed a way to deform two distinct 5d supergravity theories, which are complete in the sense that their scalar manifolds (which are PSR manifolds) are geodesically complete, into each other via a curve of such theories which corresponds point-wise to other (but possibly equivalent) theories that are also complete in the stated sense. This property also carries over to 4d and 3d supergravity theories obtainable via the supergravity r- and q-map, respectively.

In this work we also study what we call generalized PSR manifolds (GPSR manifolds for short). A GPSR manifold (of degree $\tau$) is a centro-affine hypersurface contained in the level set of a hyperbolic homogeneous polynomial of degree $\tau\geq 3$, so in comparison with PSR manifolds (which we view as a subclass of GPSR manifolds) we also allow polynomial of degree higher than $3$. GPSR manifolds with corresponding polynomial of degree $\tau\geq 4$ do not have a similar motivation from supergravity, but as for PSR manifolds they appear as level sets in the K\"ahler cones (or, more generally, index/positive cones) of compact K\"ahler $\tau$-folds. Many of our technical results like Proposition \ref{prop_std_form_h} and Proposition \ref{prop_hom_spaces_deltaPk_condition} hold for GPSR manifolds of arbitrary degree $\tau\geq 3$. However, it turned out to be very hard to prove a statement similar to Theorem \ref{thm_Cn} for GPSR manifolds of degree $\tau\geq 4$. In \cite[Thm.\,7.2]{Li} closed connected GPSR curves of degree $4$ have been classified up to isomorphisms, and it turned out that not only was this far more complicated than an analogous classification of closed connected PSR curves (cf. \cite[Rem.\,7.4]{Li}, which proceeds differently than the complete classification of PSR curves in \cite{CHM}), but also that even in this most simple case of GPSR manifolds of degree $\geq 4$ we cannot find a compact convex generating set in the way we did for closed connected PSR manifolds. In future studies, we will expand the work on GPSR manifolds of degree $4$ that was done in \cite[Sect.\,7]{Li}. In particular, it is an open problem whether or not such manifolds are geodesically complete if they are closed in the ambient space. The latter property has first been proven for closed PSR manifolds in \cite[Thm.\,2.5]{CNS} and we will present an alternative proof in Proposition \ref{prop_alternative_CCPSR_completeness_proof_numero2} which uses our main Theorem \ref{thm_Cn}. Another interesting open question is the generalization of the supergravity q-map to connected GPSR manifolds of degree $\tau\geq 4$.

Aside from the study of supergravity, PSR manifolds also appear in the study of compact K\"ahler threefolds. For a compact K\"ahler threefold $X$ consider the cubic homogeneous polynomial
	\begin{equation*}
		h:H^{1,1}(X,\mathbb{R})\to\mathbb{R},\quad [\omega]\mapsto\int_X \omega^3.
	\end{equation*}
By the Hodge-Riemann bilinear relations, we find that each point in the cone of K\"ahler classes $\mathcal{K}\subset H^{1,1}(X,\mathbb{R})$ is a hyperbolic point of $h$, and so the set $\mathcal{H}:=\{h=1\}\cap\mathcal{K}$ is a PSR manifold. This point of view has been studied in particular for Calabi-Yau threefolds. In \cite{W,TW}, the sectional curvatures of, albeit not under this name, PSR manifolds have been studied from this point of view and in \cite{KW} the defining trilinear forms and their relation to the second and third Chern classes have been studied.
Analogously, GPSR manifolds of degree $\tau\geq 4$ appear as level sets in cones of K\"ahler classes of compact K\"ahler manifolds $X$ of complex dimension $\tau$. The corresponding hyperbolic homogeneous polynomial is then given by $h:[\omega]\mapsto \int_X\omega^\tau$. For $\tau=4$, the curvature of a GPSR manifold (again, not under this name) has been studied in \cite[Sect.\,4]{T}. An expanded study of the different curvatures of GPSR manifolds, with applications to both the geometry of K\"ahler cones and supergravity, will be the subject of an upcoming paper.

\paragraph*{Acknowledgements}
This work was partly supported by the German Science Foundation (DFG) under the Research Training Group 1670 and the Collaborative Research Center (SFB) 676.
It is based on the main results of Section 5 and technical results of Sections 3 and 4 of my doctoral thesis. Proposition \ref{prop_homimpliessingatinf} is new and not part of my doctoral thesis. I would like to thank my supervisor Vicente Cort\'es for his continuous support during the writing of my thesis, and I would also like to thank my second examiners Andriy Haydys and Antonio Mart\'inez.

\section{Preliminaries}
We will quickly review the basics of centro-affine differential geometry that are needed in this work.
\begin{Def}
\label{centroaffinehypersurfacedef}
Let $f=(f_1,\ldots,f_{n+1})^T :M\to \mathbb{R}^{n+1}$ be a hypersurface immersion. It is called a \textsf{centro-affine hypersurface immersion} if the position vector field $\xi\in\Gamma(T\mathbb{R}^{n+1})$, $\xi_p=p$ for all $p\in \mathbb{R}^{n+1}$ under the canonical identification, is transversal along $f$, that is
\begin{equation*}
df(T_pM)\oplus \mathbb{R}\xi_{f(p)}=T_{f(p)}\mathbb{R}^{n+1}
\end{equation*} 
for all $p\in M$, where $\mathbb{R}\xi_{f(p)}$ denotes the $1$-dimensional vector subspace spanned by $\xi_{f(p)}$ of $T_{f(p)}\mathbb{R}^{n+1}$. We will omit writing down the map $f$ if it is clear from the context and simply call $M$ a \textsf{centro-affine hypersurface}.
\end{Def}
If $f$ is additionally an embedding, it will be called a centro-affine hypersurface embedding.
The Gau{\ss} equation for centro-affine hypersurface immersions $f:M\to\mathbb{R}^{n+1}$ is of the form
\begin{equation}
\overline{\nabla}_X(df(Y))=df(\nabla_X Y) + g(X,Y)\xi_{f},\label{eqn_centro_affine_gauss}
\end{equation}
where $\xi_f$ denotes the position vector field along $f$. This leads to the following definition.
\begin{Def}
\label{def_cafform}
Let $f:M\to\mathbb{R}^{n+1}$ be a centro-affine hypersuface immersion. The induced connection $\nabla$ in $TM$ \eqref{eqn_centro_affine_gauss} is called the \textsf{centro-affine connection}, the symmetric $(0,2)$-tensor $g\in\Gamma\left(\mathrm{Sym}^2T^\ast M\right)$ is called the \textsf{centro-affine fundamental form}. The centro-affine hypersurface immersion $f$ is called \textsf{non-degenerate} if $g$ is non-degenerate, \textsf{definite} if $g$ is definite, i.e. either positive or negative definite, \textsf{elliptic} if $g<0$, i.e. negative definite, and \textsf{hyperbolic} if $g>0$, i.e. positive definite.
\end{Def}
A tool to obtain explicit examples of centro-affine hypersurfaces together with their centro-affine fundamental form is explained in the following proposition.
\begin{Prop}\label{propCAHSE}
Let $U\subset\mathbb{R}^{n+1}$, $n\in\mathbb{N}\cup\{0\}$, be open and invariant under positive rescaling, i.e. $rp\in U$ for all $r>0$ and $p\in U$. Let $h:U\to\mathbb{R}$ be a homogeneous function of degree $k>1$. Assume that the level set $\{p\in U\mid h(p)=1\}$ is not empty and let $\mathcal{H}\subset \{p\in U\mid h(p)=1\}$ be open subset. Then the inclusion map $\iota:\mathcal{H}\to \mathbb{R}^{n+1}$ is a centro-affine hypersurface embedding with centro-affine fundamental form $g=-\frac{1}{k}\iota^\ast(\overline{\nabla}^2h)$, where $\overline{\nabla}$ denotes the canonical flat connection in $T\mathbb{R}^{n+1}$.
\begin{proof}
For a proof of this statement in a slightly more general setting see \cite[Prop.\,1.3]{CNS}.
\end{proof}
\end{Prop}
If $\mathbb{R}^{n+1}$ is equipped with linear coordinates, we will write $\partial^2$ instead of $\overline{\nabla}^2$. We will also omit writing down the map $\iota$ for an embedding $\iota:M\to \mathbb{R}^{n+1}$, that is we will write $M\subset\mathbb{R}^{n+1}$ instead of $\iota(M)\subset\mathbb{R}^{n+1}$. In this work we will study centro-affine hypersurface embeddings where $h$ as in Proposition \ref{propCAHSE} is a hyperbolic homogeneous polynomial of degree $\tau\geq 3$ with Riemannian centro-affine fundamental form.
\begin{Def}
\label{hyperbolicpointdef}
Let $U\subset \mathbb{R}^{n+1}$ be an open subset that is invariant under multiplication with positive real numbers, and let $h:U\to\mathbb{R}$ be a homogeneous function of degree $\tau>1$. Then a point $p\in \{h>0\}$ is called a \textsf{hyperbolic point (of $h$)} if $-\partial^2 h_p$ has signature $(n,1)$, i.e. it is of Lorentz type. A function $h$ that has at least one hyperbolic point is called a \textsf{hyperbolic homogeneous function}.
\end{Def}
Observe that this implies that for a hyperbolic point $p$ of $h$ we have $-\partial^2h_p|_{\ker(dh_p)\times \ker(dh_p)}>0$, which follows from $-\partial^2h_p(p,p)=-\tau(\tau-1)h(p)<0$ and $-\partial h_p(p,\cdot)=-(\tau-1)dh_p$.
\begin{Def}
\label{hyperboliccentroaffinehsDEF}
Let $\mathcal{H}\subset\{h=1\}$ be a centro-affine hypersurface as in Proposition \ref{propCAHSE}. Then $\mathcal{H}$ is called a \textsf{hyperbolic centro-affine hypersurface} if it consists only of hyperbolic points.
\end{Def}
Note that the above definition of hyperbolic centro-affine hypersurface coincides with Definition \ref{def_cafform} for $f$ the inclusion map $\iota:\mathcal{H}\to\mathbb{R}^{n+1}$. Hyperbolic centro-affine hypersurfaces equipped with their respective centro-affine fundamental form $(\mathcal{H},g)$ are Riemannian manifolds. Continuity of the determinant implies that a connected non-degenerate centro-affine hypersurface $\mathcal{H}$ is hyperbolic if and only if it contains one hyperbolic point. Note that hyperbolicity at a point is an open condition in the sense every homogeneous function $h:U\to\mathbb{R}$ as in Definition \ref{hyperbolicpointdef} with a hyperbolic point $p$ is hyperbolic on some open neighbourhood $V\subset U$ of $p$, which also follows from the continuity of the determinant of $-\partial^2 h$. Hence, for every hyperbolic homogeneous function $h$ of degree $\tau>1$ we can choose an open subset $\mathcal{H}\subset\{h=1\}$ that is a hyperbolic centro-affine hypersurface.
\begin{Def}
\label{hyppolyDEF}
A homogeneous polynomial $h:\mathbb{R}^{n+1}\to \mathbb{R}$ of degree $\tau\geq 2$ is called a \textsf{hyperbolic homogeneous polynomial} if there exists $p\in\{h>0\}$, such that $p$ is a hyperbolic point of $h$.
\end{Def}

Next, we will introduce a notion of equivalence for homogeneous polynomials and hypersurfaces in their respective level sets. 

\begin{Def}
\label{equivalenceDEF}
Two hyperbolic homogeneous polynomials $h,\overline{h}:\mathbb{R}^{n+1}\to\mathbb{R}$ of degree $\tau\geq 2$ are called \textsf{equivalent} if there exists a linear transformation $A\in\mathrm{GL}(n+1)$, such that $h\circ A=\overline{h}$. Two connected hyperbolic centro-affine hypersurfaces $\mathcal{H}$ and $\overline{\mathcal{H}}$ contained in a level set of $h$ and $\overline{h}$, respectively, are called equivalent if $h$ and $\overline{h}$ are equivalent and $A(\overline{\mathcal{H}})= \mathcal{H}$ for $h\circ A=\overline{h}$.
\end{Def}

Two centro-affine hyperbolic hypersurfaces being equivalent in particular implies the following statement.

\begin{Lem}
Any two equivalent connected centro-affine hyperbolic hypersurfaces $\mathcal{H}$ and $\overline{\mathcal{H}}$ defined by hyperbolic homogeneous polynomials of degree $\tau\geq 2$ $h,\overline{h}:\mathbb{R}^{n+1}\to\mathbb{R}$ as in Definition \ref{equivalenceDEF}, respectively, are isometric.
\begin{proof}
Let $A\in\mathrm{GL}(n+1)$, such that $h\circ A=\overline{h}$. Since $A:\mathbb{R}^{n+1}\to\mathbb{R}^{n+1}$ is linear, we get
\begin{equation*}
-\partial^2\overline{h}_p(\cdot,\cdot)=-\partial^2 h_{Ap}(A\cdot,A\cdot)=A^\ast(-\partial^2 h)_p.
\end{equation*}
By restricting the above equation to $T\overline{\mathcal{H}}$, respectively $T\mathcal{H}$, we find with Proposition \ref{propCAHSE} that $\mathcal{H}$ and $\overline{\mathcal{H}}$ are indeed isometric and one isometry is given by the linear transformation $A$ relating the two polynomials $h$ and $\overline{h}$.
\end{proof}
\end{Lem}

An additional topological property of centro-affine hypersurfaces that we will often impose is the following.
\begin{Def}
\label{def_maximality}
	A connected hyperbolic centro-affine hypersurface $\mathcal{H}\subset\{h=1\}$ is called \textsf{maximal} if it is a connected component of the set $\{h=1\}\cap\{\text{hyperbolic points of }h\}$.
\end{Def}

Now we will introduce the centro-affine hypersurfaces that are our main focus of study in this work.

\begin{Def}
\label{def_GPSR_mfs}
Let $n\in\mathbb{N}\cup\{0\}$ and $\mathcal{H}\subset\{h=1\}\subset\mathbb{R}^{n+1}$ be an $n$-dimensional hyperbolic centro-affine hypersurface as in Definition \ref{hyperboliccentroaffinehsDEF} contained in the level set of a hyperbolic homogeneous polynomial of degree $\tau\geq 3$. Then $\mathcal{H}$ will be called a \textsf{GPSR manifold} (for \textsf{G}eneralised \textsf{P}rojective \textsf{S}pecial \textsf{R}eal \textsf{manifold}) \textsf{of degree $\tau$}. If we further assume that $\mathcal{H}$ is closed and connected as a subset of $\mathbb{R}^{n+1}$, we will call $\mathcal{H}$ a \textsf{CCGPSR manifold} (for \textsf{C}losed \textsf{C}onnected \textsf{GPSR} \textsf{manifold}) \textsf{of degree $\tau$}.
For $\tau=3$, we will call $\mathcal{H}$ a \textsf{PSR manifold}, or a \textsf{CCPSR manifold} if it is closed in its ambient space.
As a convention we regard the set of CCPSR manifolds as a subset of the set of CCGPSR manifolds.
\end{Def}

If the degree $\tau\geq 3$ of a GPSR manifold is not of particular importance, we will omit the phrase ``of degree $\tau$''. Recall that according to Definition \ref{equivalenceDEF}, two CCGPSR manifolds of the same degree are called equivalent if they are related by a linear change of coordinates of the ambient space. Also note that CCGPSR manifolds are automatically maximal in the sense of Definition \ref{def_maximality}.

\begin{Lem}
\label{caff_lemma}
Let $\mathcal{H}\subset\{h=1\}\subset\mathbb{R}^{n+1}$ be an $n$-dimensional GPSR manifold of degree $\tau\geq 3$. Then its centro-affine fundamental form $g_{\mathcal{H}}$ is given by
\begin{equation*}
g_{\mathcal{H}}=-\frac{1}{\tau}\partial^2 h|_{T\mathcal{H}\times T\mathcal{H}},
\end{equation*}
where $\partial^2$ is determined by the chosen linear coordinates on the ambient space $\mathbb{R}^{n+1}$.
\begin{proof}
This follows immediately from Proposition \ref{propCAHSE}.
\end{proof}
\end{Lem}

When studying CCGPSR manifolds of degree $\tau\geq 3$ with the aim of some kind of classification, it is useful to introduce the notion of their respective moduli sets.
\begin{Def}
\label{def_moduli_space_CCGPSR_mfs}
Let $n\in\mathbb{N}\cup\{0\}$. We define the \textsf{moduli set of $n$-dimensional CCGPSR manifolds of degree $\tau$} to be the set of equivalence classes
\begin{equation*}
\left\{[\mathcal{H}]\ |\ \mathcal{H}\text{ is a }CCGPSR\text{ manifold of degree }\tau,\ \dim(\mathcal{H})=n\right\},
\end{equation*} 
where $[\widetilde{\mathcal{H}}]=[\mathcal{H}]$ if and only if $\widetilde{\mathcal{H}}$ and $\mathcal{H}$ are equivalent. For $\tau=3$, we will call the above set the \textsf{moduli set of $n$-dimensional CCPSR manifolds}.
\end{Def}

An important topological property of the cone spanned by CCGPSR manifolds which we will need in our studies is its convexity:

\begin{Prop}
\label{prop_convexcone}
Let $\mathcal{H}\subset\{h=1\}\subset\mathbb{R}^{n+1}$ be an $n$-dimensional CCGPSR manifold. Then
\begin{equation*}
U=\mathbb{R}_{>0}\cdot\mathcal{H}=\left.\left\{rp\in\mathbb{R}^{n+1}\ \right|\ r>0,\ p\in\mathcal{H}\right\}\subset\mathbb{R}^{n+1}
\end{equation*}
is a convex cone
and the map $\mathbb{R}_{>0}\times\mathcal{H}\ni (r,p)\mapsto r\cdot p\in U$
is a diffeomorphism.
\begin{proof}
\cite[Prop.\,1.10]{CNS} for the special case of CCGPSR manifolds.
\end{proof}
\end{Prop}

We will later parametrise CCGPSR manifolds over sections of their respective spanned cone with an affinely embedded tangent space, cf. equation \eqref{dom(H)_H_iso}. A key fact that we will need in order to prove the compactness of the set $\mathcal{C}_n$ in Theorem \ref{thm_Cn} is the precompactness of these sections, which is part of the following lemma:

\begin{Lem}\label{precomp_cor}

Let $\mathcal{H}$ be a CCGPSR manifold and let $U=\mathbb{R}_{>0}\cdot \mathcal{H}$. Then for every $p\in\mathcal{H}$, the intersection
\begin{equation*}
\left(p+T_p\mathcal{H}\right)\cap U\subset p+ T_p\mathcal{H}
\end{equation*}
is open, precompact, and convex. Here $\left(p+ T_p\mathcal{H}\right)\subset\mathbb{R}^{n+1}$ denotes the affinely embedded tangent space $T_p\mathcal{H}$ in the ambient vector space $\mathbb{R}^{n+1}$ equipped with the induced subspace topology.
\begin{proof}
\cite[Lem.\,1.14]{CNS} applied to homogeneous polynomials.
\end{proof}
\end{Lem}

In order to further specify types of CCGPSR manifolds we define a certain type of boundary behaviour:

\begin{Def}
\label{def_singular_GPSR}
Let $\mathcal{H}\subset\{h=1\}\subset\mathbb{R}^{n+1}$ be a CCGPSR manifold and let $U=\mathbb{R}_{>0}\cdot\mathcal{H}$ be the corresponding convex cone. We will call $\mathcal{H}$ \textsf{singular at infinity} if there exists a point $p\in\partial U\setminus\{0\}$, such that $dh_p=0$.
\end{Def}

Definition \ref{def_singular_GPSR} is not empty for CCPSR manifolds in the sense that for each $n\geq 1$, there exists an $n$-dimensional CCPSR manifold that is singular at infinity. This is a consequence of Proposition \ref{prop_homimpliessingatinf} and the existence of homogeneous CCPSR manifolds in all dimensions, cf. \cite{CHM} for dimension 1, \cite{CDL} for dimension 2, and \cite{DV} for dimension $n\geq 3$. It will turn out that being singular at infinity or not already determines the regularity of the boundary behaviour of CCPSR manifolds in the sense of Definition \ref{def_regular_bdr_GPSR}, for the result see Theorem \ref{thm_irregularity_implications_CCPSR}.

\begin{rem}
\label{rem_difficulties_PSRclass}
A natural question that arises when studying PSR manifolds is whether it is possible to classify all closed connected PSR manifolds up to equivalence. In general, this turns out to be a very difficult question. This problem is equivalent to classifying all cubic hyperbolic homogeneous polynomials up to equivalence. One of the encountered difficulties is that being hyperbolic for a cubic homogeneous polynomial is an open condition in the sense that if $h\in\mathrm{Sym}^3(\mathbb{R}^{n+1})^\ast$ is hyperbolic and $H\in\mathrm{Sym}^3(\mathbb{R}^{n+1})^\ast$ is any cubic polynomial, then there exists an $\varepsilon>0$, such that for all $0\leq k\leq \varepsilon$ the polynomial $h+kH$ is hyperbolic. This follows easily from Sylvester's law of inertia. Furthermore, the dimension of $\mathrm{Sym}^3(\mathbb{R}^{n+1})^\ast$ grows cubically in $n$ while the dimension of $\mathrm{GL}(n+1)$ grows quadratically in $n$, so we can not expect to have only finitely many examples as $n$ grows large. In dimensions $n=1$ and $n=2$ however, cubic hyperbolic homogeneous polynomials in $2$ and $3$ variables, respectively, and the corresponding closed connected PSR manifolds have been classified up to equivalence, see \cite{CHM} for $1$-dimensional PSR manifolds and \cite{CDL} for $2$-dimensional PSR manifolds. Aside from the low-dimensional restriction, another restriction to PSR manifolds is to consider only those that are contained in the level set of a reducible cubic hyperbolic homogeneous polynomial. In this case, CCPSR manifolds are classified in any dimension \cite[Thm.\,2]{CDJL}. Lastly, there is a classification of PSR manifolds that are homogeneous spaces under the action of their respective automorphism groups, cf. \cite{DV}.
\end{rem}

In this work we will use the classification result of CCPSR surfaces:
\begin{Th}[{\cite[Thm.\,1]{CDL}}]
\label{lowdimpsrclassTHM}
Let $\mathcal{H}\subset\mathbb{R}^3$ be a CCPSR surface. Then $\mathcal{H}$ is equivalent to exactly one of the following:
\begin{enumerate}[a)]
\item $\left\{xyz=1,\ x>0,\ y>0\right\}$,
\item $\left\{x(xy-z^2)=1,\ x>0\right\}$,
\item $\left\{x(yz+x^2)=1,\ x<0,\ y>0\right\}$,
\item $\left\{z(x^2+y^2-z^2)=1,\ z<0\right\}$,
\item $\left\{x(y^2-z^2)+y^3=1,\ y<0,\ x>0\right\}$,
\item $\left\{y^2z-4x^3+3xz^2+bz^3=1,\ z<0,\ 2x>z\right\}$ for precisely one $b\in(-1,1)$.
\end{enumerate}
\end{Th}

\section{Standard form for GPSR manifolds and their curvature tensors}
In this section we are going to develop the technical tools necessary to prove our results. The first and maybe central one is the existence of a certain standard form of GPSR manifolds in dependence of a chosen reference point and which, at least locally, varies smoothly along said reference point.

\begin{Prop}
\label{prop_std_form_h}
Let $\mathcal{H}\subset\{h=1\}\subset\mathbb{R}^{n+1}$ be an $n\geq1$-dimensional connected GPSR manifold of degree $\tau\geq 3$. Then for each $p\in\mathcal{H}$ there exists a linear change of coordinates on $\mathbb{R}^{n+1}$ described by $A(p)\in\mathrm{GL}(n+1)$, such that
\begin{enumerate}[(i)]
\item $(h\circ A(p))\left(\left(\begin{smallmatrix} x\\ y\end{smallmatrix}\right)\right)=x^\tau-x^{\tau-2}\langle y,y\rangle+\sum\limits_{k=3}^{\tau}x^{\tau-k}P_k(y)$,\label{std_form_i}
\item $A(p)\cdot\left(\begin{smallmatrix}
1\\
0
\end{smallmatrix}\right)=p$,\label{std_form_ii}
\end{enumerate}
where $y=(y_1,\ldots,y_n)^T$ denote the standard linear coordinates of $\mathbb{R}^n$, $\left(\begin{smallmatrix}
x\\
y
\end{smallmatrix}\right)$ denotes the corresponding coordinates of $\mathbb{R}^{n+1}\cong \mathbb{R}\times \mathbb{R}^n$, $\left(\begin{smallmatrix}
1\\
0
\end{smallmatrix}\right)$ denotes the point $\left(\begin{smallmatrix}
x\\
y
\end{smallmatrix}\right)=\left(\begin{smallmatrix}
1\\
0
\end{smallmatrix}\right)\in\mathbb{R}^{n+1}$, and $\langle\cdot,\cdot\rangle$ denotes the standard Euclidean scalar product on $\mathbb{R}^{n}$ induced by the $y$-coordinates. Furthermore, if $\mathcal{H}$ is a CCGPSR manifold then the transformations $A(p)$ can be chosen in such a way that $A:\mathcal{H}\to\mathrm{GL}(n+1)$ is smooth. If $\mathcal{H}$ is not closed as a subset of $\mathbb{R}^{n+1}$, we can still find for each $p\in\mathcal{H}$ a subset $V\subset\mathcal{H}$ that contains $p$ and is open in the subspace-topology of $\mathcal{H}\subset\mathbb{R}^{n+1}$, such that $A:V\to\mathrm{GL}(n+1)$ can be chosen so that it is a smooth map. 
\begin{proof} First we will show that \eqref{std_form_i} and \eqref{std_form_ii} hold for all connected GPSR manifolds. Then we will prove that in the case of CCGPSR manifolds, $A:\mathcal{H}\to\mathrm{GL}(n+1)$ can be chosen to be smooth. In the case of connected GPSR manifolds which are not necessarily closed we will show that for all $p\in\mathcal{H}$ there always exists an open neighbourhood $V\subset\mathcal{H}$ of $p$, and that $A:V\to\mathrm{GL}(n+1)$ can be chosen so that it is a smooth map.

Let $\mathcal{H}\subset\mathbb{R}^{n+1}$ be a connected GPSR manifold and denote by $\langle \cdot,\cdot\rangle$ the standard Euclidean scalar product on $\mathbb{R}^{n+1}$ induced by the choice of the linear coordinates on $\mathbb{R}^{n+1}$. Let $p\in\mathcal{H}$ be arbitrary. We will differentiate between two cases.
\paragraph*{Case 1: $dh_p=r\langle p,\cdot\rangle$ for some $r\ne 0$.}\hspace*{1mm}\\
Note that the property $dh_p\in\left(\mathbb{R}\setminus\{0\}\right)\cdot\langle r,\cdot\rangle$ is preserved by changing the linear coordinates of the ambient space $\mathbb{R}^{n+1}$ by rotations in $\mathrm{SO}(n+1)$ and by positive rescaling of the linear coordinates. We can thus without loss of generality assume that $p=(1,0,\ldots,0)^T$,
and denote the linear coordinates on $\mathbb{R}^{n+1}$ by $\left(x, y_1,\ldots, y_n\right)^T$. Since $h(p)=1$ is a necessary condition for $p\in\mathcal{H}$, we find that $h$ must be of the form
\begin{equation*}
h=x^\tau + x^{\tau-1}L(y)+x^{\tau-2} Q(y,y)+(\text{terms of lower order in }x),
\end{equation*}
where $L\in\mathrm{Lin}\left(\mathbb{R}^{n},\mathbb{R}\right)$ is linear in $y$ and $Q\in\mathrm{Sym}^2\left(\mathbb{R}^n\right)^*$ is a symmetric bilinear form. We can now check that $dh_p\in\left(\mathbb{R}\setminus\{0\}\right)\cdot\langle r,\cdot\rangle$ implies $L\equiv 0$. By assumption, $p$ is a hyperbolic point of $\mathcal{H}$. We calculate
\begin{equation*}
-\partial^2h_p=\left(
\begin{tabular}{c|c}
$-\tau(\tau-1)$ &  \\ \hline
   & $\overset{}{-2Q(\cdot,\cdot)}$
\end{tabular}
\right).
\end{equation*}
The hyperbolicity of the point $p$ thus shows that $Q$ must be negative definite. Hence, after a suitable transformation of the $y$-coordinates, we find that $h$ can be transformed into the desired form
\begin{equation*}
h=x^\tau-x^{\tau-2}\langle y,y\rangle+\sum\limits_{k=3}^{\tau}x^{\tau-k}P_k(y).
\end{equation*}

\paragraph*{Case 2: $dh_p\ne r\langle p,\cdot\rangle$ for all $r\ne 0$.}\hspace*{1mm}\\
Note that in this case, $r=0$ is automatically excluded by $dh_p(p)=\tau\ne 0$. We will find a linear coordinate transformation $B\in\mathrm{GL}(n+1)$ of the ambient space $\mathbb{R}^{n+1}$ of $\mathcal{H}$, such that $Bq=p$ and 
\begin{equation}\label{eqn_Btrafo_algpmoving}
dh_{Bq}(B\cdot)=r\langle q,\cdot\rangle,
\end{equation}
which will take us to the setting of the first case since $d(h\circ B)_q=dh_{Bq}(B\cdot)$. Note that in the above equation \eqref{eqn_Btrafo_algpmoving}, $\langle\cdot,\cdot\rangle$ denotes the Euclidean scalar product induced by the new coordinates, that is the standard linear coordinates in the domain of $B:\mathbb{R}^{n+1}\to\mathbb{R}^{n+1}$. In order to prove the existence of such a transformation $B$, let
\begin{equation*}
\langle\langle\cdot,\cdot\rangle\rangle:=\langle p,p\rangle \langle \cdot,\cdot\rangle-\langle p,\cdot\rangle^2+dh_p^2.
\end{equation*}
We claim that $\langle\langle\cdot,\cdot\rangle\rangle>0$. To show this, write $v\in\mathbb{R}^{n+1}\setminus \{0\}$ as $v=ap+w$, $w\in p^{\bot_{\langle\cdot,\cdot\rangle}}$.
Note that $a$ and $w$ are uniquely determined since $\mathbb{R}^{n+1}=\mathbb{R}p\oplus p^{\bot_{\langle\cdot,\cdot\rangle}}$. We obtain
\begin{equation*}
\langle\langle v,v\rangle\rangle=\langle p,p\rangle \langle w,w\rangle + \left(dh_p(ap+w)\right)^2.
\end{equation*}
For $w\ne 0$ we immediately see that $\langle\langle v,v\rangle\rangle>0$. For $w=0$, $v\ne 0$ implies $a\ne 0$. In that case $\langle\langle v,v\rangle\rangle=a^2\tau^2>0$. This shows that $\langle\langle\cdot,\cdot\rangle\rangle$ is indeed positive definite. Now let $B\in\mathrm{GL}(n+1)$ be an orthonormal basis\footnote{We interpret the columns of $B$ as the basis vectors.} of $\langle\langle\cdot,\cdot\rangle\rangle$, that is $B^*\langle\langle\cdot,\cdot\rangle\rangle=\langle \cdot,\cdot\rangle.$
Denote by $\widetilde{h}=h\circ B$ the transformed polynomial $h$ and let $q=B^{-1}p$. Then $d\widetilde{h}_q=dh_{Bq}(B\cdot)=dh_p(B\cdot)$
and 
\begin{align*}
\langle q,\cdot\rangle= \langle \langle Bq,B\cdot\rangle\rangle=\langle \langle p,B\cdot\rangle\rangle =\langle p,p\rangle \langle p,B\cdot\rangle-\langle p,p\rangle\langle p,B\cdot\rangle + dh_p(p)dh_p(B\cdot) =\tau dh_p(B\cdot) =\tau d\widetilde{h}_q.
\end{align*}
Hence, $B$ fulfils \eqref{eqn_Btrafo_algpmoving} with $r=\frac{1}{\tau}$ and we have $d\widetilde{h}_q=\frac{1}{\tau}\langle q,\cdot\rangle$ with $q\in B^{-1}\mathcal{H}$. We are now in the setting of the first case and can proceed as described therein.

Summarising up to this point, we have shown that for any $n\geq1$-dimensional connected GPSR manifold $\mathcal{H}\subset\{h=1\}$ and all $p\in\mathcal{H}$ we can find $A\in\mathrm{GL}(n+1)$, such that the conditions \eqref{std_form_i} and \eqref{std_form_ii} are fulfilled. Now we will describe how to construct $A$ explicitly.

We will start with the case where $\mathcal{H}$ is a CCGPSR manifold, and first construct the transformation $A(p)$ explicitly for one arbitrarily chosen point $p\in \mathcal{H}$, so that $A(p)$ fulfils \eqref{std_form_i} and \eqref{std_form_ii}. We start by choosing initial linear coordinates $(x,y_1,\ldots,y_n)^T$ of $\mathbb{R}^{n+1}$ and a point $p=\left(\begin{smallmatrix}
p_x\\
p_y
\end{smallmatrix}\right)\in\mathcal{H}$. After a possible reordering of the coordinates we can assume that $\frac{\partial h}{\partial x}(p)\ne 0$. This follows from $dh_p\ne 0$, since otherwise $\tau h(p)=dh_p(p)=0$. Let
\begin{equation*}
\widetilde{A}=\left(
\begin{tabular}{c|c}
$p_x$ &  $\left.-\frac{\partial_yh}{\partial_xh}\right|_p$ \\ \hline
$p_y$  & $\overset{}{\mathbbm{1}}$
\end{tabular}
\right)
\end{equation*}
where $\partial_xh:=\frac{\partial h}{\partial x}$ and $\partial_yh:=\sum\limits_{i=1}^n dh(\partial_{y_i})dy_i$. $\widetilde{A}\in \mathrm{GL}(n+1)$ follows from
\begin{align}\label{eqn_det_Atilda}
\det (\widetilde{A}) =\det\left(
\begin{tabular}{c|c}
$p_x+\left.\frac{\partial_yh}{\partial_xh}\right|_p(p_y)$ &  $\left.-\frac{\partial_yh}{\partial_xh}\right|_p$ \\ \hline
$0$  & $\overset{}{\mathbbm{1}}$
\end{tabular}
\right) =\frac{1}{\partial_xh}\left(\partial_x h_p\cdot p_x+\partial_yh_p(p_y)\right) =\frac{\tau}{\partial_x h_p} \ne 0.
\end{align}
In the above formula we have used the Euler identity for homogeneous functions. We find $\widetilde{A}\cdot\left(\begin{smallmatrix}
1\\
0
\end{smallmatrix}\right)=p$ and
\begin{align*}
h\left(\widetilde{A}\cdot\left(
\begin{matrix}
x\\
y
\end{matrix}
\right)
\right)&= x^\tau h(p) +x^{\tau-1} dh_p\left(
\begin{matrix}
-\left.\frac{\partial_yh}{\partial_xh}\right|_p(y)\\
 y
\end{matrix}
\right) +x^{\tau-2}\frac{1}{2}\partial^2 h_p\left(\left(
\begin{matrix}
-\left.\frac{\partial_yh}{\partial_xh}\right|_p(y)\\
 y
\end{matrix}
\right),\left(
\begin{matrix}
-\left.\frac{\partial_yh}{\partial_xh}\right|_p(y)\\
 y
\end{matrix}
\right)\right)\\ 
&\quad+(\text{terms of lower order in }x)\\
&= x^\tau +x^{\tau-2}\frac{1}{2}\partial^2 h_p\left(\left(
\begin{matrix}
-\left.\frac{\partial_yh}{\partial_xh}\right|_p(y)\\
 y
\end{matrix}
\right),\left(
\begin{matrix}
-\left.\frac{\partial_yh}{\partial_xh}\right|_p(y)\\
 y
\end{matrix}
\right)\right) +(\text{terms of lower order in }x).
\end{align*}
The vanishing of the $x^{\tau-1}$-term follows from $dh_p\left(
\begin{matrix}
-\left.\frac{\partial_yh}{\partial_xh}\right|_p(y)\\
y
\end{matrix}
\right)=0$ for all $y\in\mathbb{R}^n$. This is equivalent to $\left(
\begin{matrix}
-\left.\frac{\partial_yh}{\partial_xh}\right|_p(y)\\
 y
\end{matrix}
\right)\in T_p\mathcal{H}$ for all $y\in\mathbb{R}^n$. Hence, 
\begin{equation}\label{p_moving_pos_def_bilform_alternative}
\mathbb{R}^{n}\times\mathbb{R}^{n}\ni(v,w)\mapsto -\frac{1}{2}\partial^2 h_p\left(\left(
\begin{matrix}
-\left.\frac{\partial_yh}{\partial_xh}\right|_p(v)\\
v
\end{matrix}
\right),\left(
\begin{matrix}
-\left.\frac{\partial_yh}{\partial_xh}\right|_p(w)\\
w
\end{matrix}
\right)\right)
\end{equation}
is a positive definite bilinear form since $p$ is, by assumption, a hyperbolic point of $h$. This implies that there exists a linear transformation $\widetilde{E}\in\mathrm{GL}(n)$, such that
\begin{align*}
h\left(\widetilde{A}\cdot
\left(
\begin{tabular}{c|c}
$1$ &    \\ \hline
   & $\overset{}{\widetilde{E}}$
\end{tabular}
\right)\cdot\left(
\begin{matrix}
x\\
y
\end{matrix}
\right)\right)=x^\tau-x^{\tau-2}\langle y,y\rangle+\sum\limits_{k=3}^{\tau}x^{\tau-k}P_k(y).
\end{align*}
Since $\widetilde{A}\cdot
\left(
\begin{tabular}{c|c}
$1$ &    \\ \hline
   & $\overset{}{\widetilde{E}}$
\end{tabular}
\right)\cdot\left(
\begin{matrix}
1\\
0
\end{matrix}
\right)=p$, we have shown that for one choice of $p\in\mathcal{H}$ we can find a linear transformation fulfilling both \eqref{std_form_i} and \eqref{std_form_ii}.

In order to prove the statement of this proposition for all $p\in\mathcal{H}$, we have shown that we can assume without loss of generality that $h$ is of the form \eqref{std_form_i} and that $\left(\begin{smallmatrix}
1\\
0
\end{smallmatrix}
\right)\in\mathcal{H}\subset\{h=1\}$. For $p=\left(\begin{smallmatrix}
p_x\\
p_y
\end{smallmatrix}\right)\in\mathcal{H}$ and $E(p)\in\mathrm{GL}(n)$ consider the matrix
\begin{equation}\label{p_moving_A_matrix}
A(p):=\left(
\begin{tabular}{c|c}
$p_x$ &  $-\left.\frac{\partial_yh}{\partial_xh}\right|_p\circ E(p)$ \\ \hline
$p_y$  & $\overset{}{E(p)}$
\end{tabular}
\right).
\end{equation}
Firstly we need to ensure that $A(p)$ is well-defined for all $p\in \mathcal{H}$ and all choices for $E(p)\in\mathrm{GL}(n)$. This follows from
\begin{equation}
\partial_xh|_{\mathcal{H}}>0,\label{eqn_positivity_delxh}
\end{equation}
which we will prove next. In order to show that \eqref{eqn_positivity_delxh} holds for all $n\geq 1$-dimensional CCGPSR manifolds, it in facts suffices to prove it for all $1$-dimensional CCGPSR manifolds. To see this, suppose that $\mathrm{dim}(\mathcal{H})>1$ and that there exists a point $\overline{p}=\left(\begin{smallmatrix}
\overline{p}_x\\
\overline{p}_y
\end{smallmatrix}\right)\in\mathcal{H}$, such that $\partial_xh|_{\overline{p}}=0$. Then the set
\begin{equation*}
\widetilde{\mathcal{H}}:=\mathcal{H}\cap\mathrm{span}\left\{\left(\begin{smallmatrix}
1\\ 0
\end{smallmatrix}\right),\overline{p}\right\}
\end{equation*}
is a 1-dimensional CCGPSR manifold which coincides with the connected component of the level set
\begin{equation}\label{eqn_2dimdelxhreduction}
\left.\left\{\left(\begin{smallmatrix}
x\\ y
\end{smallmatrix}
\right)\in\mathbb{R}^2\ \right|\ h\left(x\left(\begin{smallmatrix}
1\\ 0
\end{smallmatrix}
\right)+y\left(\begin{smallmatrix}
0\\ v
\end{smallmatrix}
\right)\right)=1\right\}
\end{equation}
that contains the point $\left(\begin{smallmatrix}
1\\ 0
\end{smallmatrix}
\right)\in\mathbb{R}^2$. In \eqref{eqn_2dimdelxhreduction}, $v\in\mathbb{R}^n$ is chosen to fulfil $\mathrm{span}\left\{\left(\begin{smallmatrix}
1\\ 0
\end{smallmatrix}\right),\overline{p}\right\}=\mathrm{span}\left\{\left(\begin{smallmatrix}
1\\ 0
\end{smallmatrix}\right),\left(\begin{smallmatrix}
0\\ v
\end{smallmatrix}\right)\right\}$ and $\langle v,v\rangle=1$. Note that $\widetilde{h}:=h\left(x\left(\begin{smallmatrix}
1\\ 0
\end{smallmatrix}
\right)+y\left(\begin{smallmatrix}
0\\ v
\end{smallmatrix}
\right)\right)$ is then automatically of the form \eqref{std_form_i}. Denote by $\widetilde{p}=\left(\begin{smallmatrix}
\widetilde{p}_x\\ \widetilde{p}_y
\end{smallmatrix}\right)\in\mathbb{R}^2$ the point fulfilling $\widetilde{p}_x\left(\begin{smallmatrix}
1\\ 0
\end{smallmatrix}\right)+\widetilde{p}_y\left(\begin{smallmatrix}
0\\ v
\end{smallmatrix}\right)=\overline{p}$ and note that $\widetilde{p}_y\ne 0$. Then $\widetilde{p}\in\widetilde{\mathcal{H}}$ by construction and $\partial_x\widetilde{h}|_{\widetilde{p}}=0$. It now follows from Lemma \ref{precomp_cor} that there exists $R>0$, such that $\widetilde{h}\left(\widetilde{p}+R\left(\begin{smallmatrix}
1\\ 0
\end{smallmatrix}
\right)\right)=0$, since $\left(\begin{smallmatrix}
1\\ 0
\end{smallmatrix}
\right)\in T_{\widetilde{p}}\widetilde{\mathcal{H}}$ by assumption. The convexity of the cone $\widetilde{U}:=\mathbb{R}_{>0}\cdot\widetilde{\mathcal{H}}\subset\mathbb{R}^2$ (cf. Proposition \ref{prop_convexcone}) implies that
\begin{equation*}
\widetilde{U}\subset\left(\mathbb{R}_{>0}\cdot\left(\begin{smallmatrix}
1\\ 0
\end{smallmatrix}
\right) + \mathbb{R}\cdot \left(\widetilde{p}+R\left(\begin{smallmatrix}
1\\ 0
\end{smallmatrix}
\right)\right)\right)=:\widetilde{V}.
\end{equation*}
But $\widetilde{p}\not\in \widetilde{V}$, and we conclude with $\widetilde{\mathcal{H}}\subset\widetilde{U}$ that $\widetilde{p}\not\in\widetilde{\mathcal{H}}$, which is a contradiction. We have thus shown that \eqref{eqn_positivity_delxh} holds for every $n\geq 1$-dimensional CCGPSR manifold $\mathcal{H}$.

We now show that for all $p\in\mathcal{H}$ and all choices for $E(p)\in\mathrm{GL}(n)$, $A(p)\in \mathrm{GL}(n+1)$. The calculation is similar to calculating $\det(\widetilde{A})$ \eqref{eqn_det_Atilda} and yields
\begin{align*}
\det A(p)&=\frac{\tau}{\partial_x h_p}\det E(p)\ne 0.
\end{align*}
In order to obtain the conditions for $E(p)$ so that $A(p)$ fulfils condition \eqref{std_form_i}, we calculate
\begin{align*}
h\left(A(p)\cdot\left(
\begin{matrix}
x\\
y
\end{matrix}
\right)
\right)&= x^\tau h(p) +x^{\tau-1} dh_p\left(
\begin{matrix}
-\left.\frac{\partial_yh}{\partial_xh}\right|_p(E(p)y)\\
E(p)y
\end{matrix}
\right)\\
&\quad+x^{\tau-2}\frac{1}{2}\partial^2 h_p\left(\left(
\begin{matrix}
-\left.\frac{\partial_yh}{\partial_xh}\right|_p(E(p)y)\\
E(p)y
\end{matrix}
\right),\left(
\begin{matrix}
-\left.\frac{\partial_yh}{\partial_xh}\right|_p(E(p)y)\\
E(p)y
\end{matrix}
\right)\right)\\
&\quad+(\text{terms of lower order in }x).
\end{align*}
By definition, $dh_p\left(
\begin{matrix}
-\left.\frac{\partial_yh}{\partial_xh}\right|_p(E(p)y)\\
E(p)y
\end{matrix}
\right)= 0$ for all $y\in\mathbb{R}^n$ and all $E(p)\in\mathrm{GL}(n)$, which is equivalent to $\left(
\begin{matrix}
-\left.\frac{\partial_yh}{\partial_xh}\right|_p(E(p)y)\\
E(p)y
\end{matrix}
\right)\in T_p\mathcal{H}$ for all $y\in\mathbb{R}^n$ and all choices $E(p)\in\mathrm{GL}(n)$. Thus,
the bilinear form in equation \eqref{p_moving_pos_def_bilform_alternative} is a positive definite bilinear form
since $\mathcal{H}\subset\{h=1\}$ consists only of hyperbolic points of the defining polynomial $h$. We conclude that for all $p\in\mathcal{H}$, $E(p)\in\mathrm{GL}(n)$ can be chosen in such a way that 
\begin{equation}\label{E_condition}
-\frac{1}{2}\partial^2 h_p\left(\left(
\begin{matrix}
-\frac{\partial_y h(E(p)y)}{\partial_x h}\\
E(p)y
\end{matrix}
\right),\left(
\begin{matrix}
-\frac{\partial_y h(E(p)y)}{\partial_x h}\\
E(p)y
\end{matrix}
\right)\right)=\langle y,y\rangle
\end{equation}
for all $y\in\mathbb{R}^n$.

Summarising, we have demonstrated for each $p\in\mathcal{H}$ how to explicitly construct a linear change of coordinates $A(p)\in\mathrm{GL}(n+1)$ which fulfils \eqref{std_form_i} and \eqref{std_form_ii}. It remains to show that the assignment $A:\mathcal{H}\to\mathrm{GL}(n+1)$ can be chosen so that it is a smooth map. To see this observe that
\begin{equation*}
A(p)=\left(
\begin{tabular}{c|c}
$p_x$ &  $-\left.\frac{\partial_yh}{\partial_xh}\right|_p$ \\ \hline
$p_y$  & $\overset{}{\mathbbm{1}}$
\end{tabular}
\right)
\cdot
\left(
\begin{tabular}{c|c}
$1$ &    \\ \hline
   & $\overset{}{E(p)}$
\end{tabular}
\right).
\end{equation*}
The matrix $\left(
\begin{tabular}{c|c}
$p_x$ &  $-\left.\frac{\partial_yh}{\partial_xh}\right|_p$ \\ \hline
$p_y$  & $\overset{}{\mathbbm{1}}$
\end{tabular}
\right)$ in the above equation depends smoothly on $p\in\mathcal{H}$. Hence, it suffices to show that $E:\mathcal{H}\to\mathrm{GL}(n)$ can be chosen so that it is a smooth map and fulfils equation \eqref{E_condition}. This follows from the fact that, as we have seen above, 
\begin{equation*}
-\frac{1}{2}\partial^2 h_p\left(\left(
\begin{matrix}
-\left.\frac{\partial_yh}{\partial_xh}\right|_p(\cdot)\\
\cdot
\end{matrix}
\right),\left(
\begin{matrix}
-\left.\frac{\partial_yh}{\partial_xh}\right|_p(\cdot)\\
\cdot
\end{matrix}
\right)\right):\mathbb{R}^n\times\mathbb{R}^n\to\mathbb{R}
\end{equation*}
understood as in \eqref{p_moving_pos_def_bilform_alternative} is positive definite for all $p\in \mathcal{H}$, cf. \cite[Lem.\,8.13]{Le}.

It remains to deal with the cases where $\mathcal{H}\subset\{h=1\}\subset\mathbb{R}^{n+1}$ is a connected GPSR manifold, but is \underline{not} closed in $\mathbb{R}^{n+1}$. For $p\in\mathcal{H}$ arbitrary and fixed, we want to show that there exists a neighbourhood $V\subset\mathcal{H}$ of $p$ in $\mathcal{H}$, such that $A:V\to\mathrm{GL}(n+1)$ can be chosen to fulfil \eqref{std_form_i} and \eqref{std_form_ii} and to be a smooth map. We have already seen in the beginning of the proof that we can, after a possible linear transformation of the coordinates of $\mathbb{R}^{n+1}$, assume without loss of generality that $p=\left(\begin{smallmatrix}
1\\ 0
\end{smallmatrix}
\right)$, that $h$ is of the form $h=x^\tau-x^{\tau-2}\langle y,y\rangle + \sum\limits_{k=3}^\tau x^{\tau-k}P_k(y)$, and that $\mathcal{H}$ is the contained in the connected component of $\{h=1\}$ that contains the point $\left(\begin{smallmatrix}
1\\ 0
\end{smallmatrix}\right)\in\mathcal{H}$. Since $\partial_x h|_{\left(\begin{smallmatrix}
1\\ 0
\end{smallmatrix}
\right)}=\tau>0$, it immediately follows that we can find a neighbourhood $V$ of $\left(\begin{smallmatrix}
1\\ 0
\end{smallmatrix}
\right)$ in $\mathcal{H}$, such that $\partial_xh|_q>0$ for all $q\in V$. We can now define $A$ as in equation \eqref{p_moving_A_matrix} and proceed as for the case when $\mathcal{H}$ was assumed to be closed.
\end{proof}
\end{Prop}

Proposition \ref{prop_std_form_h} shows in particular that for any CCGPSR manifold $\mathcal{H}\subset\{h=1\}$ we can assume without loss of generality that $h$ is of the form 
\begin{equation}
h=x^\tau-x^{\tau-2}\langle y,y\rangle+\sum\limits_{k=3}^{\tau}x^{\tau-k}P_k(y)\label{standard_form_h}
\end{equation}
and that $\mathcal{H}$ is the precisely the connected component of $\{h=1\}$ which contains the point $\left(\begin{smallmatrix}
x\\
y
\end{smallmatrix}
\right)=\left(\begin{smallmatrix}
1\\
0
\end{smallmatrix}
\right)\in\mathbb{R}^{n+1}$. If $\mathcal{H}$ is just assumed to be a connected an not necessarily closed GPSR manifold, we can still assume without loss of generality that $\mathcal{H}$ is a connected open subset of $\{h=1\}$ with $h$ of the form \eqref{standard_form_h}, and that $\mathcal{H}$ contains the point $\left(\begin{smallmatrix}
1\\
0
\end{smallmatrix}
\right)\in\mathbb{R}^{n+1}$. Also note that whenever $\mathcal{H}$ is a CCGPSR manifold, the point $\left(\begin{smallmatrix}
1\\
0
\end{smallmatrix}
\right)\in\mathcal{H}$ is the unique point in $\mathcal{H}$ that minimises the Euclidean distance of $\mathcal{H}\subset\mathbb{R}^{n+1}$ and the origin $0\in\mathbb{R}^{n+1}$ (in the chosen linear coordinates $\left(\begin{smallmatrix}
x\\
y
\end{smallmatrix}
\right)$ of $\mathbb{R}^{n+1}$).

Further observe that for connected PSR manifolds the term $P_3$ is never uniquely determined. To see this consider $\mathcal{H}\subset\{h=1\}$ with $h$ of the form \eqref{standard_form_h}. If $P_3\ne 0$, the transformation $y\mapsto -y$ will preserve the form \eqref{standard_form_h} and send $P_3\mapsto -P_3$. For $P_3=0$, one can verify that for any point $p\in\mathcal{H}$, $p\ne \left(\begin{smallmatrix}1\\ 0\end{smallmatrix}\right)$, the corresponding coordinate transformation $A(p)$ of the form \eqref{p_moving_A_matrix} will induce a non-zero $P_3$-part in the transformed polynomial $h$.

Proposition \ref{prop_std_form_h} is also useful to unclutter the rest of our studies by introducing the term \textit{standard form} of GPSR manifolds as follows:
\begin{notation}
The statement that a GPSR manifold $\mathcal{H}$ is in \textsf{standard form} will in the following mean that
	\begin{itemize}
		\item $\mathcal{H}\subset\{h=1\}\subset\mathbb{R}^{n+1}$ is an $n\geq 1$-dimensional GPSR manifold of degree $\tau\geq 3$,
		\item we have chosen linear coordinates $\left(\begin{smallmatrix}x\\ y\end{smallmatrix}\right)=(x,y_1,\ldots,y_n)^T$ on the ambient space $\mathbb{R}^{n+1}$ of $\mathcal{H}$, such that $h$ is of the form \eqref{standard_form_h} and $\left(\begin{smallmatrix}x\\ y\end{smallmatrix}\right)=\left(\begin{smallmatrix}1\\ 0\end{smallmatrix}\right)\in\mathcal{H}$.
	\end{itemize}
By Proposition \ref{prop_std_form_h} we know that assuming that a GPSR manifold is in standard form is not a restriction of generality. We might further specify the degree $\tau$ of $\mathcal{H}$ or impose topological properties such as maximality or the dimension, which we will then denote by e.g. ``let $\mathcal{H}$ be a maximal PSR manifold of dimension $n\geq 3$ in standard form''. Using the abbreviation ``standard form'' will thus allow us to omit stating every time that the defining polynomial of a GPSR manifold $\mathcal{H}$ is assumed to be of the form \eqref{standard_form_h}, the dimension of $\mathcal{H}$ has to be $\geq 1$, and that for the linear coordinates $\left(\begin{smallmatrix}x\\ y\end{smallmatrix}\right)$ of the ambient space the point $\left(\begin{smallmatrix}1\\ 0\end{smallmatrix}\right)$ is assumed to be an element of $\mathcal{H}$. This will make the following statements considerably easier to read. Occasional, however, we might write out additional information that is already implied by the term ``standard form'' in order to make specific statements easier to understand, see e.g. Corollary \ref{cor_regularity_boundary_generating_set_CCPSR}.
\end{notation}

Next, we will calculate standard forms of CCPSR surfaces, cf. Theorem \ref{lowdimpsrclassTHM}. Aside from serving as examples for the techniques developed in Proposition \ref{prop_std_form_h}, the following calculations will be important in proving the later Theorem \ref{thm_convex_compact_PSR_generating_set} which is one of the main ingredients we need to prove our main Theorem \ref{thm_Cn}.
\begin{Ex}
\label{example_PSRsurfaces_standard_form}
Let $(x,y,z)^T$ denote the standard linear coordinates on $\mathbb{R}^3$. Recall that CCPSR surfaces $\mathcal{H}\subset\{h=1\}\subset\mathbb{R}^3$ have been classified up to equivalence in \cite[Thm.\,1]{CDL}, cf. Theorem \ref{lowdimpsrclassTHM} a)--f). In the following we will for each $h$ corresponding to the cases a)--f) give a choice of $A=A(p)\in\mathrm{GL}(3)$ corresponding to a given point $p\in\mathcal{H}$, such that $A\cdot\left(\begin{smallmatrix}
1\\ 0 \\ 0
\end{smallmatrix}\right)=p$, $h\left(A\cdot\left(\begin{smallmatrix}
x\\ y\\ z
\end{smallmatrix}
\right)\right)$ is of the form \eqref{standard_form_h}, and $A^{-1}(\mathcal{H})\subset\left\{h\circ A=1\right\}$ is precisely the connected component of $\left\{h\circ A=1\right\}\subset\mathbb{R}^3$ that contains the point $\left(\begin{smallmatrix}
x\\ y\\ z
\end{smallmatrix}
\right)=\left(\begin{smallmatrix}
1\\ 0\\ 0
\end{smallmatrix}
\right)$.
\paragraph*{a) $\mathcal{H}=\{h=xyz=1,\ x>0,\ y>0\}$.}\hspace*{0mm}\\
It is clear that $p=(1,1,1)^T\in\mathcal{H}$. One choice for the corresponding linear transformation of the form \eqref{p_moving_A_matrix} is
\begin{equation*}
A(p)=\left(
\begin{smallmatrix}
1 & -\frac{2}{\sqrt{3}} & 0\\
1 & \frac{1}{\sqrt{3}} & -1\\
1 & \frac{1}{\sqrt{3}} & 1
\end{smallmatrix}
\right),
\end{equation*}
which brings $h$ to the form
\begin{equation}\label{eqn_CDL_a_equiv_h_form}
h\left(A(p)\cdot\left(\begin{smallmatrix}
x\\
y\\
z
\end{smallmatrix}
\right)\right)=x^3-x(y^2+z^2)-\frac{2}{3\sqrt{3}}y^3+\frac{2}{\sqrt{3}}yz^2.
\end{equation}

\paragraph*{b) $\mathcal{H}=\{h=x(xy-z^2)=1,\ x>0\}$.}\hspace*{0mm}\\
Similar to the surface in \textbf{a)}, consider the point $p=(1,1,0)^T\in\mathcal{H}$ and
\begin{equation*}
A(p)=\left(\begin{smallmatrix}
1 & -\frac{1}{\sqrt{3}} & 0\\
1 & \frac{2}{\sqrt{3}} & 0\\
0 & 0 & 1
\end{smallmatrix}
\right).
\end{equation*}
Then 
\begin{equation}\label{eqn_CDL_b_equiv_h_form}
h\left(A(p)\cdot\left(\begin{smallmatrix}
x\\
y\\
z
\end{smallmatrix}
\right)\right)=x^3-x(y^2+z^2)+\frac{2}{3\sqrt{3}}y^3+\frac{1}{\sqrt{3}}yz^2.
\end{equation}

\paragraph*{c) $\mathcal{H}=\{h=x(yz+x^2)=1,\ x<0,\ y>0\}$.}\hspace*{0mm}\\
With $p=(-1,2,-2)^T\in\mathcal{H}$ and 
\begin{equation*}
A(p)=\left(
\begin{smallmatrix}
-1 & 0 & \frac{2\sqrt{2}}{\sqrt{15}}\\
1 & \frac{1}{\sqrt{2}} & -\frac{1}{\sqrt{30}}\\
-2 & \sqrt{2} & \frac{\sqrt{2}}{\sqrt{15}}
\end{smallmatrix}
\right)
\end{equation*}
we obtain
\begin{equation}\label{eqn_CDL_c_equiv_h_form}
h\left(A(p)\cdot\left(\begin{smallmatrix}
x\\
y\\
z
\end{smallmatrix}
\right)\right)=x^3-x(y^2+z^2)+\frac{2\sqrt{2}}{\sqrt{15}}y^2z+\frac{14\sqrt{2}}{15\sqrt{15}}z^3.
\end{equation}

\paragraph*{d) $\mathcal{H}=\{h=z(x^2+y^2-z^2)=1,\ z<0\}$.}\hspace*{0mm}\\
By re-ordering of the coordinates and switching one sign one quickly finds that $\mathcal{H}$ is equivalent to $\widetilde{\mathcal{H}}=\{\widetilde{h}=x^3-x(y^2+z^2)=1,\ x>0\}$, which is precisely the connected component of $\{\widetilde{h}=1\}$ that contains the point $(x,y,z)^T=(1,0,0)^T$. The corresponding point in $\mathcal{H}$ and transformation $A\in\mathrm{GL}(3)$ are given by $p=(0,0,1)^T\in\mathcal{H}$ and
\begin{equation*}
A=\left(\begin{smallmatrix}
0 & 0 & -1 \\ 0 & 1 & 0 \\ 1 & 0 & 0
\end{smallmatrix}
\right),
\end{equation*}
so that indeed
\begin{equation}\label{eqn_CDL_d_equiv_h_form}
h\left(A\cdot\left(\begin{smallmatrix}
x\\
y\\
z
\end{smallmatrix}
\right)\right)=x^3-x(y^2+z^2).
\end{equation}
The transformation $A$ is not of the form \eqref{p_moving_A_matrix} since we needed to switch the $x$- and $z$-coordinate so that $\partial_x(h\circ A)|_{p}\ne 0$.

\paragraph*{e) $\mathcal{H}=\{h=x(y^2-z^2)+y^3=1,\ y<0,\ x>0\}$.}\hspace*{0mm}\\
Consider the point $p=(2,-1,0)^T\in\mathcal{H}$ and the corresponding linear transformation as in \eqref{p_moving_A_matrix}
\begin{equation*}
A(p)=\left(
\begin{smallmatrix}
2 & \frac{1}{\sqrt{3}} & 0\\
-1 & \frac{1}{\sqrt{3}} & -\frac{1}{\sqrt{30}}\\
0 & 0 & \frac{1}{\sqrt{2}}
\end{smallmatrix}
\right).
\end{equation*}
We find
\begin{equation}\label{eqn_CDL_e_equiv_h_form}
h\left(A(p)\cdot\left(\begin{smallmatrix}
x\\
y\\
z
\end{smallmatrix}
\right)\right)=x^3-x(y^2+z^2)+\frac{2}{3\sqrt{3}}y^3-\frac{1}{2\sqrt{3}}yz^2.
\end{equation}

\paragraph*{f) $\mathcal{H}_b=\{h=y^2z-4x^3+3xz^2+bz^3=1,\ z<0,\ 2x>z\}$, $b\in(-1,1)$.}\hspace*{0mm}\\
Observe that the point $p_b=\frac{1}{\sqrt[3]{1-b}}(1/2,0,-1)^T$ is contained in $\mathcal{H}_b$ for all $b\in(-1,1)$. After switching the $x$- and $z$-coordinate via the transformation $\left(\begin{smallmatrix}
 & & 1\\ & 1 & \\ 1 & &
\end{smallmatrix}\right)$, we can apply the construction in equation \eqref{p_moving_A_matrix} in order to find $\widetilde{A}_b\in\mathrm{GL}(3)$, such that $h\circ\left(\left(\begin{smallmatrix}
 & & 1\\ & 1 & \\ 1 & &
\end{smallmatrix}\right)\cdot \widetilde{A}_b\right)$ is of the form \eqref{standard_form_h}. We find
\begin{equation*}
\widetilde{A}_b=\left(\begin{matrix}
-\frac{1}{\sqrt[3]{1-b}} & 0 & 0\\ 0 & \sqrt[6]{1-b} & 0\\ \frac{1}{2\sqrt[3]{1-b}} & 0 & -\frac{\sqrt[6]{1-b}}{\sqrt{6}}
\end{matrix}\right).
\end{equation*}
With $A_b:=\left(\begin{smallmatrix}
 & & 1\\ & 1 & \\ 1 & &
\end{smallmatrix}\right)\cdot \widetilde{A}_b$
we obtain
\begin{equation}\label{eqn_CDL_f_equiv_h_form}
h\left(A_b\cdot \left(\begin{smallmatrix}
x\\ y\\ z
\end{smallmatrix}
\right)\right)=x^3-x(y^2+z^2)+\frac{\sqrt{2}\sqrt{1-b}}{3\sqrt{3}}z^3
\end{equation}
and have thus shown that $h\circ A_b$ is of the form \eqref{standard_form_h} and that $A_b\cdot(1,0,0)^T=p_b$ for all $b\in(-1,1)$ as required. Note that equation \eqref{eqn_CDL_f_equiv_h_form} allows us to interpret the one-parameter family of CCPSR surfaces $\mathcal{H}_b$ as an interpolation between the CCPSR surfaces Theorem \ref{lowdimpsrclassTHM} d) (for $b\to 1$, see equation \eqref{eqn_CDL_d_equiv_h_form}) and Theorem \ref{lowdimpsrclassTHM} e) (for $b\to -1$). To see the latter, observe that with
\begin{equation*}
\widetilde{A}=\left(
\begin{matrix}
2^{-\frac{1}{3}}\frac{4}{3} & 2^{-\frac{1}{3}} \frac{2}{3\sqrt{3}} & 0\\
2^{-\frac{1}{3}}\frac{1}{\sqrt{3}} & 2^{-\frac{1}{3}}\frac{5}{3} & 0\\
0 & 0 & 2^{-\frac{5}{6}}\sqrt{3}
\end{matrix}
\right),
\end{equation*}
the polynomial
	\begin{equation}\label{eqn_b_alt_limit}
		\widetilde{h}:=x^3-x(y^2+z^2)+\frac{2}{3\sqrt{3}}y^3,
	\end{equation}
which is precisely the limit polynomial of \eqref{eqn_CDL_f_equiv_h_form} for $b\to-1$ after to swapping $y$ and $z$, transforms to
\begin{equation*}
\widetilde{h}\left(\widetilde{A}\cdot\left(\begin{smallmatrix}
x\\ y\\ z
\end{smallmatrix}\right)\right)=x^3-x(y^2+z^2)+\frac{2}{3\sqrt{3}}y^3-\frac{1}{2\sqrt{3}}yz^2
\end{equation*}
which coincides with equation \eqref{eqn_CDL_e_equiv_h_form}. Furthermore one can check that the point $\widetilde{A}\cdot(1,0,0)^T$ is contained in the connected component of $\left\{x^3-x(y^2+z^2)+\frac{2}{3\sqrt{3}}y^3-\frac{1}{2\sqrt{3}}yz^2=1\right\}$ that contains the point $(x,y,z)^T=(1,0,0)^T$, for which we have shown that this is equivalent to the CCPSR surface e) in Theorem \ref{lowdimpsrclassTHM}. Hence, the connected component of $\left\{x^3-x(y^2+z^2)+\frac{2}{3\sqrt{3}}y^3=1\right\}$
that contains the point $(x,y,z)^T=(1,0,0)^T$ is in particular also a CCPSR surface which is equivalent to the surface e).\footnote{In order to find the transformation $\widetilde{A}$, we have used a technique developed later in Theorem \ref{thm_convex_compact_PSR_generating_set}. Specifically we used equations \eqref{eqn_A(T)} and \eqref{eqn_h_2dim_E_id_trafo}.}
\end{Ex}

For the following considerations it is helpful to consider a certain parametrisation of connected GPSR manifolds which we will introduce next.
\begin{Def}
\label{def_dom(H)}
Let $\mathcal{H}$ be a connected GPSR manifold in standard form.
We define
\begin{equation*}
\mathrm{dom}(\mathcal{H}):=\mathrm{pr}_{\mathbb{R}^n}\left((\mathbb{R}_{>0}\cdot\mathcal{H})\cap\left.\left\{
\left(
\begin{smallmatrix}
1\\
y
\end{smallmatrix}
\right)\in\mathbb{R}^{n+1}
\
\right|\ y\in\mathbb{R}^n
\right\}\right)\subset\mathbb{R}^n
\end{equation*}
where $\mathrm{pr}_{\mathbb{R}^n}:\mathbb{R}^{n+1}\to\mathbb{R}^n$, $\left(\begin{smallmatrix}
x\\
y
\end{smallmatrix}
\right)\mapsto y$.
\end{Def}
The set $\mathrm{dom}(\mathcal{H})$ is precisely the intersection of the cone spanned by $\mathcal{H}$, that is $\mathbb{R}_{>0}\cdot\mathcal{H}\subset\mathbb{R}^{n+1}$, and $T_{\left(\begin{smallmatrix}
1\\
0
\end{smallmatrix}\right)}\mathcal{H}$ affinely embedded in $\mathbb{R}^{n+1}$ via $v\mapsto \left(\begin{smallmatrix}
1\\ 0
\end{smallmatrix}  
\right) + \left(\begin{smallmatrix}
0\\ v
\end{smallmatrix}
\right)$.
Independent of whether the connected GPSR manifold $\mathcal{H}\subset\{h=1\}\subset\mathbb{R}^{n+1}$ is closed or not, $\mathrm{dom}(\mathcal{H})\subset\mathbb{R}^n$ is well-defined, open in $\mathbb{R}^n$, and always contains an open ball $B_\varepsilon(0)\subset\mathbb{R}^n$ with respect to the standard scalar product $\langle \cdot,\cdot\rangle$ on $\mathbb{R}^n$ for $\varepsilon>0$ small enough. In order to check that these claims are true, one uses the following facts. Firstly, every ray $\mathbb{R}_{>0}\cdot p$ for $p\in\mathcal{H}$ meets $\mathcal{H}$ precisely once. This follows from the homogeneity of degree $\tau\geq3$ of the corresponding polynomial $h:\mathbb{R}^{n+1}\to\mathbb{R}$. Secondly, $\mathcal{H}\subset\{x\geq 1\}\subset\mathbb{R}^{n+1}$ and $\mathcal{H}\cap\{x=1\}=\left(\begin{smallmatrix}
1\\ 0
\end{smallmatrix}\right)$. This follows from the fact that $\mathcal{H}$ is locally around each point in $\mathcal{H}$ contained in the boundary of a strictly convex domain of in $\mathbb{R}^{n+1}$, which in turn follows from the Sacksteder-van Heijenoort Theorem\footnote{To apply said theorem, one first needs to extend the considered local neighbourhood of $\mathcal{H}$ to a Euclidean complete convex hypersurface.
} \cite{Wu}. Note that if $\mathcal{H}$ is a CCGPSR manifold, then $\mathcal{H}$ is (globally) the boundary of the strictly convex domain $\mathbb{R}_{>1}\cdot\mathcal{H}\subset\mathbb{R}^{n+1}$.
Thus, every ray $\mathbb{R}_{>0}\cdot p$ for $p\in\mathcal{H}$ has a unique intersection-point with the set $\mathrm{dom}(\mathcal{H})$. We see that $\mathrm{dom}(\mathcal{H})$ is
bijective to $\mathcal{H}$ via 
\begin{equation}\label{dom(H)_H_iso}
\Phi:\mathrm{dom}(\mathcal{H})\to \mathcal{H},\quad \Phi(z)= \frac{1}{\sqrt[\tau]{h\left(\left(\begin{smallmatrix}
1\\ z
\end{smallmatrix}
\right)\right)}}\left(\begin{matrix}
1\\ z
\end{matrix}\right).
\end{equation}
One can check that $\Phi$ is everywhere a local diffeomorphism. This and $\mathcal{H}$ being a hypersurface of $\mathbb{R}^{n+1}$ also show that $\mathrm{dom}(\mathcal{H})\subset\mathbb{R}^n$ is open and, hence, that $\Phi$ is a diffeomorphism\footnote{This is the reason behind our choice of the term ``$\mathrm{dom}$'', which is to be thought of as an abbreviation of ``domain''.}.
Note, however, that the set $\mathrm{dom}(\mathcal{H})$ does depend on the chosen linear coordinates of the ambient space $\mathbb{R}^{n+1}$.

Lemma \ref{precomp_cor} implies the following property of $\mathrm{dom}(\mathcal{H})$ if $\mathcal{H}$ is a CCGPSR manifold.
\begin{Cor}
\label{cor_propdomH_CCGPSR}
Let $\mathcal{H}$ be a CCGPSR manifold in standard form. Then $\mathrm{dom}(\mathcal{H})\subset\mathbb{R}^n$ is open, precompact, and convex.
\end{Cor}
The statement of Corollary \ref{cor_propdomH_CCGPSR} is in particular independent of the linear coordinates of the ambient space $\mathbb{R}^{n+1}$ of $\mathcal{H}$.

We will use the parametrisation \eqref{dom(H)_H_iso} of $\mathcal{H}\subset\{h=1\}$ to study infinitesimal changes of the $P_k$'s in the standard form \eqref{standard_form_h} of $h$ when we vary the point $p\in\mathcal{H}$ in Prop. \ref{prop_std_form_h} \eqref{std_form_i} near $\left(\begin{smallmatrix}
x\\ y
\end{smallmatrix}\right)=\left(\begin{smallmatrix}
1\\
0
\end{smallmatrix}\right)\in\mathcal{H}$. The results are important technical tools for our following studies.
Whenever we use $z$-variables from here on, we will be working with $\mathrm{dom}(\mathcal{H})$. The $y$-variables will be used in when working with the ambient space $\mathbb{R}^{n+1}$ of $\mathcal{H}$.

For the following calculations we will define the (globally smooth) functions
\begin{align}
& \alpha:\mathbb{R}^n\to \mathbb{R},\quad \alpha(z)=\left.\partial_xh\right|_{\left(\begin{smallmatrix}
1\\ z
\end{smallmatrix}\right)},\label{alpha_def}\\
& \beta:\mathbb{R}^n\to \mathbb{R},\quad \beta(z)=h\left(\left(\begin{smallmatrix}
1\\ z
\end{smallmatrix}\right)\right).\label{beta_def}
\end{align}
Note that whenever $\mathcal{H}$ is closed and connected, $\mathrm{dom}(\mathcal{H})$ coincides with the connected component of $\{\beta(z)>0\}$ that contains the point $z=0\in\mathbb{R}^n$, and $\beta|_{\partial\mathrm{dom}(\mathcal{H})}\equiv 0$. Also, as shown in the proof of Proposition \ref{prop_std_form_h}, $\alpha|_{\mathrm{dom}(\mathcal{H})}>0$ if $\mathcal{H}$ is a CCGPSR manifold. If $\mathcal{H}$ is not closed, we can at least find a neighbourhood $V$ of $z=0\in\mathbb{R}^n$, such that $\alpha|_V>0$, which also follows from the proof of Proposition \ref{prop_std_form_h}. Furthermore, it immediately follows from \eqref{dom(H)_H_iso} that $\Phi(z)=\frac{1}{\sqrt[\tau]{\beta(z)}}\left(\begin{smallmatrix}
1\\
z
\end{smallmatrix}
\right)$ for $z\in\mathrm{dom}(\mathcal{H})$. While $dh$ does not vanish on $\mathcal{H}$, it might vanish at a point $\left(\begin{smallmatrix}
x\\ y
\end{smallmatrix}\right)=\left(\begin{smallmatrix}
1\\ \overline{z}
\end{smallmatrix}\right)$ for $\overline{z}\in\partial\mathrm{dom}(\mathcal{H})$ or, equivalently, on the ray $\mathbb{R}_{>0}\cdot\left(\begin{smallmatrix}
1\\
\overline{z}
\end{smallmatrix}\right)\subset \partial(\mathbb{R}_{>0}\cdot\mathcal{H})$. If $\mathcal{H}$ is furthermore closed, we are in this case precisely in the setting of CCGPSR manifolds that are singular at infinity, cf. Definition \ref{def_singular_GPSR}. The following lemma characterises these cases for CCGPSR manifolds in terms of the functions $\alpha$ and $\beta$.
\begin{Lem}\label{lemma_boundary_conditions_alpha_beta}
Let $\mathcal{H}$ be a CCGPSR manifold in standard form and
let $\alpha$, $\beta$ be defined as in \eqref{alpha_def}, respectively \eqref{beta_def}. Then for all $\overline{z}\in\partial\mathrm{dom}(\mathcal{H})$ the following are equivalent:
\begin{enumerate}[(i)]
\item $dh_{\left(\begin{smallmatrix}
1\\
\overline{z}
\end{smallmatrix}\right)}=0$,
\item $\alpha(\overline{z})=0$,
\item $d\beta_{\overline{z}}=0$.
\end{enumerate}
\begin{proof}
Assume that $dh_{\left(\begin{smallmatrix}
1\\
\overline{z}
\end{smallmatrix}
\right)}=0$ for a $\overline{z}\in\partial\mathrm{dom}(\mathcal{H})$. By affinely embedding $\mathrm{dom}(\mathcal{H})$ into $\mathbb{R}^{n+1}$ via $z\mapsto \left(\begin{smallmatrix}
1\\
z
\end{smallmatrix}
\right)$ and identifying $y$ and $z$ we obtain
\begin{equation*}
dh_{\left(\begin{smallmatrix}
1\\
\overline{z}
\end{smallmatrix}
\right)}=\alpha(\overline{z})dx+d\beta_{\overline{z}}.
\end{equation*}
Since $\alpha(\overline{z})dx$ and $d\beta_{\overline{z}}$ are linearly independent we conclude that $\alpha(\overline{z})=0$ and $d\beta_{\overline{z}}=0$.

Now assume that $d\beta_{\overline{z}}=0$. Then, using the Euler-identity for homogeneous functions, $0=\tau\beta(\overline{z})=dh_{\left(\begin{smallmatrix}
1\\
\overline{z}
\end{smallmatrix}
\right)}\left(\begin{smallmatrix}
1\\
\overline{z}
\end{smallmatrix}
\right)=\alpha(\overline{z})$
showing that $\alpha(\overline{z})=0$. Hence, $dh_{\left(\begin{smallmatrix}
1\\
\overline{z}
\end{smallmatrix}
\right)}=0$. 

Lastly, assume that $\alpha(\overline{z})=0$. Similar to above, $0=\tau\beta(\overline{z})=dh_{\left(\begin{smallmatrix}
1\\
\overline{z}
\end{smallmatrix}
\right)}\left(\begin{smallmatrix}
1\\
\overline{z}
\end{smallmatrix}
\right)=d\beta_{\overline{z}}(\overline{z})$. We need to show that this implies $d\beta_{\overline{z}}=0$. Assume the latter does not hold. Then $dh_{\left(\begin{smallmatrix}
1\\
\overline{z}
\end{smallmatrix}
\right)}\ne 0$ and, hence, we can use the implicit function theorem and conclude that $\mathrm{dom}(\mathcal{H})$ has smooth boundary near $\overline{z}$, and $d\beta_{\overline{z}}(\overline{z})=0$ is equivalent to the statement that $\overline{z}\in T_{\overline{z}}\partial\mathrm{dom}(\mathcal{H})$. This, however, contradicts the assumption that $\mathcal{H}$ is a CCGPSR manifold which implies that $\mathrm{dom}(\mathcal{H})$ is a convex set containing the point $0\in\mathbb{R}^n$ (cf. Lemma \ref{precomp_cor}). To see the contradiction, observe that for each non-singular point $\overline{z}\in\partial\mathrm{dom}(\mathcal{H})$, i.e. a point satisfying $d\beta_{\overline{z}}\ne 0$, the affinely embedded tangent space $\overline{z}+T_{\overline{z}}\partial\mathrm{dom}(\mathcal{H})$ in $\mathbb{R}^n$ intersects the convex compact set $\overline{\mathrm{dom}(\mathcal{H})}$ (cf. Corollary \ref{cor_propdomH_CCGPSR}) only at its boundary, that is $\partial\mathrm{dom}(\mathcal{H})$. But if $\overline{z}\in T_{\overline{z}}\partial\mathrm{dom}(\mathcal{H})$, the intersection of $\overline{z}+T_{\overline{z}}\partial\mathrm{dom}(\mathcal{H})$ and $\overline{\mathrm{dom}(\mathcal{H})}$ will always contain $0\in\mathbb{R}^n$ which is, independently of any coordinate choice of the ambient space $\mathbb{R}^{n+1}$ of $\mathcal{H}$, always contained in $\mathrm{dom}(\mathcal{H})$ and in particular never contained in $\partial\mathrm{dom}(\mathcal{H})$. This follows directly from the definition of $\mathrm{dom}(\mathcal{H})$, see Definition \ref{def_dom(H)}. This is a contradiction to the convexity of $\mathrm{dom}(\mathcal{H})$, see Corollary \ref{cor_propdomH_CCGPSR}.
\end{proof}
\end{Lem}

Using Proposition \ref{prop_std_form_h}, we will now study the infinitesimal changes in the transformations $A(p)$ for $p\in\mathcal{H}$ near $\left(\begin{smallmatrix}
x\\ y
\end{smallmatrix}\right)=\left(\begin{smallmatrix}
1\\ 0
\end{smallmatrix}\right)\in\mathcal{H}$, and in the corresponding polynomials $P_i$ in the considered polynomial $h$ as in equation \eqref{standard_form_h}. To do so we use the parametrisation $\Phi:\mathrm{dom}(\mathcal{H})\to\mathcal{H}$ given in equation \eqref{dom(H)_H_iso}.
With this in mind the next result is an infinitesimal analogue of Proposition \ref{prop_std_form_h} and has applications in e.g. significantly simplifying the calculation of the first derivative of the scalar curvature of CCPSR manifolds, cf. Proposition \ref{prop_hom_spaces_deltaPk_condition}.

\begin{Prop}
\label{inf_p-moving_prop}
Let $\mathcal{H}$ be a connected GPSR manifold in standard form and
let $V\subset\mathcal{H}$ be an open neighbourhood of $\left(\begin{smallmatrix}
x\\ y
\end{smallmatrix}\right)=\left(\begin{smallmatrix}
1\\ 0
\end{smallmatrix}\right)$ and $A:V\to\mathrm{GL}(n+1)$,
\begin{equation*}
A(p)=\left(
\begin{tabular}{c|c}
$p_x$ &  $-\left.\frac{\partial_yh}{\partial_xh}\right|_p\circ E(p)$ \\ \hline
$p_y$  & $\overset{}{E(p)}$
\end{tabular}
\right),
\end{equation*}
as in equation \eqref{p_moving_A_matrix} so that $A(p)$ fulfils \eqref{std_form_i} and \eqref{std_form_ii} in Proposition \ref{prop_std_form_h} with $A\left(\left(\begin{smallmatrix}1\\ 0\end{smallmatrix}\right)\right)=\mathbbm{1}$. Let $\Phi:\mathrm{dom}(\mathcal{H})\to\mathcal{H}$ be the diffeomorphism given in equation \eqref{dom(H)_H_iso} and define
\begin{equation}\label{A_matrix_on_dom(H)}
\mathcal{A}:\Phi^{-1}(V)\to\mathrm{GL}(n+1),\quad \mathcal{A}(z):=A(\Phi(z)).
\end{equation}
Then there exists an $\mathfrak{so}(n)$-valued linear map $dB_0\in\mathrm{Lin}(\mathbb{R}^n,\mathfrak{so}(n))$ of the form
\begin{equation*}
dB_0=\sum\limits_{k=1}^{n(n-1)/2} a_k\otimes\langle \ell_k,dz\rangle,
\end{equation*}
where $\{a_k\ |\ 1\leq k\leq n(n-1)/2\}$ is a basis of $\mathfrak{so}(n)$ and $\ell_k\in\mathbb{R}^n$ for all $1\leq k\leq n(n-1)/2$,
such that for $\tau\geq 4$
\begin{align}
&\left.\partial_z \left(h\left(\mathcal{A}(z)\cdot\left(\begin{smallmatrix}
x\\
y
\end{smallmatrix}
\right)\right)\right)\right|_{z=0}=dh_{\left(\begin{smallmatrix}
x\\
y
\end{smallmatrix}
\right)}\left(d\mathcal{A}_0\cdot\left(\begin{smallmatrix}
x\\
y
\end{smallmatrix}
\right)\right)\notag\\
&=x^{\tau-3}\left(\frac{-2(\tau-2)}{\tau}\langle y,y\rangle \langle y,dz\rangle + 3P_3\left(y,y,dB_0 y+\frac{3}{2}P_3(y,\cdot,dz)^T\right)+4P_4(y,y,y,dz)
\right)\notag\\
&\quad+\left(
\sum\limits_{i=4}^{\tau-1} x^{\tau-i}
\left(
\frac{2(\tau-i+1)}{\tau}P_{i-1}(y)\langle y,dz\rangle\right.\right.+iP_i\left(y,\ldots,y,dB_0y+\frac{3}{2}P_3(y,\cdot,dz)^T\right)\notag\\
&\quad\quad\left.\left.+\vphantom{\sum\limits_{i=4}^{\tau-1} x^{\tau-i}}(i+1)P_{i+1}(y,\ldots,y,dz)
\right)
\right)+\frac{2}{\tau}P_{\tau-1}(y)\langle y,dz\rangle+\tau P_\tau\left(y,\ldots,y,dB_0y+\frac{3}{2}P_3(y,\cdot,dz)^T\right)\label{eqn_deltaPkdeterminer}
\end{align}
and for $\tau=3$, that is for connected PSR manifolds,
\begin{align*}
\left.\partial_z \left(h\left(\mathcal{A}(z)\cdot\left(\begin{smallmatrix}
x\\
y
\end{smallmatrix}
\right)\right)\right)\right|_{z=0}&=dh_{\left(\begin{smallmatrix}
x\\
y
\end{smallmatrix}
\right)}\left(d\mathcal{A}_0\cdot\left(\begin{smallmatrix}
x\\
y
\end{smallmatrix}
\right)\right)=-\frac{2}{3}\langle y,y\rangle\langle y,dz\rangle + 3P_3\left(y,y,dB_0y+\frac{3}{2}P_3(y,\cdot,dz)^T\right).
\end{align*}
In the above equations, $P_3(y,\cdot,dz)^T$ is to be understood as the column-vector $\frac{1}{6}\partial^2P_3|_y\left(\left(\begin{smallmatrix}
dz_1\\
\vdots\\
dz_n
\end{smallmatrix}
\right),\cdot\right)^T$
and $dB_0$ is to be understood as
\begin{equation}\label{dB0_eqn}
dB_0y=\sum\limits_{k=1}^{n(n-1)/2} (a_k\cdot y)\otimes\langle \ell_k,dz\rangle.
\end{equation}
\begin{proof}
Note that $z=0\in\Phi^{-1}(V)$ for all possible choices for $V$ since $\Phi^{-1}\left(\left(\begin{smallmatrix}
1\\ 0
\end{smallmatrix}\right)\right)=0$. Observe that for all $v\in\mathbb{R}^n$, the function $-\frac{\partial_yh(v)}{\partial_x h}$ defined on $\mathbb{R}_{>0}\cdot\mathcal{H}$ is constant along rays of the form $\mathbb{R}_{>0}\cdot p$, $p\in\mathcal{H}$. With the notation $\mathcal{E}(z)=E(\Phi(z))$ and $\alpha$, $\beta$ defined in \eqref{alpha_def}, respectively \eqref{beta_def},
\begin{equation*}
\mathcal{A}(z)=\left(
\begin{tabular}{c|c}
$\frac{1}{\sqrt[\tau]{\beta(z)}}$ &  $-\frac{d\beta_z(\mathcal{E}(z)\cdot)}{\alpha(z)}$ \\ \hline
$\frac{z}{\sqrt[\tau]{\beta(z)}}$  & $\overset{}{\mathcal{E}(z)}$
\end{tabular}
\right)
\end{equation*}
and $\mathcal{A}(0)=\mathbbm{1}$, $\mathcal{E}(0)=\mathbbm{1}$, $\alpha(0)=\tau$, $\partial^2\beta_0=-2\langle dz,dz\rangle$.
We obtain
\begin{equation}\label{eqn_dA0_on_domH}
d\mathcal{A}_0=\left(
\begin{tabular}{c|c}
$\underset{}{0}$ &  $\frac{2}{\tau}dz^T$ \\ \hline
$dz$  & $\overset{}{d\mathcal{E}_0}$
\end{tabular}
\right),
\end{equation}
where we understand $dz$ as the identity-map on $\mathbb{R}^n$ and $d\mathcal{E}_0$ as a $\mathfrak{gl}(n)$-valued 1-form, $d\mathcal{E}_0\in\Omega^1(\mathbb{R}^n,\mathfrak{gl}(n))$, both using the identification $T_0\mathrm{dom}(\mathcal{H})\cong\mathbb{R}^n$ obtained with the affine embedding $d\Phi_0$ as in equation \eqref{dom(H)_H_iso}. With
\begin{align*}
dh_{\left(\begin{smallmatrix}
x\\
y
\end{smallmatrix}\right)}&=\left(\tau x^{\tau-1}-(\tau-2)x^{\tau-3}\langle y,y\rangle + \sum\limits_{i=3}^{\tau-1}(\tau-i)x^{\tau-i-1}P_i(y)\right)dx\\
&\quad-2x^{\tau-2}\langle y,dy\rangle + \sum\limits_{i=3}^\tau x^{\tau-i} i P_i(y,\ldots,y,dy)
\end{align*}
we get
\begin{align}
dh_{\left(\begin{smallmatrix}
x\\
y
\end{smallmatrix}\right)}\left(d\mathcal{A}_0\cdot\left(\begin{smallmatrix}
x\\
y
\end{smallmatrix}\right)\right)&=x^{\tau-2}\left(-2\langle y,d\mathcal{E}_0y\rangle + 3P_3(y,y,dz)\right) +(\text{terms of lower order in }x).\label{eqn_xtauminus2stuff}
\end{align}
The assumption that $A$ fulfils \eqref{std_form_i} and \eqref{std_form_ii} in Proposition \eqref{prop_std_form_h} and $A(0)=\mathbbm{1}$ implies that the $x^{\tau-2}$-term in the above equation \eqref{eqn_xtauminus2stuff} must vanish, i.e. $-2\langle y,d\mathcal{E}_0y\rangle + 3P_3(y,y,dz)=0$.
This is true if and only if 
\begin{equation}\label{dE0_eqn}
d\mathcal{E}_0(y)=\frac{3}{2}P_3(y,\cdot,dz)^T+dB_0y
\end{equation}
for all $y\in\mathbb{R}^n$, with $dB_0\in\mathrm{Lin}(\mathbb{R}^n,\mathfrak{so}(n))$ a linear map from $\mathbb{R}^n$ to $\mathfrak{so}(n)$. Here we have identified $\mathbb{R}^n$ with $T_0\mathrm{dom}(\mathcal{H})$.
We will now justify our notation of the endomorphism $dB_0$. Consider for any smooth map $B:\mathbb{R}^n\to \mathrm{O}(n)$ with $B(0)=\mathbbm{1}$ the map
\begin{equation*}
B\cdot\mathcal{E}:\Phi^{-1}(V)\to\mathrm{GL}(n+1),\quad z\mapsto B(z)\cdot\mathcal{E}(z).
\end{equation*}
It is clear that if we replace $E$ with $(B\circ \Phi^{-1})\cdot E$ in the map $A$ (and correspondingly $B\cdot\mathcal{E}$ in $\mathcal{A}$), it will still fulfil \eqref{std_form_i} and \eqref{std_form_ii} in Proposition \ref{prop_std_form_h} and $A(0)=\mathbbm{1}$. We can thus choose for any $dB_0\in\Omega^1(\mathbb{R}^n,\mathfrak{so}(n))$ a fitting map $B:\mathbb{R}^n\to \mathrm{O}(n)$ and a smooth map $\widetilde{\mathcal{E}}:\Phi^{-1}(V)\to\mathrm{GL}(n)$ with $d\widetilde{\mathcal{E}}_0=\frac{3}{2}P_3(y,\cdot,dz)^T$, $\widetilde{\mathcal{E}}(0)=\mathbbm{1}$, so that $\mathcal{E}:=B\cdot\widetilde{\mathcal{E}}$ will fulfil equation \eqref{dE0_eqn}. Also note that the requirement $B(0)=\mathbbm{1}$ implies that the image of $B$ lies in $\mathrm{SO}(n)$.

To complete the proof, we only need to replace $d\mathcal{E}_0$ in $dh_{\left(\begin{smallmatrix}
x\\
y
\end{smallmatrix}
\right)}\left(d\mathcal{A}_0\left(\begin{smallmatrix}
x\\
y
\end{smallmatrix}
\right)\right)$ as in equation \eqref{dE0_eqn} and obtain the claimed result.
\end{proof}
\end{Prop}
Equation \eqref{eqn_deltaPkdeterminer} in Proposition \ref{inf_p-moving_prop} determines precisely the infinitesimal changes of the $P_k$'s in the polynomial $h$ as in equation \eqref{standard_form_h} when changing coordinates for $p\in\mathcal{H}\subset\{h=1\}$ parametrised by $\Phi:\mathrm{dom}(\mathcal{H})\to\mathcal{H}$ \eqref{dom(H)_H_iso} in the way described by Proposition \ref{prop_std_form_h}. Rotations in $y\in\mathbb{R}^n\subset\mathbb{R}^{n+1}$ always preserve \eqref{standard_form_h}, which is seen in the freedom of choosing $dB_0\in\mathrm{Lin}(\mathbb{R}^n,\mathfrak{so}(n))$. We will now assign symbols to the respective infinitesimal changes of the $P_k$'s in order to simplify the considerations to follow.

\begin{Def}
\label{inf_Pk_def}
With the assumptions of Proposition \ref{inf_p-moving_prop} and the definition of $\mathcal{A}$ as in \eqref{A_matrix_on_dom(H)}, we define for $\tau\geq 3$ and $3\leq k\leq \tau$
\begin{align}
\delta P_k(y)&:=\left.\frac{1}{(\tau-k)!}\partial_x^{\tau-k} \frac{\partial}{\partial z} h\left(\mathcal{A}(z)\cdot\left(\begin{smallmatrix}
x\\ y
\end{smallmatrix}
\right)\right)\right|_{\left(\begin{smallmatrix}
x\\
y
\end{smallmatrix}
\right)=\left(\begin{smallmatrix}
0\\
y
\end{smallmatrix}
\right),\ z=0}\notag
\\
&=\left.\left(\frac{1}{(\tau-k)!}\partial_x^{\tau-k}dh_{\left(\begin{smallmatrix}
x\\
y
\end{smallmatrix}
\right)}\left(d\mathcal{A}_0\cdot\left(\begin{smallmatrix}
x\\
y
\end{smallmatrix}
\right)\right)\right)\right|_{\left(\begin{smallmatrix}
x\\
y
\end{smallmatrix}
\right)=\left(\begin{smallmatrix}
0\\
y
\end{smallmatrix}
\right)},\label{delta_P_k_general_eqn}
\end{align}
where we denote by $\frac{\partial}{\partial z}=\sum\limits_{i=1}^n dz_i\otimes\partial_{z_i}$ the de-Rham differential with respect to the $z=(z_1,\ldots,z_n)^T$-coordinates. In particular, we have for $\tau=3$, that is cubic polynomials $h$,
\begin{equation}
\delta P_3(y)=-\frac{2}{3}\langle y,y\rangle\langle y,dz\rangle + 3P_3\left(y,y,dB_0y+\frac{3}{2}P_3(y,\cdot,dz)^T\right), \label{delta_P_3_cubics}
\end{equation}
and for $\tau=4$, that is quartic polynomials $h$,
\begin{align*}
\delta P_3(y)&=-\langle y,y\rangle\langle y,dz\rangle + 3P_3\left(y,y,dB_0y+\frac{3}{2}P_3(y,\cdot,dz)^T\right)+4P_4(y,y,y,dz),\\
\delta P_4(y)&=\frac{1}{2}P_3(y)\langle y,dz\rangle + 4P_4\left(y,y,y,dB_0y+\frac{3}{2}P_3(y,\cdot,dz)^T\right).\label{delta_P_4_quartics}
\end{align*}
This means that the $\delta P_k(y)$'s are precisely the factors depending on $y$ in the summands of
\begin{equation*}
dh_{\left(\begin{smallmatrix}
x\\
y
\end{smallmatrix}
\right)}\left(d\mathcal{A}_0\cdot\left(\begin{smallmatrix}
x\\
y
\end{smallmatrix}
\right)\right)=\sum\limits_{i=3}^{\tau}x^{\tau-i}\delta P_i(y)
\end{equation*}
that are of order $x^{\tau-k}$, respectively. For each $3\leq k\leq\tau$ we call $\delta P_k$ \textsf{the first variation of} $P_k$ \textsf{along} $\mathcal{H}$ with respect to the chosen $dB_0$ \eqref{dE0_eqn}, respectively $d\mathcal{A}_0$ \eqref{eqn_dA0_on_domH}, and understand $\delta P_k(y)$ as a linear map $\delta P_k(y):\mathbb{R}^n\to\mathrm{Sym}^k\left(\mathbb{R}^n\right)^*$, so that we insert vectors $v\in\mathbb{R}^n$ into the $dz$ in each $\delta P_k(y)$ and obtain a homogeneous polynomial in $(y_1,\ldots,y_n)$ of degree $k$.
\end{Def}

The first application for Proposition \ref{inf_p-moving_prop} that we will consider is calculating the first derivative of the scalar curvature of a connected GPSR manifold $\mathcal{H}$ equipped with its centro-affine fundamental form at one certain point. To do so we need a closed form of the scalar curvature (at at least one point). Its calculation uses the following result.

\begin{Lem}
\label{lemma_pullbackmetric_domH}
Let $\mathcal{H}$ be a connected GPSR manifold of degree $\tau\geq 3$ in standard form with centro-affine fundamental form $g_{\mathcal{H}}=-\frac{1}{\tau}\partial^2 h|_{T\mathcal{H}\times T\mathcal{H}}$ (cf. Lemma \ref{caff_lemma}) and
let $\Phi:\mathrm{dom}(\mathcal{H})\to\mathcal{H}$ be the diffeomorphism given in equation \eqref{dom(H)_H_iso} and $\beta$ as in equation \eqref{beta_def}. Then
\begin{equation}\label{Phi_pullback_g_eqn}
\left(\Phi^*g_{\mathcal{H}}\right)_z=-\frac{\partial^2\beta_z}{\tau\beta(z)}+\frac{(\tau-1)d\beta_z^2}{\tau^2\beta^2(z)}.
\end{equation}
\begin{proof}
This is a special case of \cite[Cor.\,1.13]{CNS}. To check the claim, one uses the homogeneity of degree $\tau-2\geq 1$ of $\partial^2h_p$ in $p$ and the first derivative of the diffeomorphism $\Phi:\mathrm{dom}(\mathcal{H})\to\mathcal{H}$ \eqref{dom(H)_H_iso}, that is
\begin{equation*}
d\Phi_z=\frac{1}{\beta^{1/\tau}(z)}\left(\begin{matrix}
0\\
dz
\end{matrix}
\right)-\frac{1}{\tau\beta^{(\tau+1)/\tau}(z)}
\left(\begin{matrix}
1\\
z
\end{matrix}
\right)\otimes  d\beta_z.
\end{equation*}
\end{proof}
\end{Lem}
We will use equation \eqref{Phi_pullback_g_eqn} to calculate the scalar curvature of $(\mathrm{dom}(\mathcal{H}),\Phi^*g_\mathcal{H})$ at $z=0\in\mathrm{dom}(\mathcal{H})$.
\begin{Prop}
\label{prop_scal_GPSR}
Let $\mathcal{H}$ be an $n\geq 2$-dimensional connected GPSR manifold of degree $\tau\geq 3$ in standard form.
Then the scalar curvature $S_\mathcal{H}:\mathcal{H}\to\mathbb{R}$ of $(\mathcal{H},g_{\mathcal{H}})$ at the point $\left(\begin{smallmatrix}
1\\
0
\end{smallmatrix}\right)\in\mathcal{H}$ is given by
\begin{equation}
S_{\mathcal{H}}\left(\left(\begin{smallmatrix}
1\\
0
\end{smallmatrix}\right)\right)=n(1-n)+\frac{9\tau}{8}\sum\limits_{a,i,\ell}\left(-P_3(\partial_a,\partial_a,\partial_\ell)P_3(\partial_i,\partial_i,\partial_\ell)+P_3(\partial_a,\partial_i,\partial_\ell)^2\right),\label{eqn_GPSR_scal_formula}
\end{equation}
where $\partial_k=\partial_{y_k}$ for $1\leq k\leq n$.
\begin{proof}
We identify $\partial_{z_i}=\partial_{y_i}=\partial_i$ when inserting vectors in the polynomials $P_k$. This is justified by the fact that $d\Phi_0$ bijectively maps $T_0\mathrm{dom}(\mathcal{H})$ to $T_{\left(\begin{smallmatrix}
1\\ 0
\end{smallmatrix}\right)}\mathcal{H}$ via $d\Phi_0:\partial_{z_i}\mapsto\partial_{y_i}$ for all $1\leq i\leq n$.
For the following calculations we will first calculate the scalar curvature $S:\mathrm{dom}(\mathcal{H})\to\mathbb{R}$ of $\left(\mathrm{dom}(\mathcal{H}),g:=\tau\Phi^*g_{\mathcal{H}}\right)$ at $z=0$. We work with $g$ instead of $\Phi^*g_\mathcal{H}$ because the necessary calculations will then require less symbols. Furthermore, we will for the general calculations assume that $\tau\geq 4$. The calculations for $\tau=3$ are analogous.
From equation \eqref{Phi_pullback_g_eqn} it follows that $g$ is given by
\begin{equation*}
g=-\frac{\partial^2\beta_z}{\beta(z)}+\frac{\tau-1}{\tau}\frac{d\beta_z^2}{\beta^2(z)}.
\end{equation*}
We abbreviate $\partial_{z_\mu}=\partial_\mu$ and obtain for the first entry-wise derivative of $g$ in $z_\mu$-direction
\begin{align*}
\partial_\mu g&=\beta^{-3}\left(\frac{-2\tau+2}{\tau}d\beta(\partial_\mu)d\beta^2\right) +\beta^{-2}\left(\frac{2\tau-2}{\tau}d\beta \partial^2\beta(\partial_\mu,\cdot)+d\beta(\partial_\mu)\partial^2\beta\right)+ \beta^{-1}\left(-\partial^3\beta(\partial_\mu,\cdot,\cdot)\right).
\end{align*}
The second partial derivatives of $g$ read
\begin{align*}
\partial_\nu\partial_\mu g&=\beta^{-4}\left(\frac{6\tau-6}{\tau}d\beta(\partial_\nu)d\beta(\partial_\mu)d\beta^2\right)\\
&+\beta^{-3}\left(\frac{-4\tau+4}{\tau}\left(d\beta(\partial_\nu)d\beta \partial^2\beta(\partial_\mu,\cdot)+d\beta(\partial_\mu)d\beta\partial^2\beta(\partial_\nu,\cdot)\right)\right.\\
&\quad\left.\vphantom{\frac{-4\tau+4}{\tau}}-2d\beta(\partial_\nu)d\beta(\partial_\mu)\partial^2\beta+\frac{-2\tau+2}{\tau}\partial^2(\partial_\nu,\partial_\mu)d\beta^2\right)\\
&+\beta^{-2}\left( d\beta(\partial_\nu)\partial^3\beta(\partial_\mu,\cdot,\cdot)+d\beta(\partial_\mu)\partial^3\beta(\partial_\nu,\cdot,\cdot) + \frac{2\tau-2}{\tau}d\beta\partial^3\beta(\partial_\nu,\partial_\mu,\cdot)\right.\\
&\quad\left.\vphantom{}+\partial^2\beta(\partial_\nu,\partial_\mu)\partial^2\beta+\frac{2\tau-2}{\tau}\partial^2\beta(\partial_\nu,\cdot)\partial^2\beta(\partial_\mu,\cdot)\right)\\
&+\beta^{-1}\left(-\partial^4\beta(\partial_\nu,\partial_\mu,\cdot,\cdot)\right).
\end{align*}
Applying the above formulas at $z=0$, we obtain
\begin{align}
g|_0&=2\langle dz,dz\rangle,\label{eqn_g_0}\\
g^{-1}|_0&=\frac{1}{2}\langle \partial_z,\partial_z\rangle,\label{eqn_g^-1_0}\\
\partial_\mu g|_0&=-\partial^3\beta_0(\partial_\mu,\cdot,\cdot)=-6P_3(\partial_\mu,\cdot,\cdot),\notag\\
\partial_\mu (g^{-1})|_0&=-g^{-1}|_0\partial_\mu g|_0 g^{-1}|_0=\frac{3}{2}P_3(\partial_\mu,\cdot,\cdot),\notag\\
\partial_\nu\partial_\mu g|_0&=\partial^2\beta_0(\partial_\mu,\partial_\nu)\partial^2\beta_0+\frac{2\tau-2}{\tau}\partial^2\beta_0(\partial_\nu,\cdot)\partial^2\beta_0(\partial_\mu,\cdot)-\partial^4\beta_0(\partial_\nu,\partial_\mu,\cdot,\cdot)\notag\\
&=4\delta_\mu^\nu\langle\cdot,\cdot\rangle+\frac{8\tau-8}{\tau}\langle\partial_\nu,\cdot\rangle\langle\partial_\mu,\cdot\rangle -24P_4(\partial_\nu,\partial_\mu,\cdot,\cdot),\notag\\
\partial_\nu\partial_\mu g_{ij}|_0&=4\delta_\mu^\nu\delta_i^j+\frac{4\tau-4}{\tau}\left(\delta_\nu^i\delta_\mu^j+\delta_\nu^j\delta_\mu^i\right)-24P_4(\partial_\nu,\partial_\mu,\partial_i,\partial_j).\notag
\end{align}
In order to calculate the scalar curvature of $(\mathrm{dom}(\mathcal{H}),g)\cong(\mathcal{H},\tau g_{\mathcal{H}})$, 
\begin{equation*}
S=\sum\limits_{a,i,j}\left(\partial_a \Gamma_{ji}^a-\partial_j\Gamma_{ia}^a+\sum\limits_{k}\left(\Gamma_{ij}^k \Gamma_{ak}^a-\Gamma_{ia}^k\Gamma_{jk}^a\right)\right)g^{ij},
\end{equation*}
at $z=0$, we need to calculate the Christoffel symbols and their first derivatives at that point. We have
\begin{align*}
\Gamma_{ij}^k&=\frac{1}{2}\sum\limits_\ell \left( \partial_i g_{j\ell} + \partial_j g_{i\ell} - \partial_\ell g_{ij}\right) g^{\ell k},\\
\partial_a \Gamma_{ij}^k&=\frac{1}{2}\sum\limits_\ell \left(\left( \partial_a\partial_i g_{j\ell} + \partial_a\partial_j g_{i\ell} - \partial_a\partial_\ell g_{ij}\right) g^{\ell k}+\left( \partial_i g_{j\ell} + \partial_j g_{i\ell} - \partial_\ell g_{ij}\right)\partial_a g^{\ell k}\right),
\end{align*}
and
\begin{align}
\left.\Gamma_{ij}^k\right|_0&=-\frac{3}{2}P_3(\partial_i,\partial_j,\partial_k),\label{eqn_christoffel_GPSR}\\
\left.\partial_a\Gamma_{ij}^k\right|_0&=\delta_a^i\delta_j^k + \delta_a^j\delta_i^k+\frac{\tau - 2}{\tau}\delta_a^k\delta_i^j - 6P_4(\partial_a,\partial_i,\partial_j,\partial_k)-\frac{9}{2}\sum\limits_\ell P_3(\partial_a,\partial_k,\partial_\ell)P_3(\partial_i,\partial_j,\partial_\ell).\label{eqn_1st_derivative_christoffel_GPSR}
\end{align}
Since $\left.g^{ij}\right|_0=\frac{1}{2}\delta_i^j$,
\begin{align*}
S(0)&=\left.\frac{1}{2}\sum\limits_{a,i}\left(\partial_a\Gamma_{ii}^a-\partial_i\Gamma_{ia}^a+\sum\limits_k\left(\Gamma_{ii}^k\Gamma_{ak}^a - \Gamma_{ia}^k\Gamma_{ik}^a\right)\right)\right|_0.
\end{align*}
We obtain
\begin{align*}
\left.\partial_a\Gamma_{ii}^a\right|_0&=\frac{\tau-2}{\tau}+2\delta_a^i-6P_4(\partial_a,\partial_a,\partial_i,\partial_i)-\frac{9}{2}\sum\limits_\ell P_3(\partial_a,\partial_a,\partial_\ell)P_3(\partial_i,\partial_i,\partial_\ell),\\
\left.\partial_i\Gamma_{ia}^a\right|_0&=1+\frac{2\tau-2}{\tau}\delta_i^a - 6P_4(\partial_a,\partial_a,\partial_i,\partial_i)-\frac{9}{2}\sum\limits_\ell P_3(\partial_a,\partial_i,\partial_\ell),\\
\left.\Gamma_{ii}^k\Gamma_{ak}^a\right|_0&=\frac{9}{4}P_3(\partial_a,\partial_a,\partial_k)P_3(\partial_i,\partial_i,\partial_k),\quad
\left.\Gamma_{ia}^k\Gamma_{ik}^a\right|_0=\frac{9}{4}\left(P_3(\partial_a,\partial_i,\partial_k)\right)^2.
\end{align*}
Hence,
\begin{align*}
S(0)&=\frac{n(1-n)}{\tau}+\frac{9}{8}\sum\limits_{a,i,\ell}\left(-P_3(\partial_a,\partial_a,\partial_\ell)P_3(\partial_i,\partial_i,\partial_\ell)+\left(P_3(\partial_a,\partial_i,\partial_\ell)\right)^2\right).
\end{align*}
Recall that $d\Phi_0\left(\partial_{z_i}\right)=\partial_{y_i}$ for all $1\leq i\leq n$, which one can easily verify. Thus $S_{\mathcal{H}}\left(\left(\begin{smallmatrix}
1\\
0
\end{smallmatrix}\right)\right)=\tau S(0)$ together with the above equation prove our claim. Observe that $S_{\mathcal{H}}\left(\left(\begin{smallmatrix}
1\\
0
\end{smallmatrix}\right)\right)$ only depends on the dimension $n$ of $\mathcal{H}$, the degree of homogeneity $\tau$, and the cubic polynomial $P_3$. Also note that $S_\mathcal{H}\equiv 0$ for $\dim(\mathcal{H})=n=1$ is consistent with formula \eqref{eqn_GPSR_scal_formula}.
\end{proof}
\end{Prop}
Proposition \ref{prop_scal_GPSR} gives us, at least in theory, a simple way of calculating the scalar curvature of a connected GPSR manifold $\mathcal{H}$ equipped with its centro-affine fundamental form (and thus of all GPSR manifolds by considering each connected component) at every point $p\in\mathcal{H}$. This, however, requires calculating $A(p)$ as in Proposition \ref{prop_std_form_h} for each $p\in\mathcal{H}$. This amounts basically to determining an orthonormal basis for a positive definite bilinear form depending on $p\in\mathcal{H}$. This is certainly easier than calculating Christoffel-symbols and their derivatives at each point, but nevertheless complicated enough to require a (both $p$- and $\mathcal{H}$-dependent) case-by-case study and not giving us a closed form of $S_\mathcal{H}(p)$ for all $p\in\mathcal{H}$.
Calculating the first derivative of the scalar curvature $S_\mathcal{H}$ at the point $\left(\begin{smallmatrix}
x\\
y
\end{smallmatrix}\right)=\left(\begin{smallmatrix}
1\\
0
\end{smallmatrix}\right)\in\mathcal{H}$ can of course also be done in a direct way, but the calculations require the (local) calculation of the third partial derivatives of the metric $g_\mathcal{H}$ and, hence, are very long and have a huge potential for human error. One can however also make use of Proposition \ref{inf_p-moving_prop} to obtain a formula for $dS_\mathcal{H}|_{\left(\begin{smallmatrix}
1\\
0
\end{smallmatrix}\right)}$.

\begin{Prop}
\label{scalar_curvature_first_derivative_prop}
With the assumptions and notations of Proposition \ref{prop_scal_GPSR} and Definition \ref{inf_Pk_def} and identifying $T_{\left(\begin{smallmatrix}
1\\
0
\end{smallmatrix}\right)}\mathcal{H}$ with the affinely embedded hyperplane $\left.\left\{{\left(\begin{smallmatrix}
0\\
y
\end{smallmatrix}\right)}\ \right|\ y\in\mathbb{R}^n\right\}\subset\mathbb{R}^{n+1}$, we have for $\tau\geq 4$
\begin{align*}
\left.dS_\mathcal{H}\right|_{\left(\begin{smallmatrix}
1\\
0
\end{smallmatrix}\right)}&=
\left(
\sum\limits_{a}\frac{3(n-1)(\tau-2)}{2}P_3(\partial_a,\partial_a,dy)
\right)\notag\\
&+
\left(
\sum\limits_{a,i,j}9\tau (-P_3(\partial_a,\partial_a,\partial_j)P_4(\partial_i,\partial_i,\partial_j,dy)+P_3(\partial_a,\partial_i,\partial_j)P_4(\partial_a,\partial_i,\partial_j,dy))
\right)\notag\\
&+
\sum\limits_{a,i,j,\ell} \frac{27\tau}{8}P_3(\partial_j,\partial_\ell,dy)\big(-P_3(\partial_a,\partial_a,\partial_j)P_3(\partial_i,\partial_i,\partial_\ell)\notag\\
&\quad -2P_3(\partial_a,\partial_a,\partial_i)P_3(\partial_i,\partial_j,\partial_\ell) + 3P_3(\partial_a,\partial_i,\partial_j)P_3(\partial_a,\partial_i,\partial_\ell)\big)
\end{align*}
and for $\tau=3$
\begin{align*}
\left.dS_\mathcal{H}\right|_{\left(\begin{smallmatrix}
1\\
0
\end{smallmatrix}\right)}&=
\left(
\sum\limits_{a}\frac{3(n-1)}{2}P_3(\partial_a,\partial_a,dy) 
\right) +
\sum\limits_{a,i,j,\ell} \frac{81}{8}P_3(\partial_j,\partial_\ell,dy)\big(-P_3(\partial_a,\partial_a,\partial_j)P_3(\partial_i,\partial_i,\partial_\ell) \notag\\
&\quad-2P_3(\partial_a,\partial_a,\partial_i)P_3(\partial_i,\partial_j,\partial_\ell) + 3P_3(\partial_a,\partial_i,\partial_j)P_3(\partial_a,\partial_i,\partial_\ell)\big).
\end{align*}

\begin{proof}
In the following calculations we will identify $dz$ and $dy$, respectively each $\partial_{z_i}$ and $\partial_{y_i}$ (and write $\partial_i$ instead) via $d\Phi_0$, cf. equation \eqref{dom(H)_H_iso}, which has the property that $d\Phi_0\left(\partial_{z_i}\right)=\partial_{y_i}$ for all $1\leq i\leq n$. We start with the case $\tau\geq4$. With the notations of Definition \ref{inf_Pk_def}, Propositions \ref{prop_scal_GPSR} and Proposition \ref{inf_p-moving_prop} equation \eqref{eqn_deltaPkdeterminer} imply
\begin{equation}\label{dS_pre-eqn}
\left.dS_\mathcal{H}\right|_{\left(\begin{smallmatrix}
1\\
0
\end{smallmatrix}\right)}=\frac{9\tau}{4}\sum\limits_{a,i,\ell}\left(- P_3(\partial_a,\partial_a,\partial_\ell) \delta P_3(\partial_i,\partial_i,\partial_\ell)+  P_3(\partial_a,\partial_i,\partial_\ell) \delta P_3(\partial_a,\partial_i,\partial_\ell) \right),
\end{equation}
where
\begin{equation}\label{delta_P3_tau_geq4}
\delta P_3(y)=\frac{-2(\tau-2)}{\tau}\langle y,y\rangle \langle y,dz\rangle + 3P_3\left(y,y,dB_0 y+\frac{3}{2}P_3(y,\cdot,dz)^T\right)+4P_4(y,y,y,dz).
\end{equation}
Recall that $dB_0\in\mathrm{Lin}(\mathbb{R}^n,\mathfrak{so}(n))$. We thus need to determine a formula for $\delta P_3(\partial_i,\partial_j,\partial_k)$ for all $1\leq i,j,k\leq n$. The safest way in the sense that possible errors in the pre-factors do not occur is to determine $\partial^2(\delta P_3(y))$, where we regard $dz$ in equation \eqref{delta_P3_tau_geq4} as a constant vector. We obtain
\begin{align*}
d(\delta P_3)_y(v)&=\frac{-4(\tau-2)}{\tau}\langle y,v\rangle\langle y,dz\rangle+\frac{-2(\tau-2)}{\tau}\langle y,y\rangle\langle v,dz\rangle\notag\\
&+6P_3\left(y,v,dB_0y+\frac{3}{2}P_3(y,\cdot,dz)^T\right) + 3P_3\left(y,y,dB_0v+\frac{3}{2}P_3(v,\cdot,dz)^T\right)\notag\\
&+12P_4(y,y,v,dz)
\end{align*}
for all $y,v\in\mathbb{R}^n$ and, hence,
\begin{align*}
\partial^2(\delta P_3)_y(v,w)&=\frac{-4(\tau-2)}{\tau}(\langle v,w\rangle\langle y,dz\rangle + \langle y,v\rangle\langle w,dz\rangle + \langle y,w\rangle\langle v,dz\rangle)\\
& + 6 P_3\left(v,w,dB_0y + \frac{3}{2} P_3(y,\cdot,dz)^T\right)\\
& + 6 P_3\left(y,v,dB_0w + \frac{3}{2} P_3(w,\cdot,dz)^T\right)\\
& + 6 P_3\left(y,w,dB_0v + \frac{3}{2} P_3(v,\cdot,dz)^T\right)\\
& + 24 P_4(y,v,w,dz)
\end{align*}
for all $y,v,w\in\mathbb{R}^n$. Since $\delta P_3(y)$ is homogeneous of degree $3$ in $y$, we have the identities
\begin{equation*}
d(\delta P_3)_y(v) =3 \delta P_3(y,y,v),\quad
\partial^2(\delta P_3)_y(v,w) =6\delta P_3(y,v,w),
\end{equation*}
when we regard $dz$ as a constant vector and interpret $\delta P_3$ as a cubic tensor. We use the above identities and obtain
\begin{align*}
&\sum\limits_{a,i,\ell} P_3(\partial_a,\partial_a,\partial_\ell) \delta P_3(\partial_i,\partial_i,\partial_\ell)\\
&=\left(\sum\limits_{a,\ell}P_3(\partial_a,\partial_a,\partial_\ell)\frac{(\tau-2)(-2n-4)}{3\tau}\langle\partial_\ell, dz\rangle\right)\\
&+\left(\sum\limits_{a,i,\ell}P_3(\partial_a,\partial_a,\partial_\ell)\left(2P_3(\partial_i,\partial_\ell,dB_0\partial_i)+P_3(\partial_i,\partial_i,dB_0\partial_\ell)+4P_4(\partial_i,\partial_i,\partial_\ell,dz)\right)\right)\\
&+\left(\sum\limits_{a,i,\ell,j}P_3(\partial_a,\partial_a,\partial_\ell)\left(3P_3(\partial_i,\partial_\ell,\partial_j)P_3(\partial_i,\partial_j,dz)+\frac{3}{2}P_3(\partial_i,\partial_i,\partial_j)P_3(\partial_\ell,\partial_j,dz)\right)\right)
\end{align*}
and
\begin{align*}
&\sum\limits_{a,i,\ell} P_3(\partial_a,\partial_i,\partial_\ell) \delta P_3(\partial_a,\partial_i,\partial_\ell)\notag\\
&=\left(\sum\limits_{a,\ell}P_3(\partial_a,\partial_\ell,\partial_\ell)\frac{-2(\tau-2)}{\tau}\langle\partial_a,dz\rangle \right)\notag\\
&+\left(\sum\limits_{a,i,\ell}P_3(\partial_a,\partial_i,\partial_\ell)\left( 3P_3(\partial_a,\partial_i,dB_0\partial_\ell) + 4P_4(\partial_a,\partial_i,\partial_\ell,dz) \right)\right)\notag\\
&+\left(\sum\limits_{a,i,\ell,j}P_3(\partial_a,\partial_i,\partial_\ell)\left( \frac{9}{2}P_3(\partial_a,\partial_i,\partial_j)P_3(\partial_\ell,\partial_j,dz) \right)\right).
\end{align*}
To see that all terms containing $dB_0:\mathbb{R}^n\to\mathfrak{so}(n)$ (understood as in equation \eqref{dB0_eqn}) vanish, observe that for all $1\leq a,i,\ell\leq n$ the tensors
\begin{align*}
P_3(\partial_a,\partial_a,\partial_\ell)P_3(\cdot,\partial_\ell,\cdot),\quad P_3(\partial_a,\partial_a,\cdot)P_3(\partial_i,\partial_i,\cdot),\quad P_3(\partial_a,\partial_i,\cdot)P_3(\partial_a,\partial_i,\cdot)
\end{align*}
are symmetric in their two arguments. Their trace with respect to the standard Euclidean scalar product $\langle \cdot,\cdot\rangle$ on $\mathbb{R}^n$ when inserting any matrix $M\in\mathfrak{so}(n)$ in one of the arguments thus vanishes.
We can now use the above formulas for $\sum\limits_{a,i,\ell} P_3(\partial_a,\partial_a,\partial_\ell) \delta P_3(\partial_i,\partial_i,\partial_\ell)$ and $\sum\limits_{a,i,\ell} P_3(\partial_a,\partial_i,\partial_\ell) \delta P_3(\partial_a,\partial_i,\partial_\ell)$ in equation \eqref{dS_pre-eqn} and, with the identification of $dz$ and $dy$ via $d\Phi_0$ \eqref{dom(H)_H_iso}, obtain our claimed result for $\tau\geq 4$. For $\tau=3$, observe that the formulas for $\delta P_3$ in equations \eqref{delta_P_3_cubics} and \eqref{delta_P3_tau_geq4} coincide when setting $P_4\equiv 0$. The calculations for the case $\tau=3$ thus coincide with the cases $\tau\geq 4$ and we obtain the claimed result.
\end{proof}
\end{Prop}
The calculations used in Proposition \ref{prop_scal_GPSR} can also be used to calculate the Riemannian curvature tensor, the Ricci curvature, and the sectional curvatures of a connected GPSR manifold $(\mathcal{H},g_\mathcal{H})$.
\begin{Lem}
\label{lemma_R_Ric_K_GPSR}
With the assumptions of Proposition \ref{prop_scal_GPSR}, let $R$ denote the Riemannian curvature tensor, $\mathrm{Ric}$ denote the Ricci curvature, and $K$ denote the sectional curvature of an $n$-dimensional connected GPSR manifold $(\mathcal{H},g_{\mathcal{H}})$, respectively. We again identify $dz$ and $dy$ at $\left(\begin{smallmatrix}
1\\
0
\end{smallmatrix}\right)\in\mathcal{H}$ via $d\Phi_0$ \eqref{dom(H)_H_iso}. Then
\begin{align}
R_{\left(\begin{smallmatrix}
1\\
0
\end{smallmatrix}\right)}(\partial_i,\partial_j)\partial_k&=\frac{2}{\tau}\left(\delta_i^k\partial_j-\delta_j^k\partial_i\right)\notag\\
&+\frac{9}{4}\sum\limits_{a,\ell}\left(-P_3(\partial_i,\partial_\ell,\partial_a)P_3(\partial_j,\partial_k,\partial_a)+P_3(\partial_i,\partial_k,\partial_a)P_3(\partial_j,\partial_\ell,\partial_a)\right)\partial_\ell,\label{eqn_R_GPSR}
\\
\mathrm{Ric}_{\left(\begin{smallmatrix}
1\\
0
\end{smallmatrix}\right)}(\partial_j,\partial_k)&=\frac{2(1-n)}{\tau}\delta_j^k\notag\\
&+\frac{9}{4}\sum\limits_{a,i}\left(-P_3(\partial_i,\partial_i,\partial_a)P_3(\partial_j,\partial_k,\partial_a)+P_3(\partial_i,\partial_j,\partial_a)P_3(\partial_i,\partial_k,\partial_a)\right),\label{eqn_Ric_GPSR}
\end{align}
and for $\dim\left(\mathrm{span}\{v,w\}\right)=2$
\begin{equation}
K_{\left(\begin{smallmatrix}
1\\
0
\end{smallmatrix}\right)}(v,w)=-1+\frac{9\tau}{8}\sum\limits_{\ell}\left(-P_3(F\partial_i,F\partial_i,F\partial_\ell)P_3(F\partial_j,F\partial_j,F\partial_\ell)+P_3(F\partial_i,F\partial_j,F\partial_\ell)^2\right),\label{eqn_sectional_curvature_GPSR}
\end{equation}
where $F\in\mathrm{O}(n)$ is any orthogonal transformation with the property that $\mathrm{span}\{v,w\}=\mathrm{span}\{F\partial_i,F\partial_j\}$. Note that such a transformation $F$ always exists for any choices of $i\ne j$, and that $K(v,w)$ does in particular not depend on that choice of $i$, $j$, and the corresponding $F$.
\begin{proof}
The formulas \eqref{eqn_R_GPSR} and \eqref{eqn_Ric_GPSR} for the Riemannian curvature tensor $R$ and the Ricci tensor $\mathrm{Ric}$, respectively, follow directly from the formulas for the Christoffel symbols \eqref{eqn_christoffel_GPSR}, their first derivatives \eqref{eqn_1st_derivative_christoffel_GPSR}, and the inverse of $g_\mathcal{H}$ at the point $\left(\begin{smallmatrix}
x\\
y
\end{smallmatrix}\right)=\left(\begin{smallmatrix}
1\\
0
\end{smallmatrix}\right)$ \eqref{eqn_g^-1_0} (up to the factor $\tau$) given in the proof of Proposition \ref{prop_scal_GPSR}. Recall that in said proof we work with $g=\tau \Phi^*g_\mathcal{H}$, $\Phi$ as in \eqref{dom(H)_H_iso}, hence we also need to rescale the formula for $g$ at $0$ \eqref{eqn_g^-1_0} at the point where we take the trace with respect $g_\mathcal{H}$. For the sectional curvature $K$, the formula for $K_{\left(\begin{smallmatrix}
1\\
0
\end{smallmatrix}\right)}(\partial_i,\partial_j)$ for $i\ne j$ follows easily from \eqref{eqn_R_GPSR} and \eqref{eqn_g_0} (and by rescaling with the overall factor $\tau$). To find the general formula $K_{\left(\begin{smallmatrix}
1\\
0
\end{smallmatrix}\right)}(v,w)$ \eqref{eqn_sectional_curvature_GPSR} for any two linearly independent vectors $v,w\in T_{\left(\begin{smallmatrix}
1\\
0
\end{smallmatrix}\right)}\mathcal{H}\cong\mathbb{R}^n$, choose $i\ne j$ and $F\in\mathrm{O}(n)$ as described such that $\mathrm{span}\{v,w\}=\mathrm{span}\{F\partial_i,F\partial_j\}$. Changing the coordinates of the ambient $\mathbb{R}^{n+1}$ via
\begin{equation*}
\left(\begin{matrix}
\widetilde{x}\\
\widetilde{y}
\end{matrix}
\right)=
\left(\begin{matrix}
x\\
F^{-1}y
\end{matrix}
\right)
\end{equation*}
corresponds to rotating $\mathcal{H}$ in the $y$-coordinates and correspondingly changing the defining cubic polynomial $h$ to
\begin{equation*}
\widetilde{h}=\widetilde{x}^3-\widetilde{x}\langle \widetilde{y},\widetilde{y}\rangle + \widetilde{P}_3(\widetilde{y}),
\end{equation*}
with $\widetilde{P}_3(\widetilde{y})=P_3(F\widetilde{y})$. In the $\left(\begin{smallmatrix}
\widetilde{x}\\
\widetilde{y}
\end{smallmatrix}
\right)$-coordinates, let $\widetilde{K}$ denote the sectional curvature. By identifying $\partial_{\widetilde{y}_k}=\partial_{y_k}=\partial_k$ for all $1\leq k\leq n$ (as the $k$th unit vector in $\mathbb{R}^n$, \underline{not} via the map $F$) we have
\begin{align*}
K_{\left(\begin{smallmatrix}
1\\
0
\end{smallmatrix}\right)}(v,w)&=\widetilde{K}_{\left(\begin{smallmatrix}
1\\
0
\end{smallmatrix}\right)}(\partial_i,\partial_j)
=-1+\frac{9\tau}{8}\sum\limits_{\ell}\left(-\widetilde{P}_3(\partial_i, \partial_i, \partial_\ell)\widetilde{P}_3( \partial_j, \partial_j, \partial_\ell)+\widetilde{P}_3( \partial_i, \partial_j, \partial_\ell)^2\right)\\
&=-1+\frac{9\tau}{8}\sum\limits_{\ell}\left(-P_3(F\partial_i,F\partial_i,F\partial_\ell)P_3(F\partial_j,F\partial_j,F\partial_\ell)+P_3(F\partial_i,F\partial_j,F\partial_\ell)^2\right).
\end{align*}
\end{proof}
\end{Lem}

Another application of the first variation of the $P_k$'s as defined in Definition \ref{inf_Pk_def} is the study of homogeneous spaces that are
CCGPSR manifolds.
We will derive a sufficient condition for a connected GPSR manifold $\mathcal{H}\subset\{h=1\}$ to be a homogeneous space with respect to the action of $G_0^h$, that is the identity-component of the automorphism group $G^h$ of $h$.
\begin{Prop}
\label{prop_hom_spaces_deltaPk_condition}
Let $\mathcal{H}$ be a maximal connected GPSR manifold in standard form.
Let $\delta P_k(y):\mathbb{R}^n\to\mathrm{Sym}^k \left(\mathbb{R}^n\right)^*$ be as in equation \eqref{delta_P_k_general_eqn} depending on $dB_0\in\mathrm{Lin}\left(\mathbb{R}^n,\mathfrak{so}(n)\right)$ \eqref{dE0_eqn}, cf. Proposition \ref{inf_p-moving_prop}. Then the connected component containing the neutral element of the automorphism group of $h$, that is $G^h_0$, acts transitively on $\mathcal{H}$ if and only if there exists a choice for $dB_0\in\mathrm{Lin}\left(\mathbb{R}^n,\mathfrak{so}(n)\right)$, such that $\delta P_k(y)\equiv 0$ for all $3\leq k\leq \tau$. Furthermore, each of the latter two equivalent statements imply that $\mathcal{H}$ is a CCGPSR manifold.
\begin{proof}
It is clear that the action $G^h_0\times\mathcal{H}\to\mathcal{H}$ is well defined. Assume that $G^h_0$ acts transitively on the maximal GPSR manifold $\mathcal{H}\subset\{h=1\}$. Then $\mathcal{H}$ is, in particular, a CCGPSR manifold. For $p=\left(\begin{smallmatrix}
p_x\\ p_y
\end{smallmatrix}
\right)\in\mathcal{H}$ let $M(p)\in G^h_0$, such that $M(p)\cdot\left(\begin{smallmatrix}
1\\ 0
\end{smallmatrix}
\right)=p$. We will show that $M(p)$ is necessarily of the form \eqref{p_moving_A_matrix}. We immediately see that $M(p)$ is of the form
\begin{equation*}
M(p)=\left(\begin{tabular}{c|c}
$p_x$ & $v_p^T$ \\ \hline
$p_y$ & $\overset{}{W}(p)$
\end{tabular}
\right)
\end{equation*}
for some $v_p\in\mathbb{R}^n$ and $W(p)\in\mathrm{Mat}(n\times n,\mathbb{R})$. We calculate
\begin{align*}
h\left(M(p)\cdot \left(\begin{smallmatrix}
x\\ y
\end{smallmatrix}
\right)\right)&=x^\tau + x^{\tau-1} dh_p\left(\left(\begin{smallmatrix}
\langle v_p,y\rangle \\ W(p)y
\end{smallmatrix}
\right)\right) + 2x^{\tau-2}\partial^2 h_p\left(\left(\begin{smallmatrix}
\langle v_p,y\rangle \\ W(p)y
\end{smallmatrix}
\right),\left(\begin{smallmatrix}
\langle v_p,y\rangle \\ W(p)y
\end{smallmatrix}
\right)\right) \\
&\quad+ (\text{terms of lower order in }x).
\end{align*}
Since by assumption $h\equiv h\circ M(p)$ it follows that
\begin{equation}\label{eqn_autgrp_Mp_d}
\partial_x h_p \langle v_p,y\rangle + \partial_y h_p (W(p)y)=0 
\end{equation}
and
\begin{equation}\label{eqn_autgrp_Mp_dd}
2\partial^2 h_p\left(\left(\begin{smallmatrix}
\langle v_p,y\rangle \\ W(p)y
\end{smallmatrix}
\right),\left(\begin{smallmatrix}
\langle v_p,y\rangle \\ W(p)y
\end{smallmatrix}
\right)\right)=-\langle y,y\rangle
\end{equation}
for all $y\in\mathbb{R}^n$.
Suppose that $W(p)\not\in\mathrm{GL}(n)$. Then there exists $\overline{y}\in\mathbb{R}^n\setminus\{0\}$, such that $W(p)\overline{y}=0$. Then by \eqref{eqn_autgrp_Mp_dd}
\begin{equation*}
2\partial^2_x h_p\left\langle v_p,\overline{y}\right\rangle ^2=-\left\langle \overline{y},\overline{y}\right\rangle<0.
\end{equation*}
This in particular shows that $\left\langle v_p,\overline{y}\right\rangle\ne 0$. But then equation \eqref{eqn_autgrp_Mp_d} cannot be fulfilled since $\partial_y h_p(W(p)\overline{y})=\partial_yh_p(0)=0$ and $\partial_xh|_\mathcal{H}>0$ is true since $\mathcal{H}$ is a CCGPSR manifold, cf. proof of Proposition \ref{prop_std_form_h} equation \eqref{eqn_positivity_delxh}. We deduce that $W(p)\in\mathrm{GL}(n)$. Hence, setting $\widetilde{v}_p=W(p)^Tv_p$ in equation \eqref{eqn_autgrp_Mp_d} implies that $\widetilde{v}_p=-\left.\frac{\partial_y h}{\partial_xh}\right|_p$. This shows that 
\begin{equation*}
M(p)=\left(\begin{tabular}{c|c}
$p_x$ & $-\left.\frac{\partial_y h}{\partial_xh}\right|_p\circ W(p)$ \\ \hline
$p_y$ & $\overset{}{W}(p)$
\end{tabular}
\right)
\end{equation*}
is of the form \eqref{p_moving_A_matrix} as claimed. The action $G^h_0\times \mathcal{H}\to\mathcal{H}$ might not be simply transitive, but near $p=\left(\begin{smallmatrix}
1\\ 0
\end{smallmatrix}
\right)\in\mathcal{H}$, that is on some open neighbourhood $U\subset \mathcal{H}$ of $\left(\begin{smallmatrix}
1\\ 0
\end{smallmatrix}
\right)$, we can choose a smooth branch of the possible maps $W:U\to\mathrm{GL}(n)$ by the implicit function theorem. Then, using the diffeomorphism $\Phi:\mathrm{dom}(\mathcal{H})\to\mathcal{H}$ \eqref{dom(H)_H_iso}, $M\circ\Phi$ is locally on $\Phi^{-1}(U)$ a valid choice for $\mathcal{A}$ as in equation \eqref{A_matrix_on_dom(H)} and $d(W\circ\Phi)_0$ must fulfil the same equation as $\mathcal{E}$ in \eqref{dE0_eqn} in the proof of Proposition \ref{inf_p-moving_prop}. We now use the equality
\begin{equation*}
h\left(W(\Phi(z))\cdot\left(\begin{smallmatrix}
x\\ y
\end{smallmatrix}
\right)\right)=h\left(\left(\begin{smallmatrix}
x\\ y
\end{smallmatrix}
\right)\right)
\end{equation*}
for all $z\in\Phi^{-1}(U)$ to conclude with the definition of the $\delta P_k$'s \eqref{delta_P_k_general_eqn} that there exists a linear map $dB_0\in\mathrm{Lin}(\mathbb{R}^n,\mathfrak{so}(n))$, such that the corresponding functions $\delta P_k(y):\mathbb{R}^n\to\mathrm{Sym}^k\left((\mathbb{R}^n)^*\right)$ identically vanish for all $y\in\mathbb{R}^n$ and all $3\leq k\leq \tau$.

Now assume that there exists $dB_0\in\mathrm{Lin}(\mathbb{R}^n,\mathfrak{so}(n))$, such that $\delta P_k(y)\equiv 0$ for all $3\leq k\leq \tau$. Consider the corresponding map $\mathcal{A}:\Phi^{-1}(V)\to\mathrm{GL}(n+1)$ \eqref{A_matrix_on_dom(H)} for any open neighbourhood $V\subset\mathcal{H}$ of the point $\left(\begin{smallmatrix}
1\\ 0
\end{smallmatrix}\right)\in\mathcal{H}$ so that $\mathcal{A}$ is defined, with
\begin{equation*}
d\mathcal{A}_0=\left(
\begin{tabular}{c|c}
$\underset{}{0}$ &  $\frac{2}{\tau}dz^T$ \\ \hline
$dz$  & $\overset{}{\frac{3}{2}P_3(\cdot,\cdot,dz)^T+dB_0}$
\end{tabular}
\right),
\end{equation*}
cf. equations \eqref{dE0_eqn} and \eqref{eqn_dA0_on_domH}. Then $\delta P_k(y)\equiv 0$ for all $3\leq k\leq \tau$ implies that for all $v\in T_0\mathrm{dom}(\mathcal{H})\cong\mathbb{R}^n$
\begin{equation}\label{eqn_lie_alg_h_symgrp}
dh_{\left(\begin{smallmatrix}
x\\ y
\end{smallmatrix}
\right)}\left(d\mathcal{A}_0(v)\cdot\left(\begin{smallmatrix}
x\\ y
\end{smallmatrix}
\right)\right)\equiv 0,
\end{equation}
where $d\mathcal{A}_0(v)$ denotes the $\mathfrak{gl}(n+1)$-valued 1-form $d\mathcal{A}$ at $z=0$ applied to $v\in T_0\mathrm{dom}(\mathcal{H})$. With
\begin{equation*}
a_i:=d\mathcal{A}_0\left(\partial_{z_i}\right)
\end{equation*}
for $1\leq i\leq n$, the set of matrices $\{a_1,\ldots,a_n\}$ is linearly independent. Furthermore $\{a_1,\ldots,a_n\}\subset T_\mathbbm{1}G^h=T_\mathbbm{1}G^h_0$ which follows from \eqref{eqn_lie_alg_h_symgrp}. Let $\mu:G^h_0\to\mathcal{H}$, $\mu(a)=a\cdot \left(\begin{smallmatrix}
1\\ 0
\end{smallmatrix}\right)$, denote the action of $G^h_0$ on the point $\left(\begin{smallmatrix}
1\\ 0
\end{smallmatrix}\right)\in\mathcal{H}$. Then $d\mu_\mathbbm{1}(a_i)=\partial_{y_i}$
for all $1\leq i\leq n$. Hence, $d\mu_\mathbbm{1}:T_\mathbbm{1}G^h_0\to T_{\left(\begin{smallmatrix}
1\\ 0
\end{smallmatrix}\right)}\mathcal{H}$ is surjective (recall that with $h$ of the form \eqref{standard_form_h}, we view $T_{\left(\begin{smallmatrix}
1\\ 0
\end{smallmatrix}\right)}\mathcal{H}$ as the vector subspace $\left.\left\{\left(\begin{smallmatrix}
0\\ v
\end{smallmatrix}\right)\ \right|\ v\in\mathbb{R}^n\right\}\subset\mathbb{R}^{n+1}$). This shows that there exists an open subset $U\subset\mathcal{H}$, such that $\left(\begin{smallmatrix}
1\\ 0
\end{smallmatrix}\right)\in U$ and $U\subset G^h_0\cdot\left(\begin{smallmatrix}
1\\ 0
\end{smallmatrix}\right)$.
Suppose that the orbit $G^h_0\cdot\left(\begin{smallmatrix}
1\\ 0
\end{smallmatrix}\right)\subset\mathcal{H}$ is not open in $\mathcal{H}$. Then the set $\mathcal{H}\cap \partial\left(G^h_0\cdot\left(\begin{smallmatrix}
1\\ 0
\end{smallmatrix}\right)\right)\cap \left(G^h_0\cdot\left(\begin{smallmatrix}
1\\ 0
\end{smallmatrix}\right)\right)$ is non-empty. Let $q\in \mathcal{H}\cap \partial\left(G^h_0\cdot\left(\begin{smallmatrix}
1\\ 0
\end{smallmatrix}\right)\right)\cap \left(G^h_0\cdot\left(\begin{smallmatrix}
1\\ 0
\end{smallmatrix}\right)\right)$ and let $a(q)\in G^h_0$, such that $q=a(q)\cdot\left(\begin{smallmatrix}
1\\ 0
\end{smallmatrix}
\right)$. Since $q$ is by assumption an element of $\partial\left(G^h_0\cdot\left(\begin{smallmatrix}
1\\ 0
\end{smallmatrix}\right)\right)$ and $a(q)$ acts via linear transformations on $\mathbb{R}^{n+1}$ restricted to $\mathcal{H}$, there must exist $p\in\partial U$, such that $a(q)p=q$, because otherwise $q\not\in a(q)\cdot\partial U$ and $q\in a(q)\cdot U$ would imply that $q\not\in \partial\left(G^h_0\cdot\left(\begin{smallmatrix}
1\\ 0
\end{smallmatrix}\right)\right)$. But we have by definition of $G^h_0$ that $G^h_0\subset\mathrm{GL}(n+1)$ and, hence, $\left(\begin{smallmatrix}
1\\ 0
\end{smallmatrix}\right)=a(q)^{-1}q=p$, this is a contradiction to $p\in \partial U$. We conclude that the orbit $G^h_0\cdot\left(\begin{smallmatrix}
1\\ 0
\end{smallmatrix}\right)\subset\mathcal{H}$ is open in $\mathcal{H}$. Since $\mathcal{H}\subset\mathbb{R}^{n+1}$ is maximal and being a hyperbolic point of $h$ is an open condition in $\mathbb{R}^{n+1}$ it follows that $\mathcal{H}\cap\partial\mathcal{H}=\emptyset$. This shows that the same also holds for the relative to $\mathcal{H}$ open orbit $G^h_0\cdot\left(\begin{smallmatrix}
1\\ 0
\end{smallmatrix}\right)$, i.e. that $\left(G^h_0\cdot\left(\begin{smallmatrix}
1\\ 0
\end{smallmatrix}\right)\right)\cap\partial\left(G^h_0\cdot\left(\begin{smallmatrix}
1\\ 0
\end{smallmatrix}\right)\right)=\emptyset$ where the boundary of $G^h_0\cdot\left(\begin{smallmatrix}
1\\ 0
\end{smallmatrix}\right)$ is relative to $\mathbb{R}^{n+1}$. This implies that $G^h_0\cdot\left(\begin{smallmatrix}
1\\ 0
\end{smallmatrix}\right)$ is an $n$-dimensional submanifold of $\mathbb{R}^{n+1}$. Furthermore, $\left(G^h_0\cdot\left(\begin{smallmatrix}
1\\ 0
\end{smallmatrix}\right),\left.g_\mathcal{H}\right|_{G^h_0\cdot\left(\begin{smallmatrix}
1\\ 0
\end{smallmatrix}\right)}\right)$ is also by construction a homogeneous Riemannian manifold and, hence, in particular geodesically complete.
This implies that $G^h_0\cdot\left(\begin{smallmatrix}
1\\ 0
\end{smallmatrix}\right)\subset\mathbb{R}^{n+1}$ is closed, which can be seen the following way. Suppose that $G^h_0\cdot\left(\begin{smallmatrix}
1\\ 0
\end{smallmatrix}\right)$ is not closed in $\mathbb{R}^{n+1}$ but geodesically complete with respect to the restriction of $g_\mathcal{H}$ and let $p_0$ be a point in the boundary $\partial \left(G^h_0\cdot\left(\begin{smallmatrix}
1\\ 0
\end{smallmatrix}\right)\right)$. For any other point $p\in G^h_0\cdot\left(\begin{smallmatrix}
1\\ 0
\end{smallmatrix}\right)$ consider a curve $\gamma:[0,1)\to G^h_0\cdot\left(\begin{smallmatrix}
1\\ 0
\end{smallmatrix}\right)$ with $\gamma(0)=p$ and $\lim\limits_{t\to 1,\, t<1}\gamma(t)=p_0$. Since $G^h_0\cdot\left(\begin{smallmatrix}
1\\ 0
\end{smallmatrix}\right)\subset\mathcal{H}\subset\{h=1\}$ and $h:\mathbb{R}^{n+1}\to\mathbb{R}$, we conclude that $1=\lim\limits_{t\to 1,\, t<1}h(\gamma(t))=h\left(\lim\limits_{t\to 1,\, t<1} \gamma(t)\right)=h(p_0)$. Since $g_\mathcal{H}=-\frac{1}{\tau}\partial^2 h|_{T\mathcal{H}\times T\mathcal{H}}$ it in particular follows from the fact that $h(p_0)=1$ and that $h$ is a homogeneous polynomial of homogeneity-degree $\tau$ that $g_\mathcal{H}$ can be smoothly extended to $p_0\in G^h_0\cdot\left(\begin{smallmatrix}
1\\ 0
\end{smallmatrix}
\right)$.
This implies
\begin{equation*}
\lim\limits_{t\to 1,\, t<1}\int\limits_0^t\sqrt{g_\mathcal{H}(\dot\gamma(t),\dot\gamma(t))}dt<\infty.
\end{equation*}
This is a contradiction to the geodesic completeness of $\left(G^h_0\cdot\left(\begin{smallmatrix}
1\\ 0
\end{smallmatrix}
\right),\left. g_\mathcal{H}\right|_{G^h_0\cdot\left(\begin{smallmatrix}
1\\ 0
\end{smallmatrix}
\right)}\right)$ (recall that by the Hopf-Rinow theorem a Riemannian manifold is geodesically complete if and only if all unbounded curves have infinite length).
By assumption, $\mathcal{H}\subset\mathbb{R}^{n+1}$ is maximal, and we have shown that $G^h_0\cdot\left(\begin{smallmatrix}
1\\ 0
\end{smallmatrix}
\right)\subset\mathbb{R}^{n+1}$ is closed. We deduce that $\mathcal{H}=G^h_0\cdot\left(\begin{smallmatrix}
1\\ 0
\end{smallmatrix}
\right)$ and that the action of $G^h_0$ on $\mathcal{H}$ is, in fact, transitive. In particular, $\mathcal{H}$ is a CCGPSR manifold.
\end{proof}
\end{Prop}

\section{Construction of a compact convex generating set of the mo\-du\-li set of CCPSR manifolds}
In order to prove Theorem \ref{thm_Cn}, broadly speaking we need to construct estimates for $P_3$ and eigenvalues of its second derivative and study properties of $\mathrm{dom}(\mathcal{H})$, cf. Definition \ref{def_dom(H)}.
\begin{Lem}\label{lemma_P_3(y)_bounds}
Let $\mathcal{H}$ be a CCPSR manifold in standard form.
Then
\begin{align}\label{eqn_P3_sphere_bdrs}
\forall\, \widehat{z}\in\{z\in\mathbb{R}^n\mid\langle z,z\rangle=1\}:\quad |P_3(\widehat{z})|\leq \frac{2}{3\sqrt{3}}.
\end{align}
\begin{proof}
Consider $f(t):=\beta(t\widehat{z})=1-t^2+t^3P_3(\widehat{z})$, where $\beta:\mathbb{R}^n\to\mathbb{R}$ as in equation \eqref{beta_def}. Since $\mathrm{dom}(\mathcal{H})$ is precompact (Lemma \ref{precomp_cor}), $f$ must have at least one positive and one negative real root. We will determine the range for $P_3(\widehat{z})$ such that this holds. The first and second derivative of $f$ are
\begin{align*}
\dot{f}(t)=-2t+3t^2P_3(\widehat{z}),\quad\ddot{f}(t)=-2+6tP_3(\widehat{z}).
\end{align*}
Hence, $\dot{f}(t)=0$ if and only if $t=0$ or $t=\frac{2}{3P_3(\widehat{z})}$. We obtain $\ddot{f}(0)=-2$ and $\ddot{f}\left(\frac{2}{3P_3(\widehat{z})}\right)=2$. This implies that $f(t)$ has a local maximum at $t=0$ and a local minimum at $t=\frac{2}{3P_3(\widehat{z})}$. If $P_3(\widehat{z})=0$, $f(t)=0$ if and only if $t=\pm 1$, so in this case $f(t)$ has exactly one positive and one negative real root. Now assume $P_3(\widehat{z})>0$. In that case, $\frac{2}{3P_3(\widehat{z})}>0$ and $\lim\limits_{t\to -\infty}f(t)=-\infty$. Since $f(0)=1$, this implies that $f(t)$ has at least one negative real root (one can show that it is the only negative real root by showing that $\dot{f}(t)>0$ for all $t<0$ if $P_3(\widehat{z})>0$). We have seen that $f(t)$ attains its unique local minimum at $t=\frac{2}{3P_3(\widehat{z})}$. Furthermore $f(0)=0$, and $\lim\limits_{t\to\infty}f(t)=\infty$. Hence, $f(t)$ has a positive real root if and only if
\begin{equation*}
f\left(\frac{2}{3P_3(\widehat{z})}\right)\leq 0\quad \Leftrightarrow\quad 1-\frac{4}{27P_3(\widehat{z})^2}\leq 0\quad \Leftrightarrow\quad P_3(\widehat{z})\leq \frac{2}{3\sqrt{3}}.
\end{equation*}
For $P_3(\widehat{z})<0$ we define $\widetilde{f}(t):=1-t^2+t^3(-P_3(\widehat{z}))$. Similarly as for $P_3(\widehat{z})$ we then obtain $-P_3(\widehat{z})\leq\frac{2}{3\sqrt{3}}$.
Summarising, we have shown that $|P_3(\widehat{z})|\leq \frac{2}{3\sqrt{3}}$.
\end{proof}
\end{Lem}
Note that the bounds \eqref{eqn_P3_sphere_bdrs} for $P_3(\widehat{z})$, $\widehat{z}\in\{z\in\mathbb{R}^n\mid\langle z,z\rangle=1\}$, are independent of the CCPSR manifold and of its dimension. We will later show that these bounds are in fact sharp and optimal in all dimensions, see Theorem \ref{thm_convex_compact_PSR_generating_set}. An immediate consequence of the calculations in Lemma \ref{lemma_P_3(y)_bounds} is the following.
\begin{Cor}\label{cor_PSR_closedness_condition_P3}
Let $h:\mathbb{R}^{n+1}\to\mathbb{R}$, $h=x^3-x\langle y,y\rangle + P_3(y)$, be a cubic homogeneous polynomial and let $\mathcal{H}$ denote the connected component of $\{h=1\}\subset\mathbb{R}^{n+1}$ that contains the point $\left(\begin{smallmatrix}
x\\ y
\end{smallmatrix}\right)=\left(\begin{smallmatrix}
1\\ 0
\end{smallmatrix}\right)$. Then the connected component of the set
\begin{equation*}
\{h>0\}\cap \left.\left\{\left(\begin{smallmatrix}
1\\
z
\end{smallmatrix}
\right)\in\mathbb{R}^{n+1}\ \right|\ z\in\mathbb{R}^n\right\}\subset\mathbb{R}^{n+1}
\end{equation*}
which contains the point $\left(\begin{smallmatrix}
x \\ y
\end{smallmatrix}
\right)=\left(\begin{smallmatrix}
1 \\ 0
\end{smallmatrix}
\right)$
is precompact if and only if $\max\limits_{\|z\|=1}|P_3(z)|\leq \frac{2}{3\sqrt{3}}$.
\end{Cor}
Recall that we know from Lemma \ref{precomp_cor} that the connected component of the set $\{h>0\}\cap \left.\left\{\left(\begin{smallmatrix}
1\\
z
\end{smallmatrix}
\right)\in\mathbb{R}^{n+1}\ \right|\ z\in\mathbb{R}^n\right\}$ that contains the point $\left(\begin{smallmatrix}
x \\ y
\end{smallmatrix}
\right)=\left(\begin{smallmatrix}
1 \\ 0
\end{smallmatrix}
\right)$ being pre-compact is a necessary condition for the connected component of $\{h=1\}$ that contains $\left(\begin{smallmatrix}
x \\ y
\end{smallmatrix}
\right)=\left(\begin{smallmatrix}
1 \\ 0
\end{smallmatrix}
\right)$ to be a CCPSR manifold. Also note that if the connected component $\mathcal{H}\subset\{h=1\}$ that contains the point $\left(\begin{smallmatrix}
x \\ y
\end{smallmatrix}
\right)=\left(\begin{smallmatrix}
1 \\ 0
\end{smallmatrix}
\right)$ is a CCPSR manifold, then the connected component of the set $\{h>0\}\cap \left.\left\{\left(\begin{smallmatrix}
1\\
z
\end{smallmatrix}
\right)\in\mathbb{R}^{n+1}\ \right|\ z\in\mathbb{R}^n\right\}$ that contains the point $\left(\begin{smallmatrix}
x \\ y
\end{smallmatrix}
\right)=\left(\begin{smallmatrix}
1 \\ 0
\end{smallmatrix}
\right)$, the set $\left(\mathbb{R}_{>0}\cdot \mathcal{H}\right)\cap \left.\left\{\left(\begin{smallmatrix}
1\\
z
\end{smallmatrix}
\right)\in\mathbb{R}^{n+1}\ \right|\ z\in\mathbb{R}^n\right\}$, and $\{1\}\times\mathrm{dom}(\mathcal{H})$ coincide. One could ask if we can find similar bounds for CCGPSR manifolds of homogeneity-degree $\tau\geq 4$. This is in general not true, see \cite[Lem.\,7.9]{Li} for quartic CCGPSR manifolds, that is hyperbolic centro-affine hypersurfaces defined by quartic homogeneous polynomials.
Lemma \ref{lemma_P_3(y)_bounds} also means that we have determined positive and negative bounds for $P_3(\widehat{z})$, $\widehat{z}\in\{z\in\mathbb{R}^n\mid\langle z,z\rangle=1\}$, that guaranty that the corresponding hypersurface which is the connected component of $\{h=1\}$ containing the point $\left(\begin{smallmatrix}
1\\
0
\end{smallmatrix}
\right)\in\mathbb{R}^{n+1}$ is closed. However, it does at this point not give us information about hyperbolicity when we are studying some specific connected PSR manifold and want to know whether it is a CCPSR manifold or not. It will later turn out that this condition also shows hyperbolicity of all points contained in the connected component of $\{h=1\}$ that contains the point $\left(\begin{smallmatrix}
1\\ 0
\end{smallmatrix}\right)\in\mathbb{R}^{n+1}$, see Theorem \ref{thm_convex_compact_PSR_generating_set}.

Next, we will use Lemma \ref{lemma_P_3(y)_bounds} to determine upper and lower positive bounds for the norm of points in the boundary of $\mathrm{dom}(\mathcal{H})\subset\mathbb{R}^n$, that is $\partial\mathrm{dom}(\mathcal{H})$, corresponding to a CCPSR manifold $\mathcal{H}$.
\begin{Lem}\label{lemma_roots_estimates_PSR}
In the setting of Lemma \ref{lemma_P_3(y)_bounds}, assume without loss of generality that $P_3(\widehat{z})\geq 0$. Let $\mathcal{N}_{P_3(\widehat{z})}$ be the biggest negative real root of $f(t)$ and $\mathcal{P}_{P_3(\widehat{z})}$ be the smallest positive real root of $f(t)$, where $f(t)$ is associated to a CCPSR manifold $\mathcal{H}$ as in the previous lemma and $|P_3(\widehat{z})|\leq \frac{2}{3\sqrt{3}}$. Then 
\begin{align*}
-1\leq\ \mathcal{N}_{P_3(\widehat{z})}\leq -\frac{\sqrt{3}}{2},\quad
1\leq\ \mathcal{P}_{P_3(\widehat{z})}\leq \sqrt{3}.
\end{align*}
\begin{proof}
Let $0\leq A<B\leq \frac{2}{3\sqrt{3}}$, and define
\begin{align*}
f_A(t) :=1-t^2+t^3A,\quad
f_B(t) :=1-t^2+t^3B.
\end{align*}
$f_A(t)$ and $f_B(t)$ have a unique negative real root $\mathcal{N}_A$ and $\mathcal{N}_B$, respectively. Furthermore, $\mathcal{N}_A<\mathcal{N}_B$. To see this we calculate $\dot{f}_A(t) =-2t+3t^2A$ and $\dot{f}_B(t) =-2t+3t^2B$, from which it is immediate that $\dot{f}_A(t)>0,\ \dot{f}_B(t)>0$ for all $t<0$.
Since $\lim\limits_{t\to-\infty}f_A(t)=-\infty$, $\lim\limits_{t\to-\infty}f_B(t)=-\infty$, and $f_A(0)=f_B(0)=1$ this implies that $\mathcal{N}_A$ and $\mathcal{N}_B$ exist and are the unique negative real roots of $f_A(t)$, respectively $f_B(t)$. We further obtain
\begin{align*}
f_B(\mathcal{N}_A) =1-\mathcal{N}_A^2+\mathcal{N}_A^3B =f_A(\mathcal{N}_A)+(B-A)\mathcal{N}_A^3 =(B-A)\mathcal{N}_A^3<0.
\end{align*}
Using $\dot{f}_B|_{t<0}>0$ this shows that
\begin{equation}\label{eqn_negative_root_PSR_estimate}
\mathcal{N}_B>\mathcal{N}_A.
\end{equation}
We apply this result to $\mathcal{N}_{P_3(\widehat{z})}$ and obtain $-1=\mathcal{N}_0\leq \mathcal{N}_{P_3(\widehat{z})}\leq \mathcal{N}_{\frac{2}{3\sqrt{3}}}=-\frac{\sqrt{3}}{2}$.
The value of $\mathcal{N}_{\frac{2}{3\sqrt{3}}}$ can easily be found by checking that $f_{\frac{2}{3\sqrt{3}}}(t)=\frac{2}{3\sqrt{3}}\left(t+\frac{\sqrt{3}}{2}\right)\left(t-\sqrt{3}\right)^2$.

Now let $\mathcal{P}_A$ and $\mathcal{P}_B$ be the smallest positive root of $f_A(t)$ and $f_B(t)$, respectively. Then $\mathcal{P}_A<\mathcal{P}_B$. To see this, first note that the existence of $\mathcal{P}_A$ and $\mathcal{P}_B$ is ensured by the estimate \eqref{eqn_P3_sphere_bdrs} in Lemma \ref{lemma_P_3(y)_bounds}. We obtain
\begin{align*}
f_A(\mathcal{P}_B) =1-\mathcal{P}_B^2+\mathcal{P}_B^3A =f_B(\mathcal{P}_B)+(A-B)\mathcal{P}_B^3 =(A-B)\mathcal{P}_B^3<0.
\end{align*}
Since $f_A(0)=1$ this shows that $f_A(t)$ has a positive real root that is smaller than $\mathcal{P}_B$, and in particular that
\begin{equation}\label{eqn_positive_root_PSR_estimate}
\mathcal{P}_A<\mathcal{P}_B.
\end{equation}
Again, we apply this result to $\mathcal{P}_{P_3(\widehat{z})}$ and obtain $1=\mathcal{P}_0\leq \mathcal{P}_{P_3(\widehat{z})}\leq \mathcal{P}_{\frac{2}{3\sqrt{3}}}=\sqrt{3}$.
\end{proof}
\end{Lem}
Lemma \ref{lemma_roots_estimates_PSR} implies the following result for the Euclidean norm of points in $\partial\mathrm{dom}(\mathcal{H})$.
\begin{Cor}\label{corollary_PSR_boundary_domain_estimates}
For a CCPSR manifold $\mathcal{H}$ in standard form
and corresponding $\mathrm{dom}(\mathcal{H})$ as in Definition \ref{def_dom(H)}, the following holds true:
\begin{equation}\label{dom(H)_PSR_boundaries_eqn}
 \forall\,\overline{z}\in\partial\mathrm{dom}(\mathcal{H}):\quad \frac{\sqrt{3}}{2}\leq\|\overline{z}\|\leq \sqrt{3},
\end{equation}
where $\|\cdot\|$ denotes the norm with respect to the standard Euclidean scalar product $\langle \cdot,\cdot\rangle$ on $\mathbb{R}^n$ in the $y$-coordinates from equation \eqref{standard_form_h}.
\end{Cor}
Hence, with the notation $B_r(0)=\left\{z\in\mathbb{R}^n\mid \langle z,z\rangle < r^2\right\}$
for $r>0$, we have the inclusions $B_{\frac{\sqrt{3}}{2}}(0)\subset\mathrm{dom}(\mathcal{H})\subset B_{\sqrt{3}}(0)$ for all CCPSR manifolds $\mathcal{H}$. In particular this is also independent of the point chosen in the process (see Proposition \ref{prop_std_form_h}) of obtaining $h$ in the form \eqref{standard_form_h} for any given CCPSR manifold $\mathcal{H}\subset\{h=1\}$. Note however that the inclusion $B_{\frac{\sqrt{3}}{2}}(0)\subset\mathrm{dom}(\mathcal{H})$ might not be compact in the sense that $\partial B_{\frac{\sqrt{3}}{2}}(0)\cap \partial\mathrm{dom}(\mathcal{H})$ might not be empty. If we choose any $0<R<\frac{\sqrt{3}}{2}$, $B_R(0)$ will always be compactly embedded via the inclusion in $\mathrm{dom}(\mathcal{H})$ since the inclusion $B_R(0)\subset B_{\frac{\sqrt{3}}{2}}(0)$ is a compact embedding.

Another consequence of Lemma \ref{lemma_roots_estimates_PSR} is the following characterisation of CCPSR manifolds that are singular at infinity, cf. Definition \ref{def_singular_GPSR}.

\begin{Lem}\label{lemma_singular_PSR_max_P_3}
Let $\mathcal{H}$ be CCPSR manifold in standard form.
Then $\mathcal{H}$ is singular at infinity in the sense of Definition \ref{def_singular_GPSR} if and only if $\max\limits_{\|z\|=1}|P_3(z)|=\frac{2}{3\sqrt{3}}$.
\begin{proof}
First note that with our assumptions for $\mathcal{H}$ and $h$, $\partial(\mathbb{R}_{>0}\cdot\mathcal{H})\setminus\{0\}=\mathbb{R}_{>0}\cdot\left(\{1\}\times\partial\mathrm{dom}(\mathcal{H})\right)$. Since $dh_p$ is homogeneous of degree $2$ in $p$, it thus suffices to show that there exists a $\overline{z}\in\partial\mathrm{dom}(\mathcal{H})$, such that $dh_{\left(\begin{smallmatrix}
1\\
\overline{z}
\end{smallmatrix}
\right)}=0$ if and only if $\max\limits_{\|z\|=1}|P_3(z)|=\frac{2}{3\sqrt{3}}$. In Lemma \ref{lemma_boundary_conditions_alpha_beta} we have shown that for $\overline{z}\in\partial\mathrm{dom}(\mathcal{H})$, $dh_{\left(\begin{smallmatrix}
1\\
\overline{z}
\end{smallmatrix}
\right)}=0$ is equivalent to $\frac{\partial h}{\partial x}\left(\left(\begin{smallmatrix}
1\\
\overline{z}
\end{smallmatrix}
\right)\right)=\alpha(\overline{z})=0$, which is by the Euler identity for homogeneous functions equivalent to $d\beta_{\overline{z}}(\overline{z})=0$. Hence, $\mathcal{H}$ is singular at infinity if and only if there exists a point $\widehat{z}\in \{\|z\|=1\}$, such that the $1$-dimensional CCPSR manifold $\mathcal{H}^{\widehat{z}}$ defined by restricting $h$ to the $2$-dimensional linear subspace
\begin{equation*}
E=\mathrm{span}\left\{\left(\begin{matrix}
1\\
0
\end{matrix}\right),\left(\begin{matrix}
0\\
\widehat{z}
\end{matrix}\right)\right\}\subset\mathbb{R}^{n+1}
\end{equation*}
is singular at infinity.
More precisely, $\mathcal{H}^{\widehat{z}}$ is the connected component of $\left\{h^{\widehat{z}}:=x^3-xt^2+t^3P_3(\widehat{z})=1\right\}$ that contains the point $\left(\begin{smallmatrix}
x\\
t
\end{smallmatrix}\right)=\left(\begin{smallmatrix}
1\\
0
\end{smallmatrix}\right)$, and we have $\{1\}\times\mathrm{dom}\left(\mathcal{H}^{\widehat{z}}\right)=E\cap\left(\{1\}\times\mathrm{dom}(\mathcal{H})\right)$. The corresponding function $\beta^{\widehat{z}}$ as in \eqref{beta_def} for $h^{\widehat{z}}$ is given by
\begin{equation*}
\beta^{\widehat{z}}(t)=\beta(t\widehat{z})=1-t^2+t^3P_3(\widehat{z}).
\end{equation*}
Let $t_+$ and $t_-$ denote the smallest positive root and the biggest negative root of $\beta^{\widehat{z}}(t)$, respectively. Then $\partial\mathrm{dom}\left(\mathcal{H}^{\widehat{z}}\right)=\{t_+\widehat{z},t_-\widehat{z}\}$. We have shown in Lemma \ref{lemma_roots_estimates_PSR} (with the notation $\beta^{\widehat{z}}(t)=f_{|P_3(\widehat{z})|}(t)$) that $\partial_t \beta^{\widehat{z}}(t_+)=0$ or $\partial_t \beta^{\widehat{z}}(t_-)=0$ if $|P_3(\widehat{z})|=\frac{2}{3\sqrt{3}}$. It remains to show that $|P_3(\widehat{z})|<\frac{2}{3\sqrt{3}}$ implies that $\partial_t\beta^{\widehat{z}}(t)$ does not vanish at neither $t_+$ nor $t_-$. To do that, assume without loss of generality $P_3(\widehat{z})\geq 0$. For $P_3(\widehat{z})<0$ we can simply use the reflection $t\to -t$ and consider $\beta^{\widehat{z}}(-t)$. For $P_3(\widehat{z})=0$ it is easy to check that $\partial_t\beta^{\widehat{z}}(t_\pm)=\mp 2$. Now assume $P_3(\widehat{z})>0$. We have
\begin{equation*}
\partial_t\beta^{\widehat{z}}(t)=t(-2+3t P_3(\widehat{z}) ),
\end{equation*}
hence $\partial_t\beta^{\widehat{z}}(t_-)>0$ is always true and $\partial_t\beta^{\widehat{z}}(t_+)=0$ if and only if $t_+=\frac{2}{3 P_3\left(\widehat{z}\right) }$. One quickly finds that $\beta^{\widehat{z}}(t_+)=0$ and $P_3(\widehat{z})>0$ if and only if $P_3(\widehat{z})=\frac{2}{3\sqrt{3}}$. This shows that $\partial_t\beta^{\widehat{z}}$ vanishes at a point $\overline{z}\in\mathrm{dom}\left(\mathcal{H}^{\widehat{z}}\right)=\{t_+\widehat{z},t_-\widehat{z}\}$ (which is equivalent to $\mathcal{H}^{\widehat{z}}$ being singular at infinity) if and only if $|P_3(\widehat{z})|=\frac{2}{3\sqrt{3}}$. Summarising, we have shown that there exists a point $\overline{z}\in\partial\mathrm{dom}(\mathcal{H})$, such that $d\beta_{\overline{z}}(\overline{z})=0$ if and only if there exists a point $\overline{z}\in\partial\mathrm{dom}(\mathcal{H})$, such that $\left|P_3\left(\frac{\overline{z}}{\|\overline{z}\|}\right)\right|=\frac{2}{3\sqrt{3}}$. In Lemma \ref{lemma_roots_estimates_PSR} we have shown that this is precisely the maximal possible value for $|P_3(z)|$ on $\{\|z\|=1\}$ that does not exclude the property of $\mathcal{H}$ being closed in $\mathbb{R}^{n+1}$. We conclude that $\max\limits_{\|z\|=1}|P_3(z)|=\frac{2}{3\sqrt{3}}$ if and only if $\mathcal{H}$ is singular at infinity.
\end{proof}
\end{Lem}

Note that the set of CCPSR manifolds that are singular at infinity and of dimension $n\geq 1$ is not empty for all $n\geq1$. This is one of the consequences of Theorem \ref{thm_Cn}, but we can also use Proposition \ref{prop_hom_spaces_deltaPk_condition} and the above Lemma \ref{lemma_singular_PSR_max_P_3} to prove both the latter statement and a property of homogeneous CCPSR manifolds:

\begin{Prop}
\label{prop_homimpliessingatinf}
Homogeneous CCPSR manifolds are singular at infinity.
\begin{proof}
Let $\mathcal{H}$ be a homogeneous CCPSR manifold and without loss of generality assume that $\mathcal{H}$ is in standard form. In Proposition \ref{prop_hom_spaces_deltaPk_condition} we have seen that $\mathcal{H}$ being homogeneous is equivalent to the existence of $dB_0\in\mathrm{Lin}(\mathbb{R}^n,\mathfrak{so}(n))$ as in equation \eqref{dB0_eqn}, such that $\delta P_3(y)\equiv 0$, cf. equation \eqref{delta_P_3_cubics}. Applying $\delta P_3(y)$ to the position vector field in $\mathbb{R}^n$, we obtain
	\begin{equation}\label{eqn_hom_singatinf_delP3_posVF}
		-\frac{2}{3}\langle y,y\rangle^2 + 3P_3\left(y,y,dB_0(y)y+\frac{3}{2}P_3(y,\cdot,y)^T\right)=0
	\end{equation}
for all $y\in\mathbb{R}^n$, where $dB_0(y)y=\sum\limits_{k=1}^{n(n-1)/2} a_k y\langle \ell_k,y\rangle$. Let now $y\in\{\|y\|=1\}$ be a local positive maximum of $P_3|_{\{\|y\|=1\}}$. This means that there is $r>0$, such that $dP_3|_y=r\langle y,dy\rangle$. Since $dB_0$ has image in $\mathfrak{so}(n)$, this implies using $\|y\|=1$, \eqref{eqn_hom_singatinf_delP3_posVF}, and $3P_3(y,y,dy)=dP_3|_y$
	\begin{equation*}
		-\frac{2}{3}+\frac{1}{2}r^2=0.
	\end{equation*}
The above equation and $y$ being a local positive maximum of $P_3|_{\{\|y\|=1\}}$ show that $r=\frac{2}{\sqrt{3}}$. Using $dP_3|_y=\frac{2}{\sqrt{3}}\langle y,dy\rangle$ and the Euler identity for homogeneous functions we find $P_3(y)=\frac{2}{3\sqrt{3}}$. This is by Lemma \ref{lemma_singular_PSR_max_P_3} equivalent to $\mathcal{H}$ being singular at infinity.
\end{proof}
\end{Prop}

Note that, in theory, one could have also used the classification of homogeneous CCPSR manifolds given in \cite{DV} for $\dim(\mathcal{H})\geq 3$ and \cite{CHM,CDL} for $\dim(\mathcal{H})=1$ and $\dim(\mathcal{H})=2$, respectively, to prove Proposition \ref{prop_homimpliessingatinf}. One would then have to determine a standard form for the corresponding polynomials and could then use Lemma \ref{lemma_singular_PSR_max_P_3}. This would, however, be most likely much more time-consuming.

\begin{rem}
Note that the proof of Proposition \ref{prop_homimpliessingatinf} shows that \textit{every} local positive maximum $y\in\{\|y\|=1\}$ of $P_3|_{\{\|y\|=1\}}$ fulfils $P_3(y)=\frac{2}{3\sqrt{3}}$ (and is thereby also a global maximum). It is an interesting open question whether this is enough to completely classify homogeneous CCPSR manifolds in the sense that this statement is equivalent to a CCPSR manifold being homogeneous.
\end{rem}

We will now determine an estimate for the bilinear form $P_3(z,dz,dz)$ for all $z\in\mathrm{dom}(\mathcal{H})$. It will use the hyperbolicity property of the CCPSR manifold $\mathcal{H}$, which we first need to reformulate.
\begin{Lem}\label{lemma_general_PSR_P_3(z,.,.)_estimate}
Let $h:\mathbb{R}^{n+1}\to\mathbb{R}$ be a cubic homogeneous polynomial of the form \eqref{standard_form_h}, that is $h=x^3-x\langle y,y\rangle+P_3(y)$, and let $\mathcal{H}\subset\{h=1\}$ be the connected component of the level set $\{h=1\}\subset\mathbb{R}^{n+1}$ that contains the point $\left(\begin{smallmatrix}
x\\y
\end{smallmatrix}
\right)=\left(\begin{smallmatrix}
1\\ 0
\end{smallmatrix}
\right)$.
Then $\mathcal{H}$ is a CCPSR manifold if and only if 
\begin{equation}\label{hyp_det_condition}
 \forall \left(\begin{smallmatrix}
 1\\ z
 \end{smallmatrix}\right) \in (\mathbb{R}_{>0}\cdot\mathcal{H})\cap \left.\left\{\left(\begin{smallmatrix}
1\\
z
\end{smallmatrix}
\right)\in\mathbb{R}^{n+1}\ \right|\ z\in\mathbb{R}^n\right\}: \quad3\langle dz,dz\rangle -9 P_3(z,dz,dz)+\langle z,dz\rangle^2 >0.
\end{equation}
\begin{proof}
Assumption that $\mathcal{H}$ is a CCPSR manifold. Then $\mathcal{H}$ fulfils the assumptions of this lemma and $\left(\mathbb{R}_{>0}\cdot\mathcal{H}\right)\cap \left.\left\{\left(\begin{smallmatrix}
1\\
z
\end{smallmatrix}
\right)\in\mathbb{R}^{n+1}\ \right|\ z\in\mathbb{R}^n\right\}$ coincides with $\mathrm{dom}(\mathcal{H})$, cf. Definition \ref{def_dom(H)}. We will show that condition \eqref{hyp_det_condition} follows from the hyperbolicity of each point in $\mathcal{H}$. For each $p\in\mathcal{H}\subset\mathbb{R}^{n+1}$, the tangent space $T_p\mathcal{H}$ viewed as a the hyperplane $\ker(dh_p)\subset\mathbb{R}^{n+1}$ and the line $\mathbb{R}p\subset \mathbb{R}^{n+1}$ are orthogonal with respect to the Lorenzian inner product $-\partial^2h_p$. Recall that $-\partial^2h_p$ being Lorenzian precisely means that $p$ is a hyperbolic point, see Definition \ref{hyperbolicpointdef}. Since $-\partial^2 h_p$ is homogeneous of degree $1$ in $p$, it follows that the property that $\mathcal{H}$ consists only of hyperbolic points is equivalent to the statement that $-\partial^2 h_{\left(\begin{smallmatrix}
1\\
z
\end{smallmatrix}
\right)}$ is Lorenzian for all $z\in\mathrm{dom}(\mathcal{H})$.
Since $-\partial^2h_{\left(\begin{smallmatrix}
1\\
0
\end{smallmatrix}
\right)}$ is always Lorenzian if $h$ is of the form \eqref{standard_form_h}, $-\partial^2h_{\left(\begin{smallmatrix}
1\\
z
\end{smallmatrix}
\right)}$ being Lorenzian on $\mathrm{dom}(\mathcal{H})$ is equivalent to $\det\left(-\partial^2h_{\left(\begin{smallmatrix}
1\\
z
\end{smallmatrix}
\right)}\right)<0$ for all $z\in\mathrm{dom}(\mathcal{H})$. Consider
\begin{align*}
\det\left(-\partial^2h_{\left(\begin{smallmatrix}
1\\
z
\end{smallmatrix}
\right)}\right) &=\det\left(
	\begin{tabular}{c|c}
	$-6$ & $2z^T$ \\ \hline
	$\overset{}{2z}$ & $2\mathbbm{1}-6P_3(z,\cdot,\cdot)$
	\end{tabular}
	\right)\notag\\
	&=\det\left(
	\begin{tabular}{c|c}
	$-6$ & $2z^T$ \\ \hline
	$0$ & $\overset{}{2\mathbbm{1}-6P_3(z,\cdot,\cdot)+\frac{2}{3}z\otimes \langle z,\cdot\rangle}$
	\end{tabular}
	\right)\notag\\
	&=-\frac{2^{n+1}}{3^{n-1}}\det\left(3\mathbbm{1}-9P_3(z,\cdot,\cdot)+z\otimes \langle z,\cdot\rangle\right).
\end{align*}
Since $\left.\left(3\mathbbm{1}-9P_3(z,\cdot,\cdot)+z\otimes \langle z,\cdot\rangle\right)\right|_{z=0}=3\mathbbm{1}$, it follows that $\det\left(-\partial^2h_{\left(\begin{smallmatrix}
1\\
z
\end{smallmatrix}
\right)}\right)<0$ for all $z\in\mathrm{dom}(\mathcal{H})$ is equivalent to $3\langle dz,dz\rangle-9P_3(z,dz,dz)+\langle z,dz\rangle^2>0$ for all $z\in\mathrm{dom}(\mathcal{H})$.

For the other direction, the conditions that $\mathcal{H}$ is a connected component of $\{h=1\}$ implies that it is closed as a subset of $\mathbb{R}^{n+1}$. Furthermore, $\mathcal{H}$ is a hypersurface since $dh$ does not vanish along $\mathcal{H}$ by the Euler identity for homogeneous functions. With the same argument as before for the homogeneity of $-\partial^2h_p$ in $p$ and the same calculations as above, it follows that $\mathcal{H}$ consists only of hyperbolic points. $\mathcal{H}$ is thus a connected and also closed PSR manifold, and the set $\mathrm{pr}_{\mathbb{R}^n}\left(\left(\mathbb{R}_{>0}\cdot\mathcal{H}\right)\cap \left.\left\{\left(\begin{smallmatrix}
1\\
z
\end{smallmatrix}
\right)\in\mathbb{R}^{n+1}\ \right|\ z\in\mathbb{R}^n\right\}\right)$ and $\mathrm{dom}(\mathcal{H})$ coincide.
\end{proof}
\end{Lem}
We will use the results from Corollary \ref{corollary_PSR_boundary_domain_estimates} and Lemma \ref{lemma_general_PSR_P_3(z,.,.)_estimate} to find upper and lower bounds of the eigenvalues of $P_3(z,dz,dz)$ (when viewed as a symmetric matrix) for $z\in\mathrm{dom}(\mathcal{H})$ that are valid for all CCPSR manifolds $\mathcal{H}$ (and thus also for non-connected closed PSR manifolds).
\begin{Prop}
\label{prop_P_3(z,.,.)_boundaries_on_del_dom(H)}
Let $\mathcal{H}$ be a CCPSR manifold in standard form.
Then
\begin{equation}\label{P_3(z,.,.)_lower_upper_boundary_PSR}
\forall z\in\mathrm{dom}(\mathcal{H}):\quad -\frac{5}{6}\langle dz,dz\rangle<P_3(z,dz,dz)<\frac{2}{3}\langle dz,dz\rangle.
\end{equation}
This is equivalent to the statement that for all $z\in\mathrm{dom}(\mathcal{H})$, the eigenvalues $\lambda\in\mathbb{R}$ of the representation matrix of the symmetric bilinear form $P_3(z,dz,dz)$ induced by the $z$-coordinates fulfil $-\frac{5}{6}<\lambda<\frac{2}{3}$. Furthermore, the upper bound in \eqref{P_3(z,.,.)_lower_upper_boundary_PSR} is sharp in the sense that for all $n\geq 1$ there exists a CCPSR manifold $\mathcal{H}$ and a point $\check{z}\in\partial\mathrm{dom}(\mathcal{H})$, such that the representation matrix of $P_3(\check{z},dz,dz)$ has one eigenvalue $\lambda=\frac{2}{3}$.

\begin{proof}
We start with the upper bound in \eqref{P_3(z,.,.)_lower_upper_boundary_PSR}. Equation \eqref{hyp_det_condition} in Lemma \ref{lemma_general_PSR_P_3(z,.,.)_estimate} and equation \eqref{dom(H)_PSR_boundaries_eqn} in Corollary \ref{corollary_PSR_boundary_domain_estimates} imply for all $z\in\mathrm{dom}(\mathcal{H})$
\begin{align}
P_3(z,dz,dz)<\frac{3\langle dz,dz\rangle +\langle z,dz\rangle^2}{9}
	\leq \frac{3\langle dz,dz\rangle +\langle z,z\rangle\langle dz,dz\rangle}{9}
	< \frac{2}{3}\langle dz,dz\rangle.\label{sharper_P_3(z,.,.)_upper_boundary_PSR}
\end{align}

Obtaining the alleged lower bound in equation \eqref{P_3(z,.,.)_lower_upper_boundary_PSR} for $P_3(z,dz,dz)$ needs more work. An other, but worse, lower bound can be obtained the following way. For all $\check{z}\in\overline{\mathrm{dom}(\mathcal{H})}$ with $\|\check{z}\|=\frac{\sqrt{3}}{2}$ (recall that $\overline{B_{\frac{\sqrt{3}}{2}}(0)}\subset\overline{\mathrm{dom}(\mathcal{H})}$ is always true, see Corollary \ref{corollary_PSR_boundary_domain_estimates}), the biggest positive eigenvalue of the representation matrix of $P_3(\check{z},dz,dz)$ is bound from above by $\frac{2}{3}$. Using that $P_3(z,dz,dz)$ is linear in $z$, we obtain that the smallest eigenvalue of the representation matrix of $P_3(-2\check{z},dz,dz)$ is bounded from below by $-\frac{4}{3}$. Since $\check{z}\in\partial B_{\frac{\sqrt{3}}{2}}(0)$ was arbitrary, we obtain for all $\widetilde{z}\in\partial B_{\sqrt{3}}(0)$ the estimate $P_3(\widetilde{z},dz,dz)\geq-\frac{4}{3}\langle dz,dz\rangle$. Since for all CCPSR manifolds with the assumptions of this lemma $\mathrm{dom}(\mathcal{H})\subset B_{\sqrt{3}}(0)$, we can use the linearity of $P_3(z,dz,dz)$ in $z$ again to conclude that for all $z\in\mathrm{dom}(\mathcal{H})$ we have the estimate $P_3(z,dz,dz)>-\frac{4}{3}\langle dz,dz\rangle$. This bound is worse than $-\frac{5}{6}\langle dz,dz\rangle$, which we will derive now.

The estimate \eqref{sharper_P_3(z,.,.)_upper_boundary_PSR} shows that for all $\check{z}\in\partial\mathrm{dom}(\mathcal{H})$, every positive eigenvalue $\lambda_+$ of the representation matrix of $P_3(\check{z},dz,dz)$ fulfils
\begin{equation}
\lambda_+\leq \frac{3+\|\check{z}\|^2}{9}.\label{upper_bound_PSR_EV_del_dom(H)}
\end{equation}
Fix $\check{z}\in\partial\mathrm{dom}(\mathcal{H})\subset\mathbb{R}^2$ and let $\lambda_-$ be a negative eigenvalue of the representation matrix of $P_3(\check{z},dz,dz)$. The linearity of $P_3(z,dz,dz)$ in $z$ implies that $-\lambda_-$ is a positive eigenvalue of the representation matrix of $P_3(-\check{z},dz,dz)$. However, $-\check{z}$ might not be an element of $\overline{\mathrm{dom}(\mathcal{H})}$. In fact, $-\check{z}\in \overline{\mathrm{dom}(\mathcal{H})}$ if and only if $\|\check{z}\|\leq 1$, which holds if and only if $P_3\left(\frac{\check{z}}{\|\check{z}\|}\right)\in \left[-\frac{2}{3\sqrt{3}},0\right]$ (cf. Lemma \ref{lemma_P_3(y)_bounds} and see Figure \ref{fig_2dim_domH_example}).
\begin{figure}[h]%
\centering%
\includegraphics[scale=0.2]{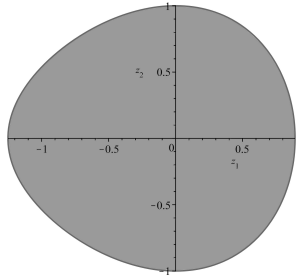}%
\caption{The set $\overline{\mathrm{dom}(\mathcal{H})}\subset\mathbb{R}^2$ corresponding to $P_3\left((z_1,z_2)^T\right)=-\frac{1}{2\sqrt{3}}z_1^3$. Observe for example that for all $\check{z}\in\overline{\mathrm{dom}(\mathcal{H})}$ of the form $\check{z}=(z_1,0)^T$, $z_1>0$, we have $P_3\left(\frac{\check{z}}{\|\check{z}\|}\right)\in \left[-\frac{2}{3\sqrt{3}},0\right]$ and one can check that $-\check{z}\in\overline{\mathrm{dom}(\mathcal{H})}$.}\label{fig_2dim_domH_example}%
\end{figure}

For such a given $\check{z}\in\partial\mathrm{dom}(\mathcal{H})$ we want to find $\check{t}>0$, such that $\check{t}(-\check{z})\in\partial\mathrm{dom}(\mathcal{H})$. When we have determined said $\check{t}$, the linearity of $P_3(z,dz,dz)$ in $z$ implies that $\check{t}(-\lambda_-)$ is a positive eigenvalue of the representation matrix of $P_3(\check{t}(-\check{z}),dz,dz)$. Using the upper bound \eqref{upper_bound_PSR_EV_del_dom(H)}, we can thus estimate
\begin{align}
\check{t}(-\lambda_-)\leq \frac{3+\check{t}^2\|\check{z}\|^2}{9}\quad
\Leftrightarrow\quad -\lambda_- \leq \frac{3+\check{t}^2\|\check{z}\|^2}{9\check{t}}=:F(\|\check{z}\|).\label{eqn_F_for_neg_EVs_PSR}
\end{align}
Our asserted lower bound $-\frac{5}{6}$ for $\lambda_-$ is now obtained via showing that the function $F:\left[\frac{\sqrt{3}}{2},\sqrt{3}\right]\to\mathbb{R}_{>0}$ defined in \eqref{eqn_F_for_neg_EVs_PSR} is continuous and by determining its maximal value, where we recall that the elements in the closed interval $\left[\frac{\sqrt{3}}{2},\sqrt{3}\right]$ are precisely all possible values for $\|\check{z}\|$ when considering all possible $n$-dimensional CCPSR manifolds $\mathcal{H}$ (cf. Corollary \ref{corollary_PSR_boundary_domain_estimates}). To find a closed formula for $\check{t}$ depending on $\|\check{z}\|$, consider the function $f(t)=\beta\left(t\frac{\check{z}}{\|\check{z}\|}\right)=1-t^2+P_3\left(\frac{\check{z}}{\|\check{z}\|}\right)t^3$ (compare with equation \eqref{beta_def}) and assume that $f(\|\check{z}\|)=0$. By assumption, $\mathcal{H}$ is a CCPSR manifold, implying that $\mathrm{dom}(\mathcal{H})\subset\mathbb{R}^n$ is precompact and, hence, $f(t)$ must have at least one more negative real root in addition to its root $t=\|\check{z}\|>0$. Hence, $(t-\|\check{z}\|)\mid f(t)$ and we obtain with $a,b\in\mathbb{R}$
\begin{align*}
f(t)&=(t-\|\check{z}\|)\left(\frac{-1}{\|\check{z}\|}+at+bt^2\right)\\
	&=1+\left(-a\|\check{z}\|-\frac{1}{\|\check{z}\|}\right)t+\left(a-b\|\check{z}\|\right)t^2+bt^3.
\end{align*}
This implies that $a=\frac{-1}{\|\check{z}\|^2}$ and $b=\frac{1}{\|\check{z}\|}-\frac{1}{\|\check{z}\|^3}$. Note that this determines $P_3\left(\frac{\check{z}}{\|\check{z}\|}\right)$ depending on $\|\check{z}\|$, and as we would expect $P_3\left(\frac{\check{z}}{\|\check{z}\|}\right)=\frac{2}{3\sqrt{3}}$ if $\|\check{z}\|=\sqrt{3}$, and $P_3\left(\frac{\check{z}}{\|\check{z}\|}\right)=-\frac{2}{3\sqrt{3}}$ if $\|\check{z}\|=\frac{\sqrt{3}}{2}$ (see Lemma \ref{lemma_P_3(y)_bounds} and Corollary \ref{corollary_PSR_boundary_domain_estimates}). We define
\begin{equation*}
\widetilde{f}(t):=\frac{f(t)}{t-\|\check{z}\|}=-\frac{1}{\|\check{z}\|}-\frac{1}{\|\check{z}\|}t+\left(\frac{1}{\|\check{z}\|}-\frac{1}{\|\check{z}\|^3}\right)t^2.
\end{equation*}
In order to determine $\check{t}$ in dependence of $\|\check{z}\|$ we need to find the roots of $\widetilde{f}(t)$, for (at least) one of the roots coincides with $\check{t}(-\|\check{z}\|)$. We will differentiate between the three cases $\|\check{z}\|=1$, $\|\check{z}\|\in\left(1,\sqrt{3}\right]$, and $\|\check{z}\|\in\left[\frac{\sqrt{3}}{2},1\right)$. We will also use these results to show that $F$ is continuous.
\paragraph*{Case 1: $\|\check{z}\|=1$.}\hspace{1mm}\\
In that case $f(t)=1-t^2$, so the roots of $f(t)$ are $t_\pm=\pm 1$ and the root of $\widetilde{f}(t)=-1-t$ is $t=-1$. Hence, $\check{t}=1$ and \eqref{eqn_F_for_neg_EVs_PSR} thus yields the estimate $-\lambda_-\leq \frac{4}{9}=F(1)$.
\paragraph*{Case 2: $\|\check{z}\|\in \left[\frac{\sqrt{3}}{2},\sqrt{3}\right]\setminus\{1\}$.}\hspace{1mm}\\
In this case,
\begin{align*}
\widetilde{f}(t_\pm)=0\quad
\Leftrightarrow\quad t_\pm =\frac{\|\check{z}\|}{2(\|\check{z}\|^2-1)}\pm\sqrt{\frac{(4\|\check{z}\|^2-3)\|\check{z}\|^2}{4(\|\check{z}\|^2-1)^2}}.
\end{align*}
Note that the sign of $\|\check{z}\|^2-1$ depends on whether $\|\check{z}\|<1$ or $\|\check{z}\|>1$. We will treat these cases separately.
\paragraph*{Case 2.1: $\|\check{z}\|\in \left(1,\sqrt{3}\right]$.}\hspace{1mm}\\
In this case, the plot of $f(t)$ is of the form as in Figure \ref{fig_f_case2_1} (except when $\|\check{z}\|=\sqrt{3}$, in which case $f(t)$ has the unique positive double root $\sqrt{3}$).
\begin{figure}[h]%
\centering%
\includegraphics[scale=0.25]{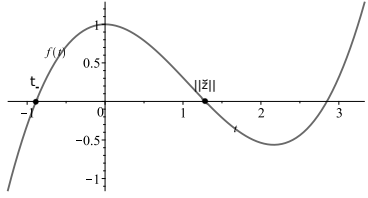}%
\caption{A typical plot of $f(t)$. Here, $P_3\left(\frac{\check{z}}{\|\check{z}\|}\right)=\frac{4}{5}\cdot\frac{2}{3\sqrt{3}}$, so that $\|\check{z}\|\in\left(1,\frac{2}{3\sqrt{3}}\right)$. The unique negative root of $f(t)$, that is $\check{t}(-\|\check{z}\|)$, and $t_-$ coincide in Case 2.1.}\label{fig_f_case2_1}%
\end{figure}
Also $\|\check{z}\|^2-1>0$ and, hence, $t_-=\check{t}(-\|\check{z}\|)$. We obtain
\begin{align*}
t_-&=\frac{\|\check{z}\|}{2(\|\check{z}\|^2-1)}\left(1-\sqrt{4\|\check{z}\|^2-3}\right),\\
\check{t}&=\frac{1}{2(\|\check{z}\|^2-1)}\left(-1+\sqrt{4\|\check{z}\|^2-3}\right),
\end{align*}
and, hence,
\begin{equation*}
F(\|\check{z}\|)=\frac{2\|\check{z}\|^4-5\|\check{z}\|^2+3+\sqrt{4\|\check{z}\|^2-3} (4\|\check{z}\|^4-7\|\check{z}\|^2+3 )}{18(\|\check{z}\|^2-1)^2}.
\end{equation*}
It is clear that $F|_{\left(1,\sqrt{3}\right]}$ is smooth. Using L'H\^ospital's rule for limits twice at $\|\check{z}\|=1$ yields
\begin{equation*}
\lim\limits_{\|\check{z}\|\to1,\; \|\check{z}\|>1}F(\|\check{z}\|)=\frac{4}{9},
\end{equation*}
which coincides with $F(1)$ determined in Case 1. This means that $F$ is continuous from the right at $\|\check{z}\|=1$. Next we will show that $F|_{\left(1,\sqrt{3}\right]}$ attains its maximum, namely at $\|\check{z}\|=\sqrt{3}$. To prove that we show that $\left.\frac{\partial F}{\partial \|\check{z}\|}\right|_{\left(1,\sqrt{3}\right)}>0$. The first derivative of $F$ is given by
\begin{equation*}
\frac{\partial F}{\partial \|\check{z}\|}( \|\check{z}\|)=\frac{ \|\check{z}\|\left( 8 \|\check{z}\|^4-18 \|\check{z}\|^2+9+\sqrt{4 \|\check{z}\|^2-3} \right)}{9\sqrt{4 \|\check{z}\|^2-3}( \|\check{z}\|^2-1)^2},
\end{equation*}
and $ \|\check{z}\|-1>0$ implies that in order to solve $\frac{\partial F}{\partial \|\check{z}\|}( \|\check{z}\|)=0$ with the restriction $\|\check{z}\|\in\left(1,\sqrt{3}\right)$ we only need to solve $8 \|\check{z}\|^4-18 \|\check{z}\|^2+9+\sqrt{4 \|\check{z}\|^2-3} =0$. Using MAPLE or any other computer algebra system one finds that the latter equation has no solutions in $\left(1,\sqrt{3}\right)$. It thus suffices check the sign of $\left.\frac{\partial F}{\partial \|\check{z}\|}\right|_{\left(1,\sqrt{3}\right)}$ at one point in the interval, say $\frac{1+\sqrt{3}}{2}$, to determine its global sign. We calculate
\begin{equation*}
\frac{\partial F}{\partial \|\check{z}\|}\left(\frac{1+\sqrt{3}}{2}\right)=\frac{4\sqrt{1+2\sqrt{3}}+2\sqrt{3}+2}{27}>0.
\end{equation*}
We conclude that
\begin{equation*}
\sup\limits_{\|\check{z}\|\in\left(1,\sqrt{3}\right]}F(\|\check{z}\|)=F\left(\sqrt{3}\right)=\frac{5}{6}.
\end{equation*}
\paragraph*{Case 2.2: $\|\check{z}\|\in \left[\frac{\sqrt{3}}{2},1\right)$.}\hspace{1mm}\\
This case works similarly to Case 2.1. Here, $f(t)$ has the shape as in Figure \ref{fig_f_case2_2} (except for the case $\|\check{z}\|=\frac{\sqrt{3}}{2}$, where $f(t)$ has the unique negative double root $\frac{\sqrt{3}}{2}$). 
\begin{figure}[h]%
\centering%
\includegraphics[scale=0.25]{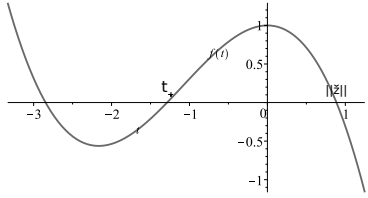}%
\caption{A plot of $f(t)$ with $P_3\left(\frac{\check{z}}{\|\check{z}\|}\right)=-\frac{4}{5}\cdot\frac{2}{3\sqrt{3}}$, so that $\|\check{z}\|\in\left(\frac{\sqrt{3}}{2},1\right)$. In Case 2.2, the biggest negative root of $f(t)$ which is $\check{t}(-\|\check{z}\|)$ by construction, and $t_+$ coincide.}\label{fig_f_case2_2}%
\end{figure}
In this case, $f(t)$ has, except if $\|\check{z}\|=\frac{\sqrt{3}}{2}$, precisely two negative roots, of which we need to consider the bigger one. Since $\|\check{z}\|^2-1<0$, we see that this is
\begin{equation*}
t_+=\frac{\|\check{z}\|}{2(\|\check{z}\|^2-1)}\left(1-\sqrt{4\|\check{z}\|^2-3}\right),
\end{equation*}
so that $t_+=\check{t}(-\|\check{z}\|)$ has the form
\begin{equation*}
\check{t}=\frac{1}{2(\|\check{z}\|^2-1)}\left(-1+\sqrt{4\|\check{z}\|^2-3}\right).
\end{equation*}
We see that formally the function $F$ for this case and $F$ in Case 2.1 coincide, i.e. we have for $F|_{\left[\frac{\sqrt{3}}{2},1\right)}$
\begin{equation*}
F(\|\check{z}\|)=\frac{2\|\check{z}\|^4-5\|\check{z}\|^2+3+\sqrt{4\|\check{z}\|^2-3} (4\|\check{z}\|^4-7\|\check{z}\|^2+3 )}{18(\|\check{z}\|^2-1)^2}
\end{equation*}
and for the derivative $\left.\frac{\partial F}{\partial \|\check{z}\|}\right|_{\left(\frac{\sqrt{3}}{2},1\right)}$
\begin{equation*}
\frac{\partial F}{\partial \|\check{z}\|}( \|\check{z}\|)=\frac{ \|\check{z}\|\left( 8 \|\check{z}\|^4-18 \|\check{z}\|^2+9+\sqrt{4 \|\check{z}\|^2-3} \right)}{9\sqrt{4 \|\check{z}\|^2-3}( \|\check{z}\|^2-1)^2}.
\end{equation*}
Proceeding analogously to Case 2.1 we will show that $\left.\frac{\partial F}{\partial \|\check{z}\|}\right|_{\left(\frac{\sqrt{3}}{2},1\right)}>0$. Note that the denominator of the formula for the first derivative of $F$ has no zeros in $\left(\frac{\sqrt{3}}{2},1\right)$, so we will not run into trouble with possibly singular values. Again, we use MAPLE to show that $8 \|\check{z}\|^4-18 \|\check{z}\|^2+9+\sqrt{4 \|\check{z}\|^2-3} =0$ has no solutions in $\left(\frac{\sqrt{3}}{2},1\right)$. Hence, the global sign of $\left.\frac{\partial F}{\partial \|\check{z}\|}\right|_{\left(\frac{\sqrt{3}}{2},1\right)}$ coincides with the sign of
\begin{equation*}
\frac{\partial F}{\partial \|\check{z}\|}\left(\frac{1}{2}\left(\frac{\sqrt{3}}{2}+1\right)\right)=-\frac{1220}{1089}\sqrt{3}\sqrt{-5+4\sqrt{3}}+\frac{5824}{3267}\sqrt{3}-\frac{6376}{3267}\sqrt{-5+4\sqrt{3}}+\frac{10112}{3267}>0,
\end{equation*}
and thus $F_{\left[\frac{\sqrt{3}}{2},1\right)}$ does not attain its maximum in its domain of definition, but at the limit $\|\check{z}\|\to 1$, assuming that limit exists. For the existence we need to check that $F$ is continuous from the left at $\|\check{z}\|=1$. This is done in the same way we have shown continuity from the right, that is by applying L'H\^ospital's rule twice. As expected, we obtain
\begin{equation*}
\lim\limits_{\|\check{z}\|\to1,\; \|\check{z}\|<1}F(\|\check{z}\|)=\frac{4}{9},
\end{equation*} 

Summarising, we have shown that $F:\left[\frac{\sqrt{3}}{2},\sqrt{3}\right]\to\mathbb{R}_{>0}$ is continuous and attains its maximum at $\|\check{z}\|=\sqrt{3}$, $F\left(\sqrt{3}\right)=\frac{5}{6}$. Since the negative eigenvalue $\lambda_-$ of the representation matrix of $P_3(\check{z},dz,dz)$ was arbitrary, we conclude with \eqref{eqn_F_for_neg_EVs_PSR} that for all such negative eigenvalues $\lambda_-$ we have
\begin{equation*}
\lambda_-\geq-\max\limits_{\|\check{z}\|\in\left[\frac{\sqrt{3}}{2},\sqrt{3}\right]}F(\|\check{z}\|)=-\frac{5}{6}.
\end{equation*}
The point $\check{z}\in\partial\mathrm{dom}(\mathcal{H})$ was also arbitrary and, thus, using the linearity of $P_3(z,dz,dz)$ we obtain
\begin{equation*}
\forall z\in\mathrm{dom}(\mathcal{H}):\quad P_3(z,dz,dz)>-\frac{5}{6}\langle dz,dz\rangle.
\end{equation*}
Note that our calculations also show that $\lambda_-=-\frac{5}{6}$ can only possibly be a negative eigenvalue of the representation matrix of $P_3(\check{z},dz,dz)$ at a point $\check{z}\in\partial\mathrm{dom}(\mathcal{H})$ with norm $\|\check{z}\|=\sqrt{3}$.

We want to stress again at this point that the obtained lower and upper bounds for $P_3(z,dz,dz)$ hold for all CCPSR manifolds $\mathcal{H}\subset\{h=1\}$ of dimension $n\geq 1$ with $h$ of the form \eqref{standard_form_h} and $\left(\begin{smallmatrix}
1\\ 0
\end{smallmatrix}\right)\in\mathcal{H}$.

It remains to show that the upper bound in \eqref{P_3(z,.,.)_lower_upper_boundary_PSR} is sharp in the stated sense.
To do so, we will give an example of a CCPSR manifold of dimension $n$ for each $n\in\mathbb{N}$. For any $n\in\mathbb{N}$, let $\left(\begin{smallmatrix}
x\\y 
\end{smallmatrix}\right)=(x,y_1,\ldots,y_n)^T$ denote linear coordinates of $\mathbb{R}^{n+1}$ as usual and consider the cubic polynomial
\begin{equation*}
h:\mathbb{R}^n\to\mathbb{R},\quad h=x^3-x\langle y,y\rangle + \frac{2}{3\sqrt{3}}y_n^3,
\end{equation*}
and the corresponding centro-affine hypersurface $\mathcal{H}\subset\{h=1\}$, which is the connected component of $\{h=1\}$ that contains the point $\left(\begin{smallmatrix}
x\\y 
\end{smallmatrix}\right)=\left(\begin{smallmatrix}
1\\
0
\end{smallmatrix}\right)$. Then $\mathcal{H}$ is a CCPSR manifold of dimension $n$. We will not prove this here,
since it follows from Theorem \ref{thm_Cn}. Note that while this proposition is used in the later proof of the latter theorem, the sharpness of the upper bound is not required therein.
To show that the upper bound is in fact sharp in the stated sense, consider the point $\check{z}_+=\left(
	0,\ldots,0,\sqrt{3}
	\right)^T\in\partial\mathrm{dom}(\mathcal{H})$.
We obtain $P_3(\check{z}_+,dz,dz)=\frac{2}{3}dz_n^2$ and the corresponding symmetric matrix has precisely the eigenvalues $\lambda_1=0$ with eigenspace-dimension $n-1$, and $\lambda_2=\frac{2}{3}$ with eigenspace-dimension $1$.
This proves our claim.
\end{proof}
\end{Prop}

Next, we will study the boundary behaviour of the centro-affine fundamental form of CCGPSR manifolds. For an explanation why this term is used, see the discussion under \cite[Thm.\,1.18]{CNS} and the related chapter in Melrose's book \cite[Ch.\,8]{Me}.
\begin{Def}
\label{def_regular_bdr_GPSR}
Let $\mathcal{H}\subset\{h=1\}$ be a CCGPSR manifold of dimension $n\geq 1$ and let $U=\mathbb{R}_{>0}\cdot\mathcal{H}$ be the corresponding convex cone (cf. Proposition \ref{prop_convexcone}). Then $\mathcal{H}$ has \textsf{regular boundary behaviour} if
\begin{enumerate}[(i)]
\item $dh_p\ne 0$ for all $p\in\partial U\setminus\{0\}$, i.e. $\mathcal{H}$ is not singular at infinity in the sense of Definition \ref{def_singular_GPSR}, \label{eqn_1.17i}
\item $\left.-\partial^2 h\right|_{T(\partial U\setminus\{0\})\times T(\partial U\setminus\{0\})}\geq 0$ and $\dim\ker\left(\left.-\partial^2 h\right|_{T(\partial U\setminus\{0\})\times T(\partial U\setminus\{0\})}\right)=1$ for all $p\in\partial U\setminus\{0\}$.\label{eqn_1.17ii}
\end{enumerate}
\end{Def}
Note that Definition \ref{def_regular_bdr_GPSR} is equivalent to \cite[Def.\,1.17]{CNS} restricted to CCGPSR manifolds. We also want to stress that Definition \ref{def_regular_bdr_GPSR} is independent of the chosen linear coordinates of the ambient space $\mathbb{R}^{n+1}$.
\begin{rem}\label{rem_1.17_alternate_form}
With the functions $\alpha$ and $\beta$ as in \eqref{alpha_def} and \eqref{beta_def},
Lemma \ref{lemma_boundary_conditions_alpha_beta} shows that the conditions \hyperref[eqn_1.17i]{(i)} and \hyperref[eqn_1.17ii]{(ii)} in Definition \ref{def_regular_bdr_GPSR} are equivalent to
\begin{enumerate}[(i)]
\item $d\beta_{\overline{z}}(\overline{z})\neq 0$ for all $\overline{z}\in\partial\mathrm{dom}(\mathcal{H})$,\label{eqn_1.17i_alternate}
\item $-\partial^2\beta_{\overline{z}}>0$ for all $\overline{z}\in\partial\mathrm{dom}(\mathcal{H})$,\label{eqn_1.17ii_alternate}
\end{enumerate}
respectively.
\end{rem}
It turns out that for CCPSR manifolds the condition Def. \ref{def_regular_bdr_GPSR} \hyperref[eqn_1.17i]{(i)} always implies Def. \ref{def_regular_bdr_GPSR} \hyperref[eqn_1.17ii]{(ii)}:
\begin{Th}
\label{thm_irregularity_implications_CCPSR}
A CCPSR manifolds of dimension $n\geq 1$ is not singular at infinity in the sense of Definition \ref{def_singular_GPSR} if and only if it has regular boundary behaviour as defined in Definition \ref{def_regular_bdr_GPSR}.
\begin{proof}
A CCPSR manifold $\mathcal{H}$ that has regular boundary behaviour is by definition not singular at infinity. For the other direction, consider first $n=1$. Then Def. \ref{def_regular_bdr_GPSR} \hyperref[eqn_1.17ii]{(ii)} is trivially satisfied.

To prove the statement of this theorem for $n\geq 2$, it suffices to prove it for $n=2$. To see this, consider any CCPSR manifold $\mathcal{H}$ of dimension $n>2$ and assume that Def. \ref{def_regular_bdr_GPSR} \hyperref[eqn_1.17i]{(i)} holds for $\mathcal{H}$. Assume without loss of generality that $\mathcal{H}$ is in standard form.
Considering Remark \ref{rem_1.17_alternate_form}, Def. \ref{def_regular_bdr_GPSR} \hyperref[eqn_1.17ii]{(ii)} holds true if and only if Rem. \ref{rem_1.17_alternate_form} \hyperref[eqn_1.17ii_alternate]{(ii)} holds true. To show the latter we need to show that $-\partial^2\beta_{\overline{z}}(v,v)>0$ for all $\overline{z}\in\partial\mathrm{dom}(\mathcal{H})$ and all $0\ne v\in T_{\overline{z}}\partial\mathrm{dom}(\mathcal{H})\subset \mathbb{R}^n$. Observe that for any $2$-dimensional linear subspace $E=\mathrm{span}\{w_1,w_2\} \subset\mathbb{R}^n$, where $w_1$ and $w_2$ are chosen such that they are orthonormal with respect to $\langle\cdot,\cdot\rangle$, the restricted polynomial
\begin{equation*}
h^E\left(\left(\begin{smallmatrix}
x\\
t_1\\
t_2
\end{smallmatrix}
\right)\right):=x^3-x(t_1^2+t_2^2)+P_3(t_1w_1+t_2w_2)
\end{equation*}
defines a $2$-dimensional CCPSR manifold $\mathcal{H}^E\subset\left\{h^E=1\right\}\subset\mathbb{R}^3$ as the connected component containing the point $\left(\begin{smallmatrix}
x\\
t_1\\
t_2
\end{smallmatrix}
\right)=\left(\begin{smallmatrix}
1\\
0\\
0
\end{smallmatrix}
\right)$. Furthermore, 
\begin{equation*}
\mathrm{dom}\left(\mathcal{H}^E\right)\ni\left(\begin{smallmatrix}
t_1\\
t_2
\end{smallmatrix}
\right)\mapsto t_1w_1+t_2w_2\in\mathrm{dom}(\mathcal{H})
\end{equation*}
is an embedding. Note that the explicit formula for $h^E$ in general depends on the choice of basis for $E$. Hence, if we want to show that $-\partial^2\beta_{\overline{z}}(v,v)>0$ for some fixed $\overline{z}\in\partial\mathrm{dom}(\mathcal{H})$ and $0\ne v\in T_{\overline{z}}(\partial\mathrm{dom}(\mathcal{H}))$, it suffices to show Rem. \ref{rem_1.17_alternate_form} \hyperref[eqn_1.17ii_alternate]{(ii)} for $\mathcal{H}^E$ and $h^E$, respectively $\beta^E\left(\left(\begin{smallmatrix}
t_1\\
t_2
\end{smallmatrix}
\right)\right)=h^E\left(\left(\begin{smallmatrix}
1\\
t_1\\
t_2
\end{smallmatrix}
\right)\right)$, with $E=\mathrm{span}\{\overline{z},v\}$\footnote{Strictly speaking, at this point we need to choose two orthonormal vectors $w_1,w_2$, such that $\mathrm{span}\{\overline{z},v\}=\mathrm{span}\{w_1,w_2\}$.} where we view $v$ as an element of $\mathbb{R}^n$. Hence, proving the statement of this theorem for all $2$-dimensional CCPSR manifolds will also prove it for these of higher dimension. Since the conditions in Definition \ref{def_regular_bdr_GPSR} are independent of the linear coordinates chosen for the ambient space $\mathbb{R}^{n+1}$, we can reduce our studies to the classification of $2$-dimensional CCPSR manifolds up to equivalence given in \cite[Thm.\,1]{CDL}\footnote{At the time the article \cite{CDL} was written and published, it was still an open problem to show that a PSR manifold is closed if and only if it is geodesically complete, which has first been proven in \cite[Thm.\,2.5]{CNS}.}, see Theorem \ref{lowdimpsrclassTHM}. We will do a case-by-case check for the surfaces a)--e) and the one-parameter family of surfaces f) in Theorem \ref{lowdimpsrclassTHM}. For the cases a)--e) we will study the $P_3$-part the calculated standard form $\widetilde{h}=x^3-x(y^2+z^2)+P_3\left(\left(\begin{smallmatrix}
y\\ z
\end{smallmatrix}\right)\right)$ \eqref{standard_form_h} of each cubic $h$ corresponding to a CCPSR surface $\mathcal{H}\subset\{h=1\}$ obtained in Example \ref{example_PSRsurfaces_standard_form} with the property that $\mathcal{H}$ is equivalent to the connected component of $\left\{\widetilde{h}=1\right\}$ that contains the point $\left(\begin{smallmatrix}
x\\ y\\ z
\end{smallmatrix}\right)=\left(\begin{smallmatrix}
1\\ 0\\ 0
\end{smallmatrix}\right)$. We can then use Lemma \ref{lemma_singular_PSR_max_P_3}, which says that the value of $\max\limits_{\left\|\left(\begin{smallmatrix}
y\\
z
\end{smallmatrix}
\right)\right\|=1}\left|P_3\left(\left(\begin{smallmatrix}
y\\
z
\end{smallmatrix}
\right)\right)\right|\in\left[0,\frac{2}{3\sqrt{3}}\right]$ determines whether $\mathcal{H}$ is singular at infinity or not. In the cases where $\mathcal{H}$ is not singular at infinity, that is fulfils Def. \ref{def_regular_bdr_GPSR} \hyperref[eqn_1.17i_alternate]{(i)}, we need to show that it also fulfils Def. \ref{def_regular_bdr_GPSR} \hyperref[eqn_1.17ii_alternate]{(ii)}.
For the one-parameter family f) we will use another method and explain why in this case the form \eqref{standard_form_h} is not the best choice to work with in order to prove our claim.

\paragraph*{a) $\mathcal{H}=\{h=xyz=1,\ x>0,\ y>0\}$.}\hspace*{0mm}\\
Equation \eqref{eqn_CDL_a_equiv_h_form} implies that $P_3\left(\left(\begin{smallmatrix}
y\\
z
\end{smallmatrix}
\right)\right)=-\frac{2}{3\sqrt{3}}y^3+\frac{2}{\sqrt{3}}yz^2$. Since $\mathcal{H}$ is a CCPSR surface and $P_3\left(\left(\begin{smallmatrix}
-1\\
0
\end{smallmatrix}
\right)\right)=\frac{2}{3\sqrt{3}}$, Lemma \ref{lemma_singular_PSR_max_P_3} implies that $\mathcal{H}$ is singular at infinity.

\paragraph*{b) $\mathcal{H}=\{h=x(xy-z^2)=1,\ x>0\}$.}\hspace*{0mm}\\
By equation \eqref{eqn_CDL_b_equiv_h_form}, $P_3\left(\left(\begin{smallmatrix}
y\\
z
\end{smallmatrix}
\right)\right)=\frac{2}{3\sqrt{3}}y^3+\frac{1}{\sqrt{3}}yz^2$ with $P_3\left(\left(\begin{smallmatrix}
1\\
0
\end{smallmatrix}
\right)\right)=\frac{2}{3\sqrt{3}}$. Hence, $\mathcal{H}$ is singular at infinity.

\paragraph*{c) $\mathcal{H}=\{h=x(yz+x^2)=1,\ x<0,\ y>0\}$.}\hspace*{0mm}\\
This case is a little more complicated in comparison with \textbf{a)} and \textbf{b)}. Equation \eqref{eqn_CDL_c_equiv_h_form} implies that $P_3\left(\left(\begin{smallmatrix}
y\\
z
\end{smallmatrix}
\right)\right)=\frac{2\sqrt{2}}{\sqrt{15}}y^2z+\frac{14\sqrt{2}}{15\sqrt{15}}z^3$.
We now need to determine $\max\limits_{\left\|\left(\begin{smallmatrix}
y\\
z
\end{smallmatrix}
\right)\right\|=1}\left|P_3\left(\left(\begin{smallmatrix}
y\\
z
\end{smallmatrix}
\right)\right)\right|$. We find for $v=\left(\begin{smallmatrix}
\frac{\sqrt{3}}{2\sqrt{2}}\\
\frac{\sqrt{5}}{2\sqrt{2}}
\end{smallmatrix}
\right)$, $\|v\|=1$, that $P_3(v)=\frac{2}{3\sqrt{3}}$. Hence, $\mathcal{H}$ being closed and connected implies that $\max\limits_{\left\|\left(\begin{smallmatrix}
y\\
z
\end{smallmatrix}
\right)\right\|=1}\left|P_3\left(\left(\begin{smallmatrix}
y\\
z
\end{smallmatrix}
\right)\right)\right|=\frac{2}{3\sqrt{3}}$. This shows that $\mathcal{H}$ is singular at infinity. Note that $v$ can be found without the help of a computer algebra system like MAPLE by considering the equation $dP_3|_{\left(\begin{smallmatrix}
y\\
z
\end{smallmatrix}
\right)}=r\left\langle \left(\begin{smallmatrix}
y\\
z
\end{smallmatrix}
\right),\cdot\right\rangle$, $r>0$, which is not difficult to solve in this case since $P_3\left(\left(\begin{smallmatrix}
y\\
z
\end{smallmatrix}
\right)\right)$ is reducible.

\paragraph*{d) $\mathcal{H}=\{h=z(x^2+y^2-z^2)=1,\ z<0\}$.}\hspace*{0mm}\\
From equation \eqref{eqn_CDL_d_equiv_h_form} we obtain that in this case $P_3\left(\left(\begin{smallmatrix}
y\\ z
\end{smallmatrix}\right)\right)\equiv 0$. Hence, $\max\limits_{\left\|\left(\begin{smallmatrix}
y\\
z
\end{smallmatrix}
\right)\right\|=1}\left|P_3\left(\left(\begin{smallmatrix}
y\\
z
\end{smallmatrix}
\right)\right)\right|=0$ and $\mathcal{H}$ is thus not singular at infinity. It is immediate that $\mathrm{dom}(\widetilde{\mathcal{H}})=\left\{\left\|\left(\begin{smallmatrix}
y\\
z
\end{smallmatrix}
\right)\right\|<1\right\}$ and that for the corresponding function $\beta(y,z)=1-y^2-z^2$ as in \eqref{beta_def} we have $d\beta=-2ydy-2zdz$. Hence, $d\beta$ vanishes at no point in $\partial\mathrm{dom}(\widetilde{\mathcal{H}})$, so Lemma \ref{lemma_boundary_conditions_alpha_beta} implies that $\widetilde{\mathcal{H}}$, and thus also $\mathcal{H}$, fulfils Def. \ref{def_regular_bdr_GPSR} \hyperref[eqn_1.17i_alternate]{(i)}. Furthermore
\begin{equation*}
\partial^2\beta_{\left(\begin{smallmatrix}
y\\
z
\end{smallmatrix}
\right)}=\left(
\begin{matrix}
-2 & 0\\
0 & -2
\end{matrix}
\right)<0\quad \forall \left(\begin{matrix}
y\\
z
\end{matrix}
\right)\in\partial\mathrm{dom}(\widetilde{\mathcal{H}}),
\end{equation*}
so $\widetilde{\mathcal{H}}$, and equivalently $\mathcal{H}$, fulfils Def. \ref{def_regular_bdr_GPSR} \hyperref[eqn_1.17ii_alternate]{(ii)}.

\paragraph*{e) $\mathcal{H}=\{h=x(y^2-z^2)+y^3=1,\ y<0,\ x>0\}$.}\hspace*{0mm}\\
From equation \eqref{eqn_CDL_e_equiv_h_form} we know that $P_3\left(\left(\begin{smallmatrix}
y\\
z
\end{smallmatrix}
\right)\right)=\frac{2}{3\sqrt{3}}y^3-\frac{1}{2\sqrt{3}}yz^2$. Hence, $P_3\left(\left(\begin{smallmatrix}
1\\
0
\end{smallmatrix}
\right)\right)=\frac{2}{3\sqrt{3}}$, which shows that $\mathcal{H}$ is singular at infinity.

\paragraph*{f)\footnote{For this one-parameter family of CCPSR surfaces which are each contained in the level set of the respective Weiersta{\ss} cubic with positive discriminant $h$, the method used for \textbf{a)}--\textbf{e)} has proven itself to be unsuitable. This is because the formulas for the corresponding function $\beta$ as in \eqref{beta_def} and the derivatives corresponding to $h$ when brought to the form \eqref{standard_form_h} might not depend on $b\in (-1,1)$ in a complicated way, but studying the system of equations $v\in\ker d\beta$, $v\in\ker\partial^2\beta$, $\beta=0$, turned out to be quite difficult. We will thus consider Definition \ref{def_regular_bdr_GPSR} and not the equivalent conditions in Remark \ref{rem_1.17_alternate_form} to prove our claim for this one-parameter family.} $\mathcal{H}_b=\{h=y^2z-4x^3+3xz^2+bz^3=1,\ z<0,\ 2x>z\}$, $b\in(-1,1)$.}\hspace*{0mm}\\
For all $b\in(-1,1)$, the projective curve $C:=\{h=y^2z-4x^3+3xz^2+bz^3=0\}\subset\mathbb{R}\mathrm{P}^2$ has no singularities, cf. \cite[Prop.\,3]{CDL}, which means that $dh_p\ne 0$ for all $p\in \{h=0\}\setminus\{0\}\subset\mathbb{R}^3$. Hence, each $\mathcal{H}_b$, $b\in(-1,1)$, is not singular at infinity in the sense of Definition \ref{def_singular_GPSR} and, hence, fulfils condition Def. \ref{def_regular_bdr_GPSR} \hyperref[eqn_1.17i]{(i)}. Note that $\mathcal{H}_b$ not being singular at infinity for all $b\in(-1,1)$ also follows easily from equation \eqref{eqn_CDL_f_equiv_h_form} in Example \ref{example_PSRsurfaces_standard_form}. We need to show that each $\mathcal{H}_b$, $b\in(-1,1)$, also fulfils Def. \ref{def_regular_bdr_GPSR} \hyperref[eqn_1.17ii]{(ii)}. In order to prove this, we need to determine $\partial(\mathbb{R}_{>0}\cdot\mathcal{H}_b)\subset\{h=0,\ z\leq 0,\ 2x\geq z\}\subset\mathbb{R}^3$ for each $b\in(-1,1)$. Observe that $\{h=0,\ z\leq 0,\ 2x\geq z\}\cap \{z=0\}=\{x=0,\ z=0\}$.
Hence, the line $\{x=0,\ z=0\}$ is contained in $\{h=0,\ z\leq 0,\ 2x\geq z\}$, but $\mathbb{R}_{>0}\cdot\mathcal{H}_b$ being a convex cone which has the property described in Lemma \ref{precomp_cor} shows that $\{x=0,\ z=0\}\cap \partial(\mathbb{R}_{>0}\cdot\mathcal{H}_b)=\left\{\left(\begin{smallmatrix}
0 \\ 0 \\ 0
\end{smallmatrix}
\right)
\right\}$. For $z<0$ we will determine the intersection $\{z=-1\}\cap \partial(\mathbb{R}_{>0}\cdot\mathcal{H}_b)$, which can then be used with the homogeneity of $h$ to obtain the whole set $\partial(\mathbb{R}_{>0}\cdot\mathcal{H}_b)$. We find
\begin{equation}
h\left(\left(\begin{smallmatrix}
x\\ y\\ -1
\end{smallmatrix}
\right)\right)=0\quad
\Leftrightarrow\quad \underbrace{-y^2-4x^3+3x-b}_{=:\rho_b}=0,\label{eqn_weierstrass_rho}
\end{equation}
where $\rho_b\left(\left(\begin{smallmatrix}
x\\ y  
\end{smallmatrix}\right)\right)=h\left(\left(\begin{smallmatrix}
x\\ y \\ -1
\end{smallmatrix}\right)\right)$. We consider $\rho_b$ to be defined for all $b\in\mathbb{R}$, not just for $b\in(-1,1)$. Let
\begin{equation*}
V=\mathrm{span}\left\{\left(
\begin{smallmatrix}
1 \\ 0 \\ 0
\end{smallmatrix}
\right),
\left(
\begin{smallmatrix}
0 \\ 1 \\ 0
\end{smallmatrix}
\right)
\right\}
\end{equation*}
and observe that $\mathcal{H}_b$ not being singular at infinity implies that the tangent space $T\partial(\mathbb{R}_{>0}\cdot\mathcal{H}_b)$ fulfils
\begin{equation*}
T_p\partial(\mathbb{R}_{>0}\cdot\mathcal{H}_b)=\mathbb{R}\cdot p \oplus \left(\ker dh_p\cap V\right)\quad \forall p \in \{z=-1\}\cap\partial(\mathbb{R}_{>0}\cdot\mathcal{H}_b).
\end{equation*}
Furthermore, the $1$-dimensional linear subspaces $\mathbb{R}\cdot p$ and $\ker dh_p\cap V$ of $T_p\partial(\mathbb{R}_{>0}\cdot\mathcal{H}_b)$ are orthogonal with respect to the positive-semidefinite bilinear form $-\partial^2 h_p$, which follows from $-\partial^2h_p(p,\cdot)=-2dh_p(\cdot)$. Also note that $\ker dh_p\cap V$ is always $1$-dimensional since the position vector $p\ne 0$ is always an element of $\ker dh_p$ for all $p\in\partial(\mathbb{R}_{>0}\cdot\mathcal{H}_b)$. Thus, in order to prove that Def. \ref{def_regular_bdr_GPSR} \hyperref[eqn_1.17ii]{(ii)} is fulfilled for each $\mathcal{H}_b$, $b\in(-1,1)$, it suffices to show that $-\partial^2h|_{(\ker dh_p\cap V)\times(\ker dh_p\cap V)}>0$. We obtain
\begin{equation*}
dh=(-12x^2+3z^2)dx + 2yz dy + (y^2+6xz+3bz^2)dz
\end{equation*}
and
\begin{equation*}
\partial^2h=\left(\begin{matrix}
-24x & 0 & 6z \\
0 & 2z & 2y \\
6z & 2y & 6x+6bz
\end{matrix}
\right).
\end{equation*}
Since $\mathcal{H}_b$ is not singular at infinity, it follows that at each point $p=\left(\begin{smallmatrix}
x \\ y \\ z
\end{smallmatrix}
\right)\in \{z=-1\}\cap\partial(\mathbb{R}_{>0}\cdot\mathcal{H}_b)$, $\ker dh_p\cap V$ is given by
\begin{equation*}
\ker dh_p\cap V=\mathrm{span}\left\{
\left(\begin{matrix}
-\partial_y h_p \\ \partial_x h_p \\ 0
\end{matrix}
\right)\right\}=\mathrm{span}\left\{\left(\begin{matrix}
2y \\ -12x^2+3 \\ 0
\end{matrix}
\right)\right\}.
\end{equation*}
Hence, $-\partial^2h|_{(\ker dh_p\cap V)\times(\ker dh_p\cap V)}>0$ if and only if 
\begin{equation}
 -\partial^2h_p\left(\left(\begin{smallmatrix}
2y \\ -12x^2+3 \\ 0
\end{smallmatrix}
\right),\left(\begin{smallmatrix}
2y \\ -12x^2+3 \\ 0
\end{smallmatrix}
\right)\right)>0
\quad\Leftrightarrow  \quad 16xy^2+\left(4\sqrt{3}x^2-\sqrt{3}\right)^2>0\label{eqn_weierstrass_ii_estimate}
\end{equation}
for all $p=\left(\begin{smallmatrix}
x \\ y \\ z
\end{smallmatrix}
\right)\in \{h=0,\ z=-1\}\cap\partial(\mathbb{R}_{>0}\cdot\mathcal{H}_b)$. We will first check the above inequality \eqref{eqn_weierstrass_ii_estimate} for $y=0$. In that case, \eqref{eqn_weierstrass_ii_estimate} can only be false if $x=\pm \frac{1}{2}$. Then with $\rho_b$ defined as in \eqref{eqn_weierstrass_rho} we obtain
\begin{equation*}
0=\rho_b\left(\left(\begin{smallmatrix}
\pm\frac{1}{2}\\ 0
\end{smallmatrix}\right)\right)=\mp\frac{1}{2}\pm \frac{3}{2}-b=\pm 1-b.
\end{equation*}
This is however a contradiction to $b\in(-1,1)$ and, hence, \eqref{eqn_weierstrass_ii_estimate} holds at all points in $\{z=-1,\ y=0\}\cap\partial(\mathbb{R}_{>0}\cdot\mathcal{H}_b)$. Now let $y\ne 0$. We see that then \eqref{eqn_weierstrass_ii_estimate} is true for all $x\geq 0$, independent of $b\in(-1,1)$. It thus remains to check the inequality \eqref{eqn_weierstrass_ii_estimate} for points in $\{z=-1,\ x<0\}\cap\partial(\mathbb{R}_{>0}\cdot\mathcal{H}_b)$. Note that the latter set might be empty, in fact one can show that it is empty if and only if $0<b<1$, but we will not need this information for our proof. Observe that for all $b_1,b_2\in\mathbb{R}$ with $b_1<b_2$,
\begin{equation*}
\rho_{b_1}\left(\left(\begin{smallmatrix}
x\\ y
\end{smallmatrix}
\right)\right)
>\rho_{b_2}\left(\left(\begin{smallmatrix}
x\\ y
\end{smallmatrix}
\right)\right)
\end{equation*}
for all $\left(\begin{smallmatrix}
x\\ y
\end{smallmatrix}
\right)\in\mathbb{R}^2$. Hence, $\rho_{b_1}|_{\{\rho_{b_2}=0\}}>0$, which in particular implies that 
\begin{equation}
\rho_{-1}|_{\{\rho_{b}=0\}}>0\label{eqn_rho_b12_estimate}
\end{equation}
for all $b\in(-1,1)$. With the fact that $\left(\begin{smallmatrix}
\frac{1}{2}\\ 0\\ -1
\end{smallmatrix}\right)\in\{z=-1\}\cap\left(\mathbb{R}_{>0}\cdot\mathcal{H}_b\right)$, cf. Example \ref{example_PSRsurfaces_standard_form}, and $\rho_{-1}\left(\left(\begin{smallmatrix}
\frac{1}{2}\\ 0
\end{smallmatrix}\right)\right)=2>0$ it follows that $\{z=-1\}\cap\overline{\left(\mathbb{R}_{>0}\cdot\mathcal{H}_b\right)}$ is a subset of the connected component of $\left\{\rho_{-1}>0\right\}\times\{-1\}\subset\mathbb{R}^{3}$ that contains the point $\left(\begin{smallmatrix}
\frac{1}{2}\\ 0\\ -1
\end{smallmatrix}\right)$, see Figure \ref{fig_rho_b_contained}.
\begin{figure}[h]
\centering
\includegraphics[scale=0.25]{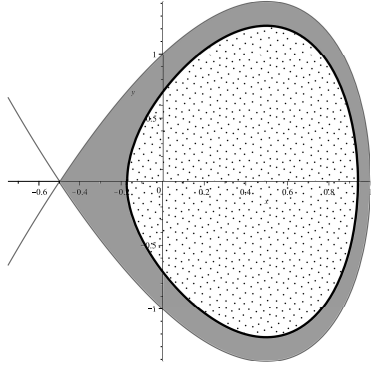}
\caption{The connected component of $\left\{\rho_{-1}>0\right\}$ that contains the point $\left(\frac{1}{2},0,-1\right)^T$ is (partly) marked in grey, its boundary is a part of the set $\left\{\rho_{-1}=0\right\}$ which is also shown in $\{-1<x<1\}\subset\mathbb{R}^2$. The dotted area in the plot is the connected component of $\left\{\rho_{-\frac{1}{2}}>0\right\}$ that contains the point $\left(\frac{1}{2},0,-1\right)^T$.}\label{fig_rho_b_contained}
\end{figure}
Further observe that $\rho_b\left(\left(\begin{smallmatrix}
-\frac{1}{2}\\ y
\end{smallmatrix}
\right)\right)=-y^2-1-b<0$ for all $y\in\mathbb{R}$ and $b\in(-1,1)$, and that $\mathcal{H}_b\subset\{z<0,\ 2x>z\}$ implies that $\{z=-1\}\cap\partial(\mathbb{R}_{>0}\cdot\mathcal{H}_b)$ is contained in $\left\{z=-1,\ x>-\frac{1}{2}\right\}$ for all $b\in(-1,1)$. In particular there exists no $b\in(-1,1)$, such that the $x$-coordinate of an element in $\{z=-1\}\cap\partial(\mathbb{R}_{>0}\cdot\mathcal{H}_b)$ has the value $-\frac{1}{2}$. Hence, \eqref{eqn_weierstrass_ii_estimate} and \eqref{eqn_rho_b12_estimate} imply that in order to prove that $\mathcal{H}_b$ fulfils Def. \ref{def_regular_bdr_GPSR} \hyperref[eqn_1.17ii]{(ii)} it suffices to show
\begin{equation}
\left\{16xy^2+\left(4\sqrt{3}x^2-\sqrt{3}\right)^2=0\right\}\cap\left\{\rho_{-1}=0,\ x\geq -\frac{1}{2}\right\}\cap\{x<0\}=\left\{\left(\begin{matrix}
x \\ y
\end{matrix}
\right)=\left(\begin{matrix}
-\frac{1}{2} \\ 0
\end{matrix}
\right)\right\},\label{eqn_weierstrass_-1limit}
\end{equation}
since $\left.\left(16xy^2+(4\sqrt{3}x^2-\sqrt{3})^2\right)\right|_{\{x>-\frac{1}{2},\ y=0\}}>0$, see also Figure \ref{fig_rho_b_section}.
\begin{figure}[h]
\centering
\includegraphics[scale=0.25]{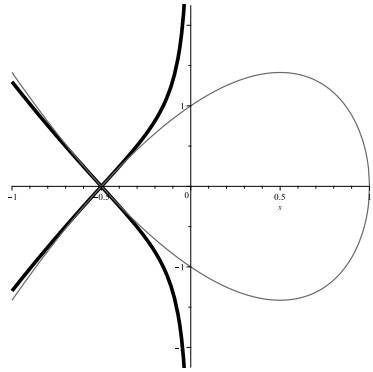}
\caption{The thick black curves represent the set $\left\{16xy^2+\left(4\sqrt{3}x^2-\sqrt{3}\right)^2=0\right\}\cap\{-1<x<0\}$, the thinner grey curve is the set $\left\{\rho_{-1}=0\right\}\cap\{-1<x<1\}$.}\label{fig_rho_b_section}
\end{figure}
We insert $\rho_{-1}=0$, which is equivalent to $y^2=-4x^3+3x+1$, into $16xy^2+\left(4\sqrt{3}x^2-\sqrt{3}\right)^2=0$ and obtain
\begin{equation*}
F(x):=-16x^4+24x^2+16x+3=0.
\end{equation*}
One can now use a computer algebra system like MAPLE and find that $F(x)=0$ and $-\frac{1}{2}\leq x<0$ if and only if $x=-\frac{1}{2}$. This proves \eqref{eqn_weierstrass_-1limit} and, hence, shows that each $\mathcal{H}_b$, $b\in(-1,1)$, fulfils Def. \ref{def_regular_bdr_GPSR} \hyperref[eqn_1.17ii]{(ii)}.

This finishes the proof of Theorem \ref{thm_irregularity_implications_CCPSR}.
\end{proof}
\end{Th}
Lemma \ref{lemma_singular_PSR_max_P_3} and Theorem \ref{thm_irregularity_implications_CCPSR} imply the following.
\begin{Cor}
\label{cor_regularity_boundary_generating_set_CCPSR}
A CCPSR manifold in standard form $\mathcal{H}\subset\{h=x^3-x\langle y,y\rangle +P_3(y)=1\}$
has regular boundary behaviour if and only if $\max\limits_{\|z\|=1}P_3(z)<\frac{2}{3\sqrt{3}}$.
\end{Cor}

Next, we will prove a property of CCPSR manifolds in standard form which already implies that the moduli set of CCPSR manifolds is generated by a path-connected subset of $\mathrm{Sym}^3(\mathbb{R}^n)^*$.
\begin{Prop}
\label{prop_PSR_pathconnectedness}
Let $\mathcal{H}\subset\{h=x^3-x\langle y,y\rangle +P_3(y)=1\}$ be a CCPSR manifold in standard form.
Then for all $s\in[0,1]$, the connected component $\mathcal{H}_s\subset\{h_s:=x^3-x\langle y,y\rangle +sP_3(y)=1\}$ that contains the point $\left(\begin{smallmatrix}
x\\
y
\end{smallmatrix}\right)=\left(\begin{smallmatrix}
1\\
0
\end{smallmatrix}
\right)$ is a CCPSR manifold.
\begin{proof}
For all $s\in[0,1]$,
\begin{equation*}
\max\limits_{\|z\|=1}|sP_3(z)|\leq \max\limits_{\|z\|=1}|P_3(z)|\leq\frac{2}{3\sqrt{3}}.
\end{equation*}
Hence Corollary \ref{cor_PSR_closedness_condition_P3} shows that for each corresponding $\mathcal{H}_s$, which is by definition closed as a subset of $\mathbb{R}^{n+1}$, the necessary condition for $\mathcal{H}_s$ to be a CCPSR manifold, namely that the set $\left(\mathbb{R}_{>0}\cdot\mathcal{H}_s\right)\cap \left.\left\{\left(\begin{smallmatrix}
1\\
z
\end{smallmatrix}
\right)\in\mathbb{R}^{n+1}\ \right|\ z\in\mathbb{R}^n\right\}\subset\mathbb{R}^{n+1}$ is precompact, is satisfied. For $s=1$, $\mathcal{H}_1$ and $\mathcal{H}$ coincide. For $s=0$, \eqref{hyp_det_condition} in Lemma \ref{lemma_general_PSR_P_3(z,.,.)_estimate} immediately shows that $\mathcal{H}_0$ is a CCPSR manifold. Now consider $s\in(0,1)$ and let $\left(\begin{smallmatrix}
1\\
z
\end{smallmatrix}\right)\in  \left(\mathbb{R}_{>0}\cdot\mathcal{H}_s\right)\cap \left.\left\{\left(\begin{smallmatrix}
1\\
z
\end{smallmatrix}
\right)\in\mathbb{R}^{n+1}\ \right|\ z\in\mathbb{R}^n\right\}$ be arbitrary. For $z=0$, \eqref{hyp_det_condition} in Lemma \ref{lemma_roots_estimates_PSR} is always true. For $z\ne 0$, we will differentiate between the cases $P_3(z)\geq 0$ and $P_3(z)<0$. In the first case, that is $P_3(z)\geq 0$, the estimate \eqref{eqn_positive_root_PSR_estimate} in Lemma \ref{lemma_P_3(y)_bounds} for $f_{sP_3\left(\frac{z}{\|z\|}\right)}(t)=h_s\left(\left(\begin{smallmatrix}1\\ t\frac{z}{\|z\|}\end{smallmatrix}\right)\right)$ (note: $B=P_3\left(\frac{z}{\|z\|}\right)$ and $A=sP_3\left(\frac{z}{\|z\|}\right)$) show that $z\in\mathrm{dom}(\mathcal{H})$ for all $s\in (0,1)$. Hence, using the hyperbolicity of $\mathcal{H}$ we estimate
\begin{equation}\label{eqn_pathconn_estimate_1}
3\langle dz,dz\rangle -9 sP_3(z,dz,dz)+\langle z,dz\rangle^2  >s\left(3\langle dz,dz\rangle -9 P_3(z,dz,dz)+\langle z,dz\rangle^2\right)>0.
\end{equation}
This shows that all points in $\left(\mathbb{R}_{>0}\cdot\mathcal{H}_s\right)\cap \left.\left\{\left(\begin{smallmatrix}
1\\
z
\end{smallmatrix}
\right)\in\mathbb{R}^{n+1}\ \right|\ z\in\mathbb{R}^n\right\}$ with $P_3(z)\geq 0$ satisfy \eqref{hyp_det_condition} in Lemma \ref{lemma_general_PSR_P_3(z,.,.)_estimate} for all $s\in(0,1)$.

Next, consider the case $P_3(z)<0$. This case is a bit more complicated, since the estimate \eqref{eqn_negative_root_PSR_estimate}\footnote{With corresponding values $B=-P_3\left(\frac{z}{\|z\|}\right)$ and $A=-sP_3\left(\frac{z}{\|z\|}\right)$.} in Lemma \ref{lemma_roots_estimates_PSR} for $f_{-sP_3\left(\frac{z}{\|z\|}\right)}(t)=h_s\left(\left(\begin{smallmatrix}1\\ -t\frac{z}{\|z\|}\end{smallmatrix}\right)\right)$ shows that for all $s\in(0,1)$ there exist points in $\left(\mathbb{R}_{>0}\cdot\mathcal{H}_s\right)\cap \left.\left\{\left(\begin{smallmatrix}
1\\
z
\end{smallmatrix}
\right)\in\mathbb{R}^{n+1}\ \right|\ z\in\mathbb{R}^n\right\}$ that are not contained in the set
\begin{equation*}
\left.\left\{\left(\begin{smallmatrix}
1 \\ z
\end{smallmatrix}\right)\in\mathbb{R}^{n+1}\ \right|\ z\in\mathrm{dom}(\mathcal{H})\right\}=\left(\mathbb{R}_{>0}\cdot\mathcal{H}\right)\cap \left.\left\{\left(\begin{smallmatrix}
1\\
z
\end{smallmatrix}
\right)\in\mathbb{R}^{n+1}\ \right|\ z\in\mathbb{R}^n\right\}
\end{equation*}
(see Figure \ref{fig_starshapeproof_plot} for an example).
\begin{figure}[h]%
\centering%
\includegraphics[scale=0.3]{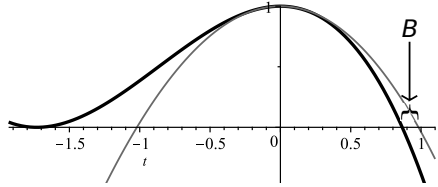}%
\caption{The thick black curve is a plot of $f_{-\frac{2}{3\sqrt{3}}}(t)$ corresponding to $s=1$, that is $\mathcal{H}_1$. The thin grey curve is a plot of $f_{-\frac{1}{10}\cdot\frac{2}{3\sqrt{3}}}(t)$ corresponding to $s=\frac{1}{10}\in(0,1)$, that is $\mathcal{H}_{-\frac{1}{10}}$. The set $B$ is to be understood as points in that are not contained in $\mathrm{dom}\left(\mathcal{H}_1\right)$, but are contained in (a fitting projection of the set) $\left(\mathbb{R}_{>0}\cdot\mathcal{H}_s\right)\cap \left.\left\{\left(1,z
\right)^T\in\mathbb{R}^{n+1}\ \right|\ z\in\mathbb{R}^n\right\}$.}\label{fig_starshapeproof_plot}%
\end{figure}
Consider for $\overline{z}\in \mathbb{R}_{>0}\cdot z$, such that $\overline{z}\in\partial\mathrm{dom}(\mathcal{H})$, and for $t\in[0,1]$ the function $r:[0,1]\to [1,\infty)$ implicitly defined by
\begin{equation*}
F(r,t)=1-r^2\langle\overline{z},\overline{z}\rangle+(1-t)r^3P_3(\overline{z})=0.
\end{equation*}
The condition that $r(t)$ is a positive function and the uniqueness of the positive real root of $r\mapsto F(r,t)$ for all $t\in[0,1]$ show that $F(r,t)=0$ indeed defines $r(t)$ in a unique way, and furthermore that $r(t)$ is smooth for $t\in(0,1)$ and continuous for $t\in[0,1]$ (note: $P_3(\overline{z})<0$).
The map
\begin{align*}
\Psi:\{1\}\times((\mathbb{R}_{>0}\cdot z)\cap\mathrm{dom}(\mathcal{H})) \to&\  (\{1\}\times\mathbb{R}_{>0}\cdot z) \cap\left(\mathbb{R}_{>0}\cdot\mathcal{H}_s\right),\\
\left(\begin{smallmatrix}
1\\ \widetilde{z}
\end{smallmatrix}
\right)\mapsto& \left(\begin{smallmatrix}
1 \\  r(1-s)\widetilde{z}
\end{smallmatrix}
\right),
\end{align*}
is thus a diffeomorphism for all $s\in[0,1]$. Furthermore, $\Psi$ can be continuously extended to be defined on $\{1\}\times\left((\mathbb{R}_{>0}\cdot z)\cap\overline{\mathrm{dom}(\mathcal{H})}\right)$ for all $s\in[0,1]$, with the property that
\begin{equation*}
\Psi(\overline{z})\in\partial \left(\left(\mathbb{R}_{>0}\cdot\mathcal{H}_s\right)\cap \left.\left\{\left(\begin{smallmatrix}
1\\
z
\end{smallmatrix}
\right)\in\mathbb{R}^{n+1}\ \right|\ z\in\mathbb{R}^n\right\}\right).
\end{equation*}
We obtain for the first $t$-derivative of $r=r(t)$ for all $t\in (0,1)$
\begin{align}
&-2r(t)\dot{r}(t)\langle\overline{z},\overline{z}\rangle-r^3(t)P_3(\overline{z})+3(1-t)r^2(t)\dot{r}(t)P_3(\overline{z})=0\notag\\
\Leftrightarrow\quad& \dot{r}(t)=\frac{-r^2(t)P_3(\overline{z})}{2\langle\overline{z},\overline{z}\rangle-3(1-t)r(t)P_3(\overline{z})}.\label{eqn_r_dot_pathconn}
\end{align}
Since $P_3(\overline{z})<0$ and $t\in(0,1)$, this in particular shows that $\dot{r}(t)>0$ for all $t\in(0,1)$. If the considered point $\left(\begin{smallmatrix}
1\\z
\end{smallmatrix}
\right)\in\left(\mathbb{R}_{>0}\cdot\mathcal{H}_s\right)\cap \left.\left\{\left(\begin{smallmatrix}
1\\
z
\end{smallmatrix}
\right)\in\mathbb{R}^{n+1}\ \right|\ z\in\mathbb{R}^n\right\}$ is also an element of $\{1\}\times\mathrm{dom}(\mathcal{H})$, then we can use estimate \eqref{eqn_pathconn_estimate_1} for all $s\in (0,1)$. For $\left(\begin{smallmatrix}
1\\z
\end{smallmatrix}
\right)\in\left(\mathbb{R}_{>0}\cdot\mathcal{H}_s\right)\cap \left.\left\{\left(\begin{smallmatrix}
1\\
z
\end{smallmatrix}
\right)\in\mathbb{R}^{n+1}\ \right|\ z\in\mathbb{R}^n\right\}\setminus(\{1\}\times\mathrm{dom}(\mathcal{H}))$, we want to show that \eqref{hyp_det_condition} holds for all $s\in(0,1)$, i.e. that $3\langle dz,dz\rangle -9 sP_3(z,dz,dz)+\langle z,dz\rangle^2  >0$ for all $s\in(0,1)$. Substituting $s=1-t$ and $z=\mathrm{pr}_{\mathbb{R}^{n}}\Psi(\widetilde{z})=r(t)\widetilde{z}$ with $\widetilde{z}\in\mathrm{dom}(\mathcal{H})$, the latter is equivalent to
\begin{equation}\label{eqn_pathconn_estimate_2}
3\langle dz,dz\rangle -9 (1-t)r(t)P_3(\widetilde{z},dz,dz)+r^2(t)\langle \widetilde{z},dz\rangle^2>0.
\end{equation}
Since $\mathcal{H}$ is a CCPSR manifold by assumption, we already know that
\begin{equation*}
3\langle dz,dz\rangle -9 P_3(\widetilde{z},dz,dz)+\langle \widetilde{z},dz\rangle^2>0
\end{equation*}
for all $\widetilde{z}\in\mathrm{dom}(\mathcal{H})$, cf. Lemma \ref{lemma_general_PSR_P_3(z,.,.)_estimate}. Since $r^2(t)>1$ for all $t\in(0,1)$, proving
\begin{equation*}
3\langle dz,dz\rangle -9 (1-t)r(t)P_3(\widetilde{z},dz,dz)+\langle \widetilde{z},dz\rangle^2>0
\end{equation*}
for all $t\in(0,1)$ and all $\widetilde{z}\in\mathrm{dom}(\mathcal{H})$ will in particular prove \eqref{eqn_pathconn_estimate_2}. Since the estimate $3\langle dz,dz\rangle+\langle \widetilde{z},dz\rangle^2>0$ holds true for all $\widetilde{z}\in\mathbb{R}^{n}$, it suffices to show that $(1-t)r(t)\leq 1$ for all $t\in(0,1)$. The function $(1-t)r(t)$ is non-negative and continuous on $[0,1]$, and positive and smooth on $(0,1)$. For $t=0$, $(1-t)r(t)|_{t=0}=r(0)=1$. Using \eqref{eqn_r_dot_pathconn} yields
\begin{align*}
\frac{\partial}{\partial t}((1-t)r(t)) =-r(t)+(1-t)\dot{r}(t)=\frac{-2r(t)\langle\overline{z},\overline{z}\rangle+2(1-t)r^2(t)P_3(\overline{z})}{2\langle \overline{z},\overline{z}\rangle-3(1-t)r(t)P_3(\overline{z})}<0
\end{align*}
for all $t\in(0,1)$.
Hence, $0\leq(1-t)r(t)\leq 1$ for all $t\in[0,1]$. This thus proves \eqref{eqn_pathconn_estimate_2}.

Summarising, we have shown that $3\langle dz,dz\rangle -9 sP_3(z,dz,dz)+\langle z,dz\rangle^2>0$ for all $\left(\begin{smallmatrix}
1\\ z
\end{smallmatrix}
\right)\in \left(\mathbb{R}_{>0}\cdot\mathcal{H}_s\right)\cap \left.\left\{\left(\begin{smallmatrix}
1\\
z
\end{smallmatrix}
\right)\in\mathbb{R}^{n+1}\ \right|\ z\in\mathbb{R}^n\right\}$ for all $s\in[0,1]$, and thus have proven using Lemma \ref{lemma_general_PSR_P_3(z,.,.)_estimate} that $\mathcal{H}_s$ is a CCPSR manifold for all $s\in[0,1]$. 
\end{proof}
\end{Prop}
An immediate consequence of Proposition \ref{prop_PSR_pathconnectedness} is that we can always find a continuous curve connecting two CCPSR manifolds of the same positive dimension that consists pointwise of CCPSR manifolds. However, we will prove a stronger result in the following Theorem \ref{thm_convex_compact_PSR_generating_set}, from which it will in particular follow how such an aforementioned curve can look like (see Corollary \ref{cor_curve_between_CCPSR}).

\begin{Th}
\label{thm_convex_compact_PSR_generating_set}
Let $n\in\mathbb{N}$ and $h:\mathbb{R}^{n+1}\to\mathbb{R}$ be a cubic homogeneous polynomial of the form \eqref{standard_form_h}, that is $h=x^3-x\langle y,y\rangle+P_3(y)$. Then the connected component $\mathcal{H}$ of the level set $\{h=1\}\subset\mathbb{R}^{n+1}$ that contains the point $\left(\begin{smallmatrix}
x \\ y
\end{smallmatrix}
\right)=\left(\begin{smallmatrix}
1\\0
\end{smallmatrix}
\right)$ is a CCPSR manifold if and only if $\max\limits_{\|z\|=1}P_3(z)\leq\frac{2}{3\sqrt{3}}$.
\begin{proof}
Firstly note that $P_3:\mathbb{R}^n\to\mathbb{R}$ being a cubic homogeneous polynomial and, hence, an odd function implies that $\max\limits_{\|z\|=1}P_3(z)=\max\limits_{\|z\|=1}|P_3(z)|$. Assume that $\mathcal{H}$ is a CCPSR manifold. Then Lemma \ref{lemma_P_3(y)_bounds} shows that $\max\limits_{\|z\|=1}P_3(z)\leq\frac{2}{3\sqrt{3}}$.

Now assume that $\max\limits_{\|z\|=1}P_3(z)\leq \frac{2}{3\sqrt{3}}$. Lemma \ref{lemma_P_3(y)_bounds} only shows that this is a necessary requirement for $\mathcal{H}$ to be a CCPSR manifold.
In order to show that it is also a sufficient condition, we have to show that
\begin{equation}\label{eqn_alternative_hyp_condition_z_v}
3\langle v,v\rangle -9 P_3(z,v,v)+\langle z,v\rangle^2 >0
\end{equation}
for all $\left(\begin{smallmatrix}
1\\ z
\end{smallmatrix}\right)\in \left(\mathbb{R}_{>0}\cdot\mathcal{H}\right)\cap \left.\left\{\left(\begin{smallmatrix}
1\\
z
\end{smallmatrix}
\right)\in\mathbb{R}^{n+1}\ \right|\ z\in\mathbb{R}^n\right\}$ and all $v\in\mathbb{R}^n\setminus\{0\}$, cf. Lemma \ref{lemma_general_PSR_P_3(z,.,.)_estimate}. For $z=0$, \eqref{eqn_alternative_hyp_condition_z_v} is always true.
For $z\ne 0$ and $v=rz$, $r\ne 0$, \eqref{eqn_alternative_hyp_condition_z_v} reads $r^2(3\langle z,z\rangle -9P_3(z)+\langle z,z\rangle^2)>0$. Suppose that there exists a point $\left(\begin{smallmatrix}
1\\
z
\end{smallmatrix}
\right)\in \left(\mathbb{R}_{>0}\cdot\mathcal{H}\right)\cap \left.\left\{\left(\begin{smallmatrix}
1\\
z
\end{smallmatrix}
\right)\in\mathbb{R}^{n+1}\ \right|\ z\in\mathbb{R}^n\right\}\setminus\left\{\left(\begin{smallmatrix}
1\\
0
\end{smallmatrix}
\right)\right\}$, such that $3\langle z,z\rangle -9P_3(z)+\langle z,z\rangle^2=0$. Observe that $\max\limits_{\|z\|=1}P_3(z)\leq\frac{2}{3\sqrt{3}}$ implies
\begin{align*}
3\langle z,z\rangle -9P_3(z)+\langle z,z\rangle^2&=\|z\|^2\left(3-9\|z\|P_3\left(\frac{z}{\|z\|}\right)+\|z\|^2\right)\\
&\geq \|z\|^2\left(3-2\sqrt{3}\|z\|+\|z\|^2\right).
\end{align*}
The map $\|z\|\mapsto 3-2\sqrt{3}\|z\|+\|z\|^2$ is non-negative and its only zero is at $\|z\|=\sqrt{3}$. Hence, for $\|z\|> 0$, $\|z\|^2(3-2\sqrt{3}\|z\|+\|z\|^2)=0$ if and only if $\|z\|=\sqrt{3}$. Since by assumption $\left(\begin{smallmatrix}
1\\
z
\end{smallmatrix}
\right)\in \mathbb{R}_{>0}\cdot\mathcal{H}\cap \left.\left\{\left(\begin{smallmatrix}
1\\
z
\end{smallmatrix}
\right)\in\mathbb{R}^{n+1}\ \right|\ z\in\mathbb{R}^n\right\}\setminus\left\{\left(\begin{smallmatrix}
1\\
0
\end{smallmatrix}
\right)\right\}$, we have $h\left(\left(\begin{smallmatrix}
1\\ z
\end{smallmatrix}
\right)\right)=1-\langle z,z\rangle+P_3(z)>0$. But with $\|z\|=\sqrt{3}$,
\begin{equation*}
h\left(\left(\begin{smallmatrix}
1\\ z
\end{smallmatrix}
\right)\right)=1-\langle z,z\rangle+P_3(z)=-2+3\sqrt{3}P_3\left(\frac{z}{\|z\|}\right)\leq -2+3\sqrt{3}\cdot\frac{2}{3\sqrt{3}}=0,
\end{equation*}
which is a contradiction. We conclude that whenever $z\ne 0$ and $v\ne 0$ are linearly dependent, the estimate \eqref{eqn_alternative_hyp_condition_z_v} holds. Note that this already finishes the proof for $n=1$.

Now assume that $\dim(\mathcal{H})\geq 2$ and let $\left(\begin{smallmatrix}
1\\
z
\end{smallmatrix}
\right)\in \left(\mathbb{R}_{>0}\cdot\mathcal{H}\right)\cap \left.\left\{\left(\begin{smallmatrix}
1\\
z
\end{smallmatrix}
\right)\in\mathbb{R}^{n+1}\ \right|\ z\in\mathbb{R}^n\right\}\setminus\left\{\left(\begin{smallmatrix}
1\\
0
\end{smallmatrix}
\right)\right\}$ be arbitrary. Let $v\in\mathbb{R}^n\setminus\{0\}$, such that $z$ and $v$ are linearly independent. In order to show \eqref{eqn_alternative_hyp_condition_z_v}, choose an orthonormal basis $\{e_1,e_2\}$ of $\mathrm{span}\{z,v\}\subset\mathbb{R}^n$ with respect to $\langle \cdot,\cdot \rangle$ and consider the cubic homogeneous polynomial $\check{h}:\mathbb{R}^3\to\mathbb{R}$ given by
\begin{equation}\label{eqn_h_2dim_restriction}
\check{h}_{(z,v)}\left(\left(\begin{smallmatrix}
x\\ a\\ b
\end{smallmatrix}
\right)\right)=h\left(\left(\begin{smallmatrix}
x\\ ae_1 + be_2
\end{smallmatrix}
\right)\right)=x^3-x(a^2+b^2)+\underbrace{P_3(ae_1+be_2)}_{=:\check{P}_3\left(\left(\begin{smallmatrix}
a\\ b
\end{smallmatrix}
\right)\right) }.
\end{equation}
Let $\check{\mathcal{H}}$ be the connected component of the level set $\left\{\check{h}_{(z,v)}=1\right\}\subset\mathbb{R}^3$ that contains the point $\left(\begin{smallmatrix}
1 \\ 0\\ 0
\end{smallmatrix}
\right)\in\mathbb{R}^3$ and observe that
\begin{align}
&\left(\mathbb{R}_{>0}\cdot\check{\mathcal{H}}\right)\cap \left.\left\{\left(\begin{smallmatrix}
1\\
a\\
b
\end{smallmatrix}
\right)\in\mathbb{R}^{3}\ \right|\ \left(\begin{smallmatrix}
a\\
b
\end{smallmatrix}
\right)\in\mathbb{R}^2\right\}\label{eqn_Hhatcapset}\\
&\cong\left.\left\{\left(\begin{smallmatrix}
1\\
w
\end{smallmatrix}
\right)\ \right|\ w\in \mathrm{span}\{z,v\}\right\}\cap\left(\left(\mathbb{R}_{>0}\cdot\mathcal{H}\right)\cap \left.\left\{\left(\begin{smallmatrix}
1\\
z
\end{smallmatrix}
\right)\in\mathbb{R}^{n+1}\ \right|\ z\in\mathbb{R}^n\right\}\right)\notag
\end{align}
via the linear map $\left(\begin{smallmatrix}
x\\ a\\ b
\end{smallmatrix}
\right)\mapsto\left(\begin{smallmatrix}
x\\ ae_1 + be_2
\end{smallmatrix}
\right)$. Hence, if we prove that the inequality \eqref{hyp_det_condition} in Lemma \ref{lemma_general_PSR_P_3(z,.,.)_estimate} holds for all cubic homogeneous polynomials $\check{h}_{(z,v)}$ of the form \eqref{eqn_h_2dim_restriction} with corresponding set \eqref{eqn_Hhatcapset}, we will also have proven \eqref{hyp_det_condition} in Lemma \ref{lemma_general_PSR_P_3(z,.,.)_estimate} for our considered $h$ with corresponding set $\left(\mathbb{R}_{>0}\cdot\mathcal{H}\right)\cap \left.\left\{\left(\begin{smallmatrix}
1\\
z
\end{smallmatrix}
\right)\in\mathbb{R}^{n+1}\ \right|\ z\in\mathbb{R}^n\right\}$ (recall that for $z$ and $v$ linearly dependent, \eqref{eqn_alternative_hyp_condition_z_v} has already been shown to hold true). Furthermore note that
\begin{equation*}
0\leq\max\limits_{\left\|\left(\begin{smallmatrix}
a\\ b
\end{smallmatrix}
\right)\right\|=1}\check{P}_3\left(\left(\begin{smallmatrix}
a\\ b
\end{smallmatrix}
\right)\right)\leq\max\limits_{\|z\|=1}P_3(z)\leq\frac{2}{3\sqrt{3}}.
\end{equation*}
We thus see that it suffices to prove the statement of this theorem for all considered manifolds $\mathcal{H}$ with the additional restriction $\mathrm{dim}(\mathcal{H})=2$ in order to conclude that it holds true for all $\mathcal{H}$ with $\mathrm{dim}(\mathcal{H})\geq 2$. In the following, we will use the notation used in \cite{CDL} and consider $\mathbb{R}^3$ with linear coordinates $\left(\begin{smallmatrix}
x\\ y\\ z
\end{smallmatrix}
\right)$,
\begin{equation*}
h:\mathbb{R}^3\to\mathbb{R},\quad h=x^3-x(y^2+z^2)+P_3\left(\left(\begin{smallmatrix}
y\\ z
\end{smallmatrix}
\right)\right),
\end{equation*}
such that
\begin{equation*}
\max\limits_{\left\|\left(\begin{smallmatrix}
y\\ z
\end{smallmatrix}
\right)\right\|=1}P_3\left(\left(\begin{smallmatrix}
y\\ z
\end{smallmatrix}
\right)\right)\leq\frac{2}{3\sqrt{3}}.
\end{equation*}
As before, we consider the centro-affine surface $\mathcal{H}$ which is the connected component of the level set $\{h=1\}\subset\mathbb{R}^3$ that contains the point $\left(\begin{smallmatrix}
x\\ y\\ z
\end{smallmatrix}
\right)=\left(\begin{smallmatrix}
1\\ 0\\ 0
\end{smallmatrix}
\right)$, and we want to show that $\mathcal{H}$ is a CCPSR surface (which is equivalent to the condition \eqref{hyp_det_condition} in Lemma \ref{lemma_general_PSR_P_3(z,.,.)_estimate}). For $P_3\equiv 0$, the condition \eqref{hyp_det_condition} in Lemma \ref{lemma_general_PSR_P_3(z,.,.)_estimate} is immediately seen to be true. For $P_3\not\equiv 0$, Proposition \ref{prop_PSR_pathconnectedness} implies that it suffices to prove that $\mathcal{H}$ is a CCPSR surface if $\max\limits_{\left\|\left(\begin{smallmatrix}
y\\ z
\end{smallmatrix}
\right)\right\|=1}P_3\left(\left(\begin{smallmatrix}
y\\ z
\end{smallmatrix}
\right)\right)=\frac{2}{3\sqrt{3}}$, since for all non-vanishing cubic homogeneous polynomials $P_3:\mathbb{R}^2\to\mathbb{R}$ with $\max\limits_{\left\|\left(\begin{smallmatrix}
y\\ z
\end{smallmatrix}
\right)\right\|=1}P_3\left(\left(\begin{smallmatrix}
y\\ z
\end{smallmatrix}
\right)\right)<\frac{2}{3\sqrt{3}}$ we can always choose a positive real number $r>0$, such that $\max\limits_{\left\|\left(\begin{smallmatrix}
y\\ z
\end{smallmatrix}
\right)\right\|=1}rP_3\left(\left(\begin{smallmatrix}
y\\ z
\end{smallmatrix}
\right)\right)=\frac{2}{3\sqrt{3}}$. Consequently assume that $\max\limits_{\left\|\left(\begin{smallmatrix}
y\\ z
\end{smallmatrix}
\right)\right\|=1}P_3\left(\left(\begin{smallmatrix}
y\\ z
\end{smallmatrix}
\right)\right)=\frac{2}{3\sqrt{3}}$. We can, after a possible orthogonal transformation of the $\left(\begin{smallmatrix}
y\\ z
\end{smallmatrix}
\right)$-coordinates (which in particular does not change the form \eqref{standard_form_h} of $h$),
assume that $P_3|_{\left\{\left\|\left(\begin{smallmatrix}
y\\ z
\end{smallmatrix}
\right)\right\|=1\right\}}$ attains its maximum at $\left(\begin{smallmatrix}
y\\ z
\end{smallmatrix}
\right)=\left(\begin{smallmatrix}
1\\ 0
\end{smallmatrix}
\right)$, so that $P_3$ is of the form
\begin{equation*}
P_3\left(\left(\begin{smallmatrix}
y\\ z
\end{smallmatrix}
\right)\right)=\frac{2}{3\sqrt{3}}y^3+kyz^2+\ell z^3.
\end{equation*}
We immediately see that $\ell\in\mathbb{R}$ needs to fulfil $|\ell|\leq\frac{2}{3\sqrt{3}}$. Furthermore, we can without loss of generality assume that $\ell\geq 0$, which can be achieved via $z\mapsto -z$ if necessary.

Now we will show that for all $\ell\in\left[0,\frac{2}{3\sqrt{3}}\right]$, $\max\limits_{\left\|\left(\begin{smallmatrix}
y\\ z
\end{smallmatrix}
\right)\right\|=1}P_3\left(\left(\begin{smallmatrix}
y\\ z
\end{smallmatrix}
\right)\right)=\frac{2}{3\sqrt{3}}$ implies
\begin{equation}\label{eqn_k_not_domain}
k\in\left[-\frac{2}{\sqrt{3}},\frac{1}{\sqrt{3}}\right].
\end{equation}
It will become clear how to use this information in the step thereafter.

First assume $\ell=0$, so that $P_3\left(\left(\begin{smallmatrix}
y\\ z
\end{smallmatrix}
\right)\right)=\frac{2}{3\sqrt{3}}y^3+kyz^2$. We want to determine the positive extremal values and corresponding critical points of $P_3$ when restricted to the set $\left\{\left\|\left(\begin{smallmatrix}
y\\ z
\end{smallmatrix}
\right)\right\|=1\right\}$ aside from $\frac{2}{3\sqrt{3}}$, respectively $\left(\begin{smallmatrix}
y\\ z
\end{smallmatrix}
\right)=\left(\begin{smallmatrix}
1\\ 0
\end{smallmatrix}
\right)$. Suppose that there exists $k>\frac{1}{\sqrt{3}}$ or $k<-\frac{2}{\sqrt{3}}$, such that $\max\limits_{\left\|\left(\begin{smallmatrix}
y\\ z
\end{smallmatrix}
\right)\right\|=1}P_3\left(\left(\begin{smallmatrix}
y\\ z
\end{smallmatrix}
\right)\right)=\frac{2}{3\sqrt{3}}$. In order to find the extremal values of $P_3$ on $\left\{\left\|\left(\begin{smallmatrix}
y\\ z
\end{smallmatrix}
\right)\right\|=1\right\}$ we need to solve $dP_3|_{\left(\begin{smallmatrix}
y\\ z
\end{smallmatrix}
\right)}=r\left\langle\left(\begin{smallmatrix}
y\\ z
\end{smallmatrix}
\right),
\left(\begin{smallmatrix}
dy\\ dz
\end{smallmatrix}
\right)
\right\rangle$, $r\in\mathbb{R}$, that is
\begin{equation}\label{eqn_extremal_points_P3_2dim_ell_0}
\left(\begin{matrix}
\frac{2}{\sqrt{3}}y^2+kz^2 \\
2kyz
\end{matrix}
\right) =\left(\begin{matrix}
ry\\ rz
\end{matrix}
\right),\quad y^2+z^2=1.
\end{equation}
We already know that $\left(\begin{smallmatrix}
y\\ z
\end{smallmatrix}
\right) =\left(\begin{smallmatrix}
1\\ 0
\end{smallmatrix}
\right) $ is an extremal point with $P_3>0$, so we assume now that $z\ne 0$. Then by \eqref{eqn_extremal_points_P3_2dim_ell_0} $r=2ky$, which implies
\begin{equation}\label{eqn_extremal_points_P3_2dim_ell_0_z_ne0}
z^2=\frac{2\sqrt{3}k-2}{\sqrt{3}k}y^2.
\end{equation}
Note that
\begin{equation*}
\frac{2\sqrt{3}k-2}{\sqrt{3}k}>0\quad\forall k\in\mathbb{R}\setminus\left[-\frac{2}{\sqrt{3}},\frac{1}{\sqrt{3}}\right],
\end{equation*}
so \eqref{eqn_extremal_points_P3_2dim_ell_0_z_ne0} will always have non-trivial solutions. For $k>\frac{1}{\sqrt{3}}$ or $k<-\frac{2}{\sqrt{3}}$ consider the two points
\begin{equation*}
\eta_\pm=\sqrt{\frac{\sqrt{3}k}{3\sqrt{3}k-2}}\left(\begin{matrix}
1\\
\pm\sqrt{\frac{2\sqrt{3}k-2}{\sqrt{3}k}}
\end{matrix}
\right)\in\mathbb{R}^2.
\end{equation*}
One quickly checks that $\|\eta_\pm\|=1$
and that $\eta_\pm$ both solve equation \eqref{eqn_extremal_points_P3_2dim_ell_0_z_ne0}. We obtain
\begin{equation*}
P_3(\eta_\pm)=\frac{2k}{3}\sqrt{\frac{\sqrt{3}k}{3\sqrt{3}k-2}}=:\phi(k)
\end{equation*}
and
\begin{equation*}
\partial_k \phi(k)=\frac{2}{3}\sqrt{\frac{\sqrt{3}k}{3\sqrt{3}k-2}}\left(1-\frac{1}{3\sqrt{3}k-2}\right).
\end{equation*}
Furthermore,
\begin{align}
\lim\limits_{k\to \frac{1}{\sqrt{3}},\ k>  \frac{1}{\sqrt{3}}}\phi(k)&=\phi\left(\frac{1}{\sqrt{3}}\right)=\frac{2}{3\sqrt{3}},\notag\\ 
\lim\limits_{k\to -\frac{2}{\sqrt{3}},\ k< - \frac{2}{\sqrt{3}}}\phi(k)&=\phi\left(-\frac{2}{\sqrt{3}}\right)=-\frac{2}{3\sqrt{3}},\label{eqn_P_3_eta_pm_-_lim}
\end{align}
and we see that
\begin{equation*}
\partial_k\phi(k)>0\quad \forall k\in\mathbb{R}\setminus\left[-\frac{2}{\sqrt{3}},\frac{1}{\sqrt{3}}\right].
\end{equation*}
This shows that for $\ell=0$ there exists no $k\in\mathbb{R}\setminus\left[-\frac{2}{\sqrt{3}},\frac{1}{\sqrt{3}}\right]$, such that $\max\limits_{\left\|\left(\begin{smallmatrix}
y\\ z
\end{smallmatrix}
\right)\right\|=1}P_3\left(\left(\begin{smallmatrix}
y\\ z
\end{smallmatrix}
\right)\right)=\frac{2}{3\sqrt{3}}$. 

It remains to consider the case $\ell\in\left(0,\frac{2}{3\sqrt{3}}\right]$. For $P_3\left(\left(\begin{smallmatrix}
y\\ z
\end{smallmatrix}
\right)\right)=\frac{2}{3\sqrt{3}}y^3+kyz^2+\ell z^3$ we get
\begin{equation*}
P_3(\eta_\pm)=\phi(k)\pm\ell\left(\frac{2\sqrt{3}k-2}{3\sqrt{3}k-2}\right)^{\frac{3}{2}}
\end{equation*}
(note that $\|\eta_\pm\|=1$ independently of the chosen $\ell$). Since
\begin{equation*}
\frac{2\sqrt{3}k-2}{3\sqrt{3}k-2}>0\quad \forall k\in\mathbb{R}\setminus\left[-\frac{2}{\sqrt{3}},\frac{1}{\sqrt{3}}\right],
\end{equation*}
it follows that 
\begin{align}
\partial_\ell P_3(\eta_+)&=\left(\frac{2\sqrt{3}k-2}{3\sqrt{3}k-2}\right)^{\frac{3}{2}}>0,\label{eqn_d_ell_P_3_eta+}\\
\partial_\ell P_3(\eta_-)&=-\left(\frac{2\sqrt{3}k-2}{3\sqrt{3}k-2}\right)^{\frac{3}{2}}<0\label{eqn_d_ell_P_3_eta-}
\end{align} 
for all $k\in\mathbb{R}\setminus\left[-\frac{2}{\sqrt{3}},\frac{1}{\sqrt{3}}\right]$. With
\begin{equation*}
P_3(\eta_+)|_{\ell=0}>\frac{2}{3\sqrt{3}}\quad \forall k>\frac{1}{\sqrt{3}}
\end{equation*}
and
\begin{equation*}
P_3(\eta_-)|_{\ell=0}<-\frac{2}{3\sqrt{3}}\quad \forall k<-\frac{2}{\sqrt{3}}
\end{equation*}
we can now conclude that for all $\ell>0$, i.e. in particular for all $\ell\in\left(0,\frac{2}{3\sqrt{3}}\right]$, we have $P_3(\eta_+)>\frac{2}{3\sqrt{3}}$ and $P_3(\eta_-)<-\frac{2}{3\sqrt{3}}$.

Summarising, we have shown that for all $k\in\mathbb{R}\setminus\left[-\frac{2}{\sqrt{3}},\frac{1}{\sqrt{3}}\right]$ and all $\ell\in\left[0,\frac{2}{3\sqrt{3}}\right]$
\begin{equation*}
\max\limits_{\left\|\left(\begin{smallmatrix}
y\\ z
\end{smallmatrix}
\right)\right\|=1}P_3\left(\left(\begin{smallmatrix}
y\\ z
\end{smallmatrix}
\right)\right)>\frac{2}{3\sqrt{3}},
\end{equation*}
which in particular implies that for all $\ell\in\left[0,\frac{2}{3\sqrt{3}}\right]$, $\max\limits_{\left\|\left(\begin{smallmatrix}
y\\ z
\end{smallmatrix}
\right)\right\|=1}P_3\left(\left(\begin{smallmatrix}
y\\ z
\end{smallmatrix}
\right)\right)=\frac{2}{3\sqrt{3}}$ implies $k\in\left[-\frac{2}{\sqrt{3}},\frac{1}{\sqrt{3}}\right]$ as claimed in \eqref{eqn_k_not_domain}.

Next, we will deal with the cases with
\begin{equation*}
k\in\left\{-\frac{2}{\sqrt{3}},\frac{1}{\sqrt{3}}\right\}.
\end{equation*}
Equations \eqref{eqn_P_3_eta_pm_-_lim} and \eqref{eqn_d_ell_P_3_eta-} (for the lower limit $k=-\frac{2}{\sqrt{3}}$) imply that for $k=-\frac{2}{\sqrt{3}}$ and all $\ell\in\left(0,\frac{2}{3\sqrt{3}}\right]$
\begin{equation*}
\max\limits_{\left\|\left(\begin{smallmatrix}
y\\ z
\end{smallmatrix}
\right)\right\|=1}P_3\left(\left(\begin{smallmatrix}
y\\ z
\end{smallmatrix}
\right)\right)>\frac{2}{3\sqrt{3}}.
\end{equation*}
Hence, for $k=-\frac{2}{\sqrt{3}}$, $\ell=0$ is the only allowed value for $\ell\in\left[0,\frac{2}{3\sqrt{3}}\right]$ such that $\max\limits_{\left\|\left(\begin{smallmatrix}
y\\ z
\end{smallmatrix}
\right)\right\|=1}P_3\left(\left(\begin{smallmatrix}
y\\ z
\end{smallmatrix}
\right)\right)=\frac{2}{3\sqrt{3}}$.
The corresponding connected component $\mathcal{H}$ of $\{h=1\}$ is equivalent to the CCPSR surface a) in Theorem \ref{lowdimpsrclassTHM}, cf. equation \eqref{eqn_CDL_a_equiv_h_form} after a sign-flip in $y$ and, hence, in particular a CCPSR manifold. The case $k=\frac{1}{\sqrt{3}}$ is a little more complicated since then $\eta_\pm=\left(\begin{smallmatrix}
1\\ 0
\end{smallmatrix}
\right)$, for which in particular $\partial_\ell P_3(\eta_\pm)$ vanishes, see \eqref{eqn_d_ell_P_3_eta+} and \eqref{eqn_d_ell_P_3_eta-}. Instead of $\eta_\pm$ consider for $\ell\geq 0$ the point
\begin{equation*}
p=\frac{1}{\sqrt{27l^2+1}}\left(\begin{smallmatrix}
1\\ 3\sqrt{3}l
\end{smallmatrix}
\right),\quad \|p\|=1.
\end{equation*}
One can check that $dP_3|_{p}\in\mathbb{R}\langle p,\cdot\rangle$ and
\begin{align*}
P_3(p)=\frac{27\ell^2+2}{3\sqrt{3}\sqrt{27\ell^2+1}},\quad
\partial_\ell (P_3(p))=\left(\frac{3\sqrt{3}\ell}{\sqrt{27\ell^2+1}}\right)^3.
\end{align*}
For $\ell=0$ we have $P_3(p)=\frac{2}{3\sqrt{3}}$ and since $\partial_\ell (P_3(p))>0$ for all $\ell >0$ we deduce that 
\begin{equation*}
\forall\ell>0:\quad P_3(p)>\frac{2}{3\sqrt{3}}.
\end{equation*}
This proves that for $k=\frac{1}{\sqrt{3}}$, $\ell=0$ is the only value allowed for $\ell\in\left[0,\frac{2}{3\sqrt{3}}\right]$. For $k=\frac{1}{\sqrt{3}}$, $\ell=0$, the connected component $\mathcal{H}$ of $\{h=1\}$ is equivalent to the CCPSR surface b) in Theorem \ref{lowdimpsrclassTHM} which follows from equation \eqref{eqn_CDL_b_equiv_h_form}. Hence, $\mathcal{H}$ is a CCPSR manifold.

Now, as stated before, we will use \eqref{eqn_k_not_domain}. Considering \eqref{hyp_det_condition} in Lemma \ref{lemma_general_PSR_P_3(z,.,.)_estimate} for points in the set
\begin{equation*}
 (\mathbb{R}_{>0}\cdot\mathcal{H})\cap \left.\left\{\left(\begin{smallmatrix}
1\\
y\\
0
\end{smallmatrix}
\right)\in\mathbb{R}^{3}\ \right|\ y\in\mathbb{R}\right\}=\left.\left\{\left(\begin{smallmatrix}
1\\
y\\
0
\end{smallmatrix}
\right)\in\mathbb{R}^{3}\ \right|\ y\in\left(-\frac{\sqrt{3}}{2},\sqrt{3}\right)\right\}
\end{equation*}
yields
\begin{align*}
&\left.\left(3(dy^2+dz^2)-9P_3\left(\left(\begin{smallmatrix}
y\\ z
\end{smallmatrix}
\right),\left(\begin{smallmatrix}
dy\\ dz
\end{smallmatrix}
\right),\left(\begin{smallmatrix}
dy\\ dz
\end{smallmatrix}
\right)\right)+(ydy+zdz)^2\right)\right|_{\left(\begin{smallmatrix}
y\\ z
\end{smallmatrix}
\right)=\left(\begin{smallmatrix}
y\\ 0
\end{smallmatrix}
\right)}\\
	&=\left(y-\sqrt{3}\right)^2dy^2+3(1-ky)dz^2.
\end{align*}
With \eqref{eqn_k_not_domain}, that is $k\in\left[-\frac{2}{\sqrt{3}},\frac{1}{\sqrt{3}}\right]$, and $y\in\left(-\frac{\sqrt{3}}{2},\sqrt{3}\right)$
we deduce
\begin{equation*}
\left(y-\sqrt{3}\right)^2dy^2+3(1-ky)dz^2>0.
\end{equation*}
This means that the line segment $\left.\left\{\left(\begin{smallmatrix}
1\\
y\\
0
\end{smallmatrix}
\right)\in\mathbb{R}^{3}\ \right|\ y\in\left(-\frac{\sqrt{3}}{2},\sqrt{3}\right)\right\}\subset\mathbb{R}^3$ consists only of hyperbolic points of $h$, independently of the choice of $k\in\left[-\frac{2}{\sqrt{3}},\frac{1}{\sqrt{3}}\right]$, and the same statement is of course also true if we project it
to $\mathcal{H}$ via point-wise multiplication with $\frac{1}{\sqrt[3]{h\left(\left(1,y,0\right)^T\right)}}$. Since being a hyperbolic point of $h$ is an open condition in $\mathbb{R}^3$, we are in the setting of Proposition \ref{prop_std_form_h} and can transform $h$ with linear transformations of the form \eqref{p_moving_A_matrix} along that set\footnote{Note that this subset of $\mathcal{H}$ is connected and contains the point $(1,0,0)^T$. Furthermore $\partial_xh=3x^2-y$, which is positive at all points $(1,y,0)^T$, $y\in\left(-\frac{\sqrt{3}}{2},\sqrt{3}\right)$. Hence, we can in fact transform $h$ along these points via transformations of the form \eqref{p_moving_A_matrix}.}, that is along
\begin{equation*}
\left.\left\{\left(\begin{smallmatrix}
\frac{1}{\sqrt[3]{h\left(\left(\begin{smallmatrix}1\\ y\\ 0\end{smallmatrix}\right)\right)}}\\
\frac{y}{\sqrt[3]{h\left(\left(\begin{smallmatrix}1\\ y\\ 0\end{smallmatrix}\right)\right)}}\\
0
\end{smallmatrix}
\right)\in\mathbb{R}^{3}\ \right|\ y\in\left(-\frac{\sqrt{3}}{2},\sqrt{3}\right)\right\}\subset\mathcal{H}.
\end{equation*}
In order not to confuse coordinates with parametrisation of said subset of $\mathcal{H}$, we replace $y$ in the above set with the parameter $T\in\left(-\frac{\sqrt{3}}{2},\sqrt{3}\right)$. We start with $E=\mathbbm{1}$ in \eqref{p_moving_A_matrix} and assign for $T\in\left(-\frac{\sqrt{3}}{2},\sqrt{3}\right)$
\begin{equation}\label{eqn_A(T)}
A(T)=\left(
\begin{matrix}
\frac{1}{\sqrt[3]{1-T^2+\frac{2}{3\sqrt{3}}T^3}} & \frac{2T}{\sqrt{3}T+3} & 0 \\
\frac{T}{\sqrt[3]{1-T^2+\frac{2}{3\sqrt{3}}T^3}} & 1 & 0\\
0 & 0 & 1
\end{matrix}
\right)\in\mathrm{GL}(3).
\end{equation}
We obtain
\begin{align}
h\left(A(T)\cdot\left(\begin{smallmatrix}
x\\ y\\ z
\end{smallmatrix}
\right)\right)&=x^3-x\left(\frac{3\left(1-T^2+\frac{2}{3\sqrt{3}}T^3\right)^{\frac{2}{3}}}{\left(T+\sqrt{3}\right)^2}y^2+\frac{1-kT}{\sqrt[3]{1-T^2+\frac{2}{3\sqrt{3}}T^3}}z^2\right)\notag \\
	&\quad+\frac{2\left(1-T^2+\frac{2}{3\sqrt{3}}T^3\right)}{\left(T+\sqrt{3}\right)^3}y^3+\left(k-\frac{2T}{\sqrt{3}T+3}\right)yz^2+\ell z^3.\label{eqn_h_2dim_E_id_trafo}
\end{align}
Note that $1-T^2+\frac{2}{3\sqrt{3}}T^3>0$ and $1-kT>0$ for all $T\in\left(-\frac{\sqrt{3}}{2},\sqrt{3}\right)$ and all $k\in\left[-\frac{2}{\sqrt{3}},\frac{1}{\sqrt{3}}\right]$, which is in accordance with
the positivity of the bilinear form in equation \eqref{p_moving_pos_def_bilform_alternative}.
We have already shown that for $k\in\left\{-\frac{2}{\sqrt{3}},\frac{1}{\sqrt{3}}\right\}$, $\max\limits_{\left\|\left(\begin{smallmatrix}
y\\ z
\end{smallmatrix}
\right)\right\|=1}P_3\left(\left(\begin{smallmatrix}
y\\ z
\end{smallmatrix}
\right)\right)=\frac{2}{3\sqrt{3}}$ implies $\ell=0$ and that the corresponding surfaces $\mathcal{H}$ are indeed CCPSR manifolds. We will from here on assume that $k\in\left(-\frac{2}{\sqrt{3}},\frac{1}{\sqrt{3}}\right)$. Before bringing $h$ in \eqref{eqn_h_2dim_E_id_trafo} to the standard form \eqref{standard_form_h} we will check that we can always solve $k-\frac{2T}{\sqrt{3}T+3}=0$ (the left hand side of which can be viewed as the ``transformed $k$'', up to scale) for $k\in\left(-\frac{2}{\sqrt{3}},\frac{1}{\sqrt{3}}\right)$. We obtain
\begin{align*}
k-\frac{2T}{\sqrt{3}T+3}=0\quad
\Leftrightarrow\quad T=\frac{3k}{2-\sqrt{3}k}=:T(k).
\end{align*}
We have to check that for all $k\in\left(-\frac{2}{\sqrt{3}},\frac{1}{\sqrt{3}}\right)$, $T(k)\in\left(-\frac{\sqrt{3}}{2},\sqrt{3}\right)$. For the limit points $k\in\left\{-\frac{2}{\sqrt{3}},\frac{1}{\sqrt{3}}\right\}$ we have 
\begin{equation*}
T\left(-\frac{2}{\sqrt{3}}\right)=-\frac{\sqrt{3}}{2},\quad T\left(\frac{1}{\sqrt{3}}\right)=\sqrt{3},
\end{equation*}
and
\begin{equation*}
\partial_k T(k)=\frac{6}{(2-\sqrt{3}k)^2}>0
\end{equation*}
for all $k\in \left(-\frac{2}{\sqrt{3}},\frac{1}{\sqrt{3}}\right)$. Hence, 
\begin{equation*}
\forall k\in\left(-\frac{2}{\sqrt{3}},\frac{1}{\sqrt{3}}\right):\quad T(k)\in\left(-\frac{\sqrt{3}}{2},\sqrt{3}\right)
\end{equation*}
as required. Considering \eqref{eqn_h_2dim_E_id_trafo}, we rescale $y$ and $z$ with
\begin{equation}\label{eqn_E(T)}
E(T)=\left(\begin{matrix}
\frac{T+\sqrt{3}}{\sqrt{3}\sqrt[3]{1-T^2+\frac{2}{3\sqrt{3}}T^3}} & \\
 & \frac{\sqrt[6]{1-T^2+\frac{2}{3\sqrt{3}}T^3}}{\sqrt{1-kT}}
\end{matrix}
\right)
\end{equation}
and set $T=T(k)$ to obtain that $h$ is equivalent to
\begin{equation}\label{eqn_h_2dim_fulltrafo}
h\left(A(T)\cdot\left(
\begin{smallmatrix}
1 & \\
 & E(T)
\end{smallmatrix}
\right)\cdot\left(\begin{smallmatrix}
x\\ y\\ z
\end{smallmatrix}
\right)\right)=x^3-x(y^2+z^2)+\frac{2}{3\sqrt{3}}y^3+\ell \frac{\sqrt{1-T(k)^2+\frac{2}{3\sqrt{3}}T(k)^3}}{\left(1-kT(k)\right)^{\frac{3}{2}}}z^3.
\end{equation}
The next question one has to ask is if $k\in\left(-\frac{2}{\sqrt{3}},\frac{1}{\sqrt{3}}\right)$ and $\max\limits_{\left\|\left(\begin{smallmatrix}
y\\ z
\end{smallmatrix}
\right)\right\|=1}P_3\left(\left(\begin{smallmatrix}
y\\ z
\end{smallmatrix}
\right)\right)=\frac{2}{3\sqrt{3}}$ (for the $P_3$-term in $h$, i.e. $P_3\left(\left(\begin{smallmatrix}
y\\ z
\end{smallmatrix}
\right)\right)=\frac{2}{3\sqrt{3}}y^3+kyz^2+\ell z^3$) imply
\begin{equation}\label{eqn_transformed_z3_prefactor_condition}
\ell \frac{\sqrt{1-T(k)^2+\frac{2}{3\sqrt{3}}T(k)^3}}{\left(1-kT(k)\right)^{\frac{3}{2}}}\leq \frac{2}{3\sqrt{3}}
\end{equation}
which is a necessary requirement for
\begin{equation*}
\max\limits_{\left\|\left(\begin{smallmatrix}
y\\ z
\end{smallmatrix}
\right)\right\|=1}\left(\frac{2}{3\sqrt{3}}y^3+\ell \frac{\sqrt{1-T(k)^2+\frac{2}{3\sqrt{3}}T(k)^3}}{\left(1-kT(k)\right)^{\frac{3}{2}}}z^3\right)=\frac{2}{3\sqrt{3}}
\end{equation*}
and thus also a necessary requirement that the transformed cubic in \eqref{eqn_h_2dim_fulltrafo} needs to fulfil so that the corresponding connected component of its level set $\left\{ h\left(A(T)\cdot\left(
\begin{smallmatrix}
1 & \\
 & E(T)
\end{smallmatrix}
\right)\cdot\left(\begin{smallmatrix}
x\\ y\\ z
\end{smallmatrix}
\right)\right)=1\right\}$ which contains the point $\left(\begin{smallmatrix}
x\\ y\\ z
\end{smallmatrix}
\right)=\left(\begin{smallmatrix}
1\\ 0\\ 0
\end{smallmatrix}
\right)$ can be a CCPSR manifold, cf. Corollary \ref{cor_PSR_closedness_condition_P3}. Instead of attempting to calculate the supremum of $\ell \frac{\sqrt{1-T(k)^2+\frac{2}{3\sqrt{3}}T(k)^3}}{\left(1-kT(k)\right)^{\frac{3}{2}}}$ with conditions $k\in\left(-\frac{2}{\sqrt{3}},\frac{1}{\sqrt{3}}\right)$ and $\max\limits_{\left\|\left(\begin{smallmatrix}
y\\ z
\end{smallmatrix}
\right)\right\|=1}P_3\left(\left(\begin{smallmatrix}
y\\ z
\end{smallmatrix}
\right)\right)=\frac{2}{3\sqrt{3}}$ directly, we will choose another way to prove that \eqref{eqn_transformed_z3_prefactor_condition} does, in fact, hold true.

For $k=0$, $h$ is of the form $h=x^3-x(y^2+z^2)+\frac{2}{3\sqrt{3}}y^3+\ell z^3$. Consider for $T\in\left(-\frac{\sqrt{3}}{2},\sqrt{3}\right)$ arbitrary, $A(T)$ and $E(T)$ as in \eqref{eqn_A(T)} and \eqref{eqn_E(T)}, respectively,
\begin{equation}\label{eqn_k0_T_moving}
h\left(A(T)\cdot\left(
\begin{smallmatrix}
1 & \\
 & E(T)
\end{smallmatrix}
\right)\cdot\left(\begin{smallmatrix}
x\\ y\\ z
\end{smallmatrix}
\right)\right)=x^3-x(y^2+z^2)+\frac{2}{3\sqrt{3}}y^3-\frac{2T}{3}yz^2+\ell z^3\sqrt{1-T^2+\frac{2}{3\sqrt{3}}T^3}.
\end{equation}
For the following calculations, we define
\begin{equation*}
P_{(3,\ell,T)}\left(\left(\begin{smallmatrix}
y\\ z
\end{smallmatrix}
\right)\right):=\frac{2}{3\sqrt{3}}y^3-\frac{2T}{3}yz^2+\ell z^3\sqrt{1-T^2+\frac{2}{3\sqrt{3}}T^3}.
\end{equation*}
We will show that
\begin{equation}\label{eqn_k0_ell_toobig_estimate}
\forall \ell>\frac{2}{3\sqrt{3}}\ \forall T\in\left(-\frac{\sqrt{3}}{2},\sqrt{3}\right):\quad
\max\limits_{\left\|\left(\begin{smallmatrix}
y\\ z
\end{smallmatrix}
\right)\right\|=1}P_{(3,\ell,T)}\left(\begin{smallmatrix}
y\\ z
\end{smallmatrix}
\right)>\frac{2}{3\sqrt{3}}
\end{equation}
holds true. To do so we will for $T\in\left(-\frac{\sqrt{3}}{2},\sqrt{3}\right)$ and $\ell=\frac{2}{3\sqrt{3}}$ study a critical point of 
\begin{equation*}
P_{\left(3,\frac{2}{3\sqrt{3}},T\right)}\left(\left(\begin{smallmatrix}
y\\ z
\end{smallmatrix}
\right)\right)=\frac{2}{3\sqrt{3}}y^3-\frac{2T}{3}yz^2+\frac{2}{3\sqrt{3}}z^3\sqrt{1-T^2+\frac{2}{3\sqrt{3}}T^3}
\end{equation*}
on the set $\left\{\left\|\left(\begin{smallmatrix}
y\\ z
\end{smallmatrix}
\right)\right\|=1\right\}$, namely the point
\begin{equation*}
\left(\begin{matrix}
y\\ z
\end{matrix}
\right)=\frac{1}{T+\sqrt{3}}\left(\begin{matrix}
-T\\
 \sqrt{2\sqrt{3}T+3} 
\end{matrix}
\right)=:\zeta.
\end{equation*}
Note that $\zeta$ is well-defined for all $T\in\left(-\frac{\sqrt{3}}{2},\sqrt{3}\right)$, and it is indeed a critical point of $P_{\left(3,\frac{2}{3\sqrt{3}},T\right)}\left(\left(\begin{smallmatrix}
y\\ z
\end{smallmatrix}
\right)\right)$. Using the factorisation $1-T^2+\frac{2}{3\sqrt{3}}T^3=\frac{2}{3\sqrt{3}}\left(T-\sqrt{3}\right)^2\left(T+\frac{\sqrt{3}}{2}\right)$ and $T-\sqrt{3}<0$ for all $T\in\left(-\frac{\sqrt{3}}{2},\sqrt{3}\right)$, we find
\begin{align*}
\left.dP_{\left(3,\frac{2}{3\sqrt{3}},T\right)}\right|_\zeta&=\left.\left(\left(\frac{2}{\sqrt{3}}y^2-\frac{2T}{3}z^2\right)dy+\left(-\frac{4T}{3}yz+\frac{2}{\sqrt{3}}z^2\sqrt{1-T^2+\frac{2}{3\sqrt{3}}T^3}\right)dz\right)\right|_\zeta\\
	&=\frac{2}{\sqrt{3}}\left\langle\zeta,\left(\begin{smallmatrix}
	dy\\ dz
	\end{smallmatrix}
	\right)\right\rangle.
\end{align*}
The corresponding critical value is given by $P_{\left(3,\frac{2}{3\sqrt{3}},T\right)}(\zeta)=\frac{2}{3\sqrt{3}}$,
independent of $T\in\left(-\frac{\sqrt{3}}{2},\sqrt{3}\right)$. Note that $dz(\zeta)>0$ for all $T\in\left(-\frac{\sqrt{3}}{2},\sqrt{3}\right)$ and consider the derivative
\begin{equation*}
\partial_\ell \left(P_{(3,\ell,T)}\left(\begin{smallmatrix}
y\\ z
\end{smallmatrix}
\right)\right)=z^3\sqrt{1-T^2+\frac{2}{3\sqrt{3}}}.
\end{equation*}
Hence, $\partial_\ell \left(P_{(3,\ell,T)}\left(\begin{smallmatrix}
y\\ z
\end{smallmatrix}
\right)\right)>0$ for all $T\in\left(-\frac{\sqrt{3}}{2},\sqrt{3}\right)$ and all $z>0$, in particular for $z=dz(\zeta)$. We conclude that \eqref{eqn_k0_ell_toobig_estimate} holds true.

We can now use \eqref{eqn_k0_ell_toobig_estimate} to show that \eqref{eqn_transformed_z3_prefactor_condition} holds true for all $k\in\left(-\frac{2}{\sqrt{3}},\frac{1}{\sqrt{3}}\right)$. For $\ell=0$ equation \eqref{eqn_transformed_z3_prefactor_condition} is automatically true independently of the chosen $k\in\left(-\frac{2}{\sqrt{3}},\frac{1}{\sqrt{3}}\right)$. Suppose that there exist $k\in\left(-\frac{2}{\sqrt{3}},\frac{1}{\sqrt{3}}\right)$ and $\ell\in\left(0,\frac{2}{3\sqrt{3}}\right]$, with corresponding polynomial $h=x^3-x(y^2+z^2)+\frac{2}{3\sqrt{3}}y^3+kyz^2+\ell z^3$,
fulfilling
\begin{equation}\label{eqn_k_l_final_step_assumption_maxP3}
\max\limits_{\left\|\left(\begin{smallmatrix}
y\\ z
\end{smallmatrix}
\right)\right\|=1}P_3\left(\left(\begin{smallmatrix}
y\\ z
\end{smallmatrix}
\right)\right)=\max\limits_{\left\|\left(\begin{smallmatrix}
y\\ z
\end{smallmatrix}
\right)\right\|=1}\left(\frac{2}{3\sqrt{3}}y^3+kyz^2+\ell z^3\right)=\frac{2}{3\sqrt{3}},
\end{equation}
such that for $T=T(k)=\frac{3k}{2-\sqrt{3}k}$
\begin{equation*}
\ell \frac{\sqrt{1-T(k)^2+\frac{2}{3\sqrt{3}}T(k)^3}}{\left(1-kT(k)\right)^{\frac{3}{2}}}>\frac{2}{3\sqrt{3}},
\end{equation*}
which precisely means that \eqref{eqn_transformed_z3_prefactor_condition} does not hold true for the chosen $k$, $\ell$. Combining \eqref{eqn_h_2dim_fulltrafo} and \eqref{eqn_k0_T_moving}, one obtains that $h$ is equivalent to
\begin{equation*}
\widetilde{h}=x^3-x(y^2+z^2)+\frac{2}{3\sqrt{3}}y^3-\frac{2\widetilde{T}}{3}yz^2+\ell z^3 \frac{\sqrt{1-T(k)^2+\frac{2}{3\sqrt{3}}T(k)^3}}{\left(1-kT(k)\right)^{\frac{3}{2}}}\sqrt{1-\widetilde{T}^2+\frac{2}{3\sqrt{3}}\widetilde{T}^3}
\end{equation*}
for all $\widetilde{T}\in\left(-\frac{\sqrt{3}}{2},\sqrt{3}\right)$. Furthermore, \eqref{eqn_k0_ell_toobig_estimate} implies that for all $\widetilde{T}\in\left(-\frac{\sqrt{3}}{2},\sqrt{3}\right)$:
\begin{equation}\label{eqn_Tk_tildeT_combined_estimate}
\max\limits_{\left\|\left(\begin{smallmatrix}
y\\ z
\end{smallmatrix}
\right)\right\|=1}\left(\frac{2}{3\sqrt{3}}y^3-\frac{2\widetilde{T}}{3}yz^2+\ell z^3\frac{\sqrt{1-T(k)^2+\frac{2}{3\sqrt{3}}T(k)^3}}{\left(1-kT(k)\right)^{\frac{3}{2}}}\sqrt{1-\widetilde{T}^2+\frac{2}{3\sqrt{3}}\widetilde{T}^3}\right)>\frac{2}{3\sqrt{3}}.
\end{equation}
The above estimate \eqref{eqn_Tk_tildeT_combined_estimate} must thus in particular hold for $\widetilde{T}=-\frac{3k}{2}=:\widetilde{T}(k)$ (note that $\widetilde{T}(k)\in\left(-\frac{\sqrt{3}}{2},\sqrt{3}\right)$ for all $k\in\left(-\frac{2}{\sqrt{3}},\frac{1}{\sqrt{3}}\right)$). But
\begin{align*}
& \frac{2}{3\sqrt{3}}y^3-\frac{2\widetilde{T}(k)}{3}yz^2+\ell z^3\frac{\sqrt{1-T(k)^2+\frac{2}{3\sqrt{3}}T(k)^3}}{\left(1-kT(k)\right)^{\frac{3}{2}}}\sqrt{1-\widetilde{T}(k)^2+\frac{2}{3\sqrt{3}}\widetilde{T}(k)^3}\\
&=\frac{2}{3\sqrt{3}}y^3+kyz^2+\ell z^3,
\end{align*}
which implies that \eqref{eqn_Tk_tildeT_combined_estimate} for $\widetilde{T}=\widetilde{T}(k)$ is a contradiction to the assumption \eqref{eqn_k_l_final_step_assumption_maxP3}. We conclude that \eqref{eqn_transformed_z3_prefactor_condition} holds true.

In order to complete the proof of this theorem it thus suffices to show that for all $\ell\in\left[0,\frac{2}{3\sqrt{3}}\right]$ and corresponding polynomial $h_\ell:=x^3-x(y^2+z^2)+\frac{2}{3\sqrt{3}}y^3+\ell z^3$, the connected component $\mathcal{H}_\ell\subset\{h_\ell=1\}$ that contains the point $\left(\begin{smallmatrix}
x\\ y\\ z
\end{smallmatrix}\right)=\left(\begin{smallmatrix}
1\\ 0\\ 0
\end{smallmatrix}\right)$ is a CCPSR manifold. Define the $P_3$-part of $h_\ell$ as $P_{(3,\ell)}\left(\left(\begin{smallmatrix}
y\\ z
\end{smallmatrix}
\right)\right):=\frac{2}{3\sqrt{3}}y^3+\ell z^3$. One can easily check that 
\begin{equation*}
\frac{2}{3\sqrt{3}}\leq \max\limits_{\left\|\left(\begin{smallmatrix}
y\\ z
\end{smallmatrix}
\right)\right\|=1}P_{(3,\ell)}\left(\left(\begin{smallmatrix}
y\\ z
\end{smallmatrix}
\right)\right)
	\leq \max\limits_{\left\|\left(\begin{smallmatrix}
y\\ z
\end{smallmatrix}
\right)\right\|=1}\frac{2}{3\sqrt{3}}\left(|y|^3+ |z|^3\right)
	=\max\left\{\frac{2}{3\sqrt{3}},\frac{\sqrt{2}}{3\sqrt{3}}\right\}
	=\frac{2}{3\sqrt{3}},
\end{equation*}
which shows that $\max\limits_{\left\|\left(\begin{smallmatrix}
y\\ z
\end{smallmatrix}
\right)\right\|=1}P_{(3,\ell)}\left(\left(\begin{smallmatrix}
y\\ z
\end{smallmatrix}
\right)\right)=\frac{2}{3\sqrt{3}}$ as required. We use the linear transformation
\begin{equation*}
B=\left(\begin{smallmatrix}
1 & 0 & 0 \\
\sqrt{3} & 1 & 0 \\
0 & 0 & 1
\end{smallmatrix}
\right)\in\mathrm{GL}(3)
\end{equation*}
and transform $h_\ell$ to
\begin{equation*}
\breve{h}_\ell\left(\left(\begin{smallmatrix}
x\\ y\\ z
\end{smallmatrix}
\right)\right):=h_\ell\left(B\cdot\left(\begin{smallmatrix}
x\\ y\\ z
\end{smallmatrix}
\right)\right)=x(y^2-z^2)+\frac{2}{3\sqrt{3}}y^3+\ell z^3.
\end{equation*}
In the new coordinates, $\breve{\mathcal{H}}_\ell:=B^{-1}\left(\mathcal{H}_\ell\right)\subset\left\{\breve{h}=1\right\}$ is given by
\begin{equation*}
\breve{\mathcal{H}}=\left.\left\{\left(\begin{matrix}
\frac{1-\frac{2}{3\sqrt{3}}y^3-\ell z^3}{y^2-z^2} \\ y \\ z
\end{matrix}
\right)\ \right|\ y<0,\ y^2>z^2\right\}.
\end{equation*}
This follows easily from $B\cdot\left(\begin{smallmatrix}
1\\ -\sqrt{3}\\ 0
\end{smallmatrix}
\right)=\left(\begin{smallmatrix}
1\\ 0 \\ 0
\end{smallmatrix}
\right)\in\mathcal{H}_\ell$ and that $x\to \infty$ for all sequences in $\overline{\{y< 0,\ y^2> z^2\}}=\{y\leq 0,\ y^2\geq z^2\}$ that converge to a point in $\partial\left\{y<0,\ y^2>z^2\right\}=\{y\geq 0,\ y^2=z^2\}$. The latter follows from $\frac{2}{3\sqrt{3}} +\ell\leq\frac{4}{3\sqrt{3}}<1$ for all $\ell\in\left[0,\frac{2}{3\sqrt{3}}\right]$. We know that $\left(\begin{smallmatrix}
1\\ -\sqrt{3}\\ 0
\end{smallmatrix}
\right)=B^{-1}\cdot\left(\begin{smallmatrix}
1\\ 0 \\ 0
\end{smallmatrix}
\right)\in\breve{\mathcal{H}}_\ell$ is always a hyperbolic point of $\breve{h}_\ell$ for all $\ell\in\left[0,\frac{2}{3\sqrt{3}}\right]$. Hence, in order to show that $\breve{\mathcal{H}}_\ell$ consists only of hyperbolic points of $\breve{h}_\ell$, it suffices to show that
\begin{align*}
&\det\left(-\frac{1}{2}\partial^2\breve{h}_\ell\right)=\det\left(\begin{matrix}
0 & -y & z\\
-y & -x-\frac{2}{\sqrt{3}}y & 0\\
z & 0 & x-3\ell z
\end{matrix}
\right)\\
	&=\frac{1}{y^2-z^2}\Bigg(\underbrace{\frac{2}{3\sqrt{3}}y^5-\ell z^5 + 3\ell y^4 z - \frac{2}{\sqrt{3}}yz^4 + \frac{4}{3\sqrt{3}}y^3z^2-2\ell y^2 z^3}_{=:R_\ell(y,z)} -y^2+z^2 \Bigg)<0
\end{align*}
for all $\left(\begin{smallmatrix}
y\\ z
\end{smallmatrix}
\right)\in\{y<0,\ y^2>z^2\}$. The prefactor $\frac{1}{y^2-z^2}$ is always positive if $y^2>z^2$, and the term $-y^2+z^2$ is always negative. Hence, it suffices to show $R_\ell(y,z)\leq 0$ for all $\left(\begin{smallmatrix}
y\\ z
\end{smallmatrix}
\right)\in\overline{\{y<0,\ y^2>z^2\}}$. We calculate
\begin{equation*}
R_\ell(y,\pm y)=y^5\left(\frac{2}{3\sqrt{3}}\mp \ell \pm 3\ell -\frac{2}{\sqrt{3}}+\frac{4}{3\sqrt{3}}\mp 2\ell\right)=0,
\end{equation*}
which implies that $R_\ell(y,z)$ vanishes on $\partial\{y<0,\ y^2>z^2\}$. Since the set $\{y<0,\ y^2>z^2\}$ is a cone and $R_\ell(y,z)$ is for all $\ell\in\left[0,\frac{2}{3\sqrt{3}}\right]$ a homogeneous polynomial of degree $5$, it only remains to check that
\begin{equation}\label{eqn_R_ell_condition}
\forall s\in(-1,1)\ \forall \ell\in\left[0,\frac{2}{3\sqrt{3}}\right]: \quad R_\ell(-1,s)\leq 0.
\end{equation}
We find that $s=1$ and $s=-1$ are roots of $R_\ell(-1,s)$ for all $\ell\in\left[0,\frac{2}{3\sqrt{3}}\right]$, which allows us to consider
\begin{equation*}
N_\ell(s):=\frac{R_\ell(-1,s)}{(s-1)(s+1)}=\frac{R_\ell(-1,s)}{s^2-1}=-\ell s^3+\frac{2}{\sqrt{3}}s^2-3\ell s + \frac{2}{3\sqrt{3}}.
\end{equation*}
The condition \eqref{eqn_R_ell_condition} is equivalent to
\begin{equation}\label{eqn_N_ell_condition}
\forall s\in(-1,1)\ \forall\ell\in\left[0,\frac{2}{3\sqrt{3}}\right] :\quad N_\ell(s)\geq 0.
\end{equation}
This motivates checking solutions of $N_\ell(s)=0$. We get
\begin{equation*}
N_\ell(s)=0\quad \Leftrightarrow\quad \ell=\frac{2}{3\sqrt{3}}\cdot \frac{3s^2+1}{s(s^2+3)}.
\end{equation*}
We will show that $M(s):=\frac{3s^2+1}{s(s^2+3)}\not\in[0,1]$ for all $s\in(-1,1)$, which implies that there exists no pair $(\ell,s)\in\left[0,\frac{2}{3\sqrt{3}}\right]\times (-1,1)$, such that $N_\ell(s)=0$. Since $N_0(1)=\frac{2}{\sqrt{3}}>0$, this will then shows that $N_\ell(s)>0$ for all $(\ell,s)\in\left[0,\frac{2}{3\sqrt{3}}\right]\times (-1,1)$ and in particular imply \eqref{eqn_N_ell_condition}. We see that
\begin{equation*}
\mathrm{sgn}(M(s))=\left\{
\begin{tabular}{rl}
$1,$ & $\forall s>0,$\\
$-1,$ & $\forall s<0,$
\end{tabular}
\right.
\end{equation*}
which implies that we can reduce our studies to $s\in[0,1)$. The first derivative of $M(s)$ is easily seen to fulfil
\begin{equation}\label{eqn_dM_s_estimate}
\partial_s M(s)=\frac{-3(s^4-2s^2+1)}{s^2(s^2+3)^2}<0
\end{equation}
for all $s\in(0,1)$. Furthermore
\begin{equation}\label{eqn_M_s_limits}
\lim\limits_{s\to 0,\; s>0} M(s)=\infty.
\end{equation}
The estimate \eqref{eqn_dM_s_estimate} and the limit \eqref{eqn_M_s_limits} imply
\begin{equation*}
\forall s\in(0,1):\quad M(s)>M(1)=1.
\end{equation*}
Hence, the equation $M(s)=1$ has no solutions in the half-open interval $[0,1)$. We conclude that \eqref{eqn_N_ell_condition} holds true.

Summarising, we have proven that for all $\ell \in\left[0,\frac{2}{3\sqrt{3}}\right]$, $\breve{\mathcal{H}}_\ell$ is a CCPSR manifold of dimension $2$, which implies the same statement for $\mathcal{H}_\ell$. This finishes the proof of Theorem~\ref{thm_convex_compact_PSR_generating_set}.
\end{proof}
\end{Th}

Now we have all necessary results at hand to complete the proof of our main result Theorem \ref{thm_Cn}.
\begin{proof}[Proof of Theorem \ref{thm_Cn}:]
The existence of $\widetilde{h}\in\mathcal{C}_n$ follows from Proposition \ref{prop_std_form_h} and Theorem \ref{thm_convex_compact_PSR_generating_set}. The statement $\widetilde{h}\in\partial\mathcal{C}_n=\left\{x^3-x\langle y,y\rangle +P_3(y)\ \left|\ \max\limits_{\|z\|=1} P_3(z)=\frac{2}{3\sqrt{3}}\right\}\right.$ if and only if the initial $\mathcal{H}$ does not have regular boundary behaviour follows from Lemma \ref{lemma_singular_PSR_max_P_3} and Theorem \ref{thm_irregularity_implications_CCPSR}.
The convexity of $\mathcal{C}_n$ follows from
	\begin{equation*}
		\max\limits_{\|z\|=1}\left((1-t)P_3(z)+t\widetilde{P_3}(z)\right) \leq (1-t)\left(\max\limits_{\|z\|=1}P_3(z)\right)+t\left(\max\limits_{\|z\|=1}\widetilde{P_3}(z)\right).
	\end{equation*}

It remains to show that $\mathcal{C}_n\subset \left\{x^3-x\langle y,y\rangle + P_3(y)\ \left|\ P_3\in\mathrm{Sym}^3\left(\mathbb{R}^n\right)^*\right\}\right.\subset \mathrm{Sym}^3\left(\mathbb{R}^{n+1}\right)^*$ is compact and that $\partial\mathcal{C}_n\subset\mathrm{Sym}^3\left(\mathbb{R}^{n+1}\right)^*$ is a continuous submanifold. For compactness we need to show that the condition $\max\limits_{\|z\|=1} P_3(z)\leq\frac{2}{3\sqrt{3}}$ automatically implies that $P_3(\cdot,\cdot,\cdot)$ viewed as a symmetric $3$-tensor is bounded entry-wise, and we need to show that $\mathcal{C}_n$ is closed in the subspace topology\footnote{With respect to the topology induced by the linear homeomorphy of $\mathrm{Sym}^3\left(\mathbb{R}^{n+1}\right)^*$ and $\mathbb{R}^{\frac{n^3+6n^2+11n+6}{6}}$. Note that said topology on $\mathrm{Sym}^3\left(\mathbb{R}^{n+1}\right)^*$ does not depend on the choice of the linear homeomorphism.}. This is equivalent to showing that all third derivatives of $P_3(z)$ are bounded on $\{\|z\|=1\}$. This follows from the fact that for all $P_3$ fulfilling $\max\limits_{\|z\|=1} P_3(z)=\frac{2}{3\sqrt{3}}$, the corresponding $h=x^3-x\langle y,y\rangle +P_3(y)\in\mathcal{C}_n$ defines a CCPSR manifold and, hence, we can use Corollary \ref{corollary_PSR_boundary_domain_estimates} and Proposition \ref{prop_P_3(z,.,.)_boundaries_on_del_dom(H)} to conclude that each entry in $P_3(\cdot,\cdot,\cdot)$ is indeed bounded. $\mathcal{C}_n$ being closed follows from the continuity of $\max\limits_{\|z\|=1} P_3(z)$ with respect to the prefactors of the monomials in $P_3(y)$, or equivalently the prefactors in the corresponding symmetric $3$-tensor $P_3(\cdot,\cdot,\cdot)$. We conclude that $\mathcal{C}_n\subset \mathrm{Sym}^3\left(\mathbb{R}^{n+1}\right)^*$ is compact in the subspace topology. The fact that $\partial\mathcal{C}$ is a continuous hypersurface in $\mathrm{Sym}^3\left(\mathbb{R}^{n+1}\right)^*$ also follows from the continuity of the map $P_3\mapsto \max\limits_{\|z\|=1} P_3(z)$. However, note that this map is for $n\geq 2$ in general not smooth, or even differentiable.
To see this, consider the one-parameter family $P^t_3(y)=y_1^3+ty_2^3$, $t\in\left[0,\frac{2}{3\sqrt{3}}\right]$, in $\mathrm{Sym}^3\left(\mathbb{R}^{n+1}\right)^*$ and observe that
\begin{equation*}
t\mapsto\max\limits_{\|z\|=1}P^t_3(z)=\max\limits_{\|z\|=1}\left(z_1^3+tz_2^3\right)=\left\{
\begin{tabular}{ll}
$1$,& $0\leq t\leq 1$,\\
$t$,&1 $\leq t \leq \frac{2}{3\sqrt{3}}$,
\end{tabular}
\right.
\end{equation*}
does depend only continuously on $t$ and is not continuously differentiable at $t=1$.
\end{proof}

\begin{rem}
\label{rem_Cn_comparison}
Note that for any compact set $C\subset \mathrm{Sym}^3\left(\mathbb{R}^{n+1}\right)^*$
that contains an open neighbourhood of $0$,
and any given CCPSR manifold $\mathcal{H}\subset\{h=1\}$, we can always choose $r>0$, such that $rh\in C$. Then $\mathcal{H}$ is equivalent to $r^{-\frac{1}{3}}\cdot\mathcal{H}\subset\{rh=1\}$. This shows that one can choose a generating set for the moduli set of $n$-dimensional CCPSR manifolds that is contained in a compact set $C$ and, hence, bounded. It was however until now for $n\geq 2$ not known whether one can choose a \textsf{compact generating set}, like $\mathcal{C}_n$ in Theorem \ref{thm_Cn}. For $n=1$ it was already shown in \cite[Cor.\,4]{CHM} that the moduli set of CCPSR curves is generated by the set $\{x^2y,\,x(x^2-y^2)\}\subset\mathrm{Sym}^3\left(\mathbb{R}^2\right)^{*}$, which is a compact set. One can show that $x^2y$ is equivalent to $x^3-xy^2+\frac{2}{3\sqrt{3}}y^3$. By comparing with $\mathcal{C}_1=\left\{x^3-xy^2+Ly^3\ \left|\ |L|\leq \frac{2}{3\sqrt{3}}\right\}\right.$, we see that $x(x^2-y^2)=x^3-xy^2$ is an inner point of $\mathcal{C}_1$ and $x^3-xy^2+\frac{2}{3\sqrt{3}}y^3$ is one of the two points in $\partial\mathcal{C}_1$.
\end{rem}

Theorem \ref{thm_Cn} in particular implies the following property of the moduli set of $n$-dimensional CCPSR manifolds (cf. Definition \ref{def_moduli_space_CCGPSR_mfs}).

\begin{Cor}
\label{cor_curve_between_CCPSR}
For $n\in\mathbb{N}$ fixed, let $h,\widetilde{h}\in\mathcal{C}_n$ and let $\mathcal{H}\subset\{h=x^3-x\langle y,y\rangle +P_3(y)=1\}$, respectively $\widetilde{\mathcal{H}}\subset\{\widetilde{h}=x^3-x\langle y,y\rangle +\widetilde{P_3}(y)=1\}$, denote the corresponding CCPSR manifolds containing the point $\left(\begin{smallmatrix}
x\\ y
\end{smallmatrix}
\right)=\left(\begin{smallmatrix}
1\\ 0
\end{smallmatrix}
\right)$. Then the smooth curve
\begin{equation*}
\gamma:[0,1]\to\mathcal{C}_n\subset\mathrm{Sym}^3\left(\mathbb{R}^{n+1}\right)^*,\quad \gamma(t)=(1-t)h+t\widetilde{h},
\end{equation*}
defines an $n$-dimensional CCPSR manifold $\mathcal{H}_{t}\subset\left\{\gamma(t)=(1-t)h+t\widetilde{h}=1\right\}$ as the connected component containing $\left(\begin{smallmatrix}
1\\ 0
\end{smallmatrix}
\right)$ for all $t\in[0,1]$, and $\mathcal{H}_0=\mathcal{H}$, $\mathcal{H}_1=\widetilde{\mathcal{H}}$.
\end{Cor}

\begin{Ex}\label{ex_curves_CCPSRs}
Corollary \ref{cor_curve_between_CCPSR} does not say anything about whether the CCPSR manifolds $\mathcal{H}_t$ are pairwise equivalent or not.
\begin{enumerate}[i)]
\item
	For an example of a curve $\gamma$ as in Corollary \ref{cor_curve_between_CCPSR} where the CCPSR manifolds $\mathcal{H}_t$ are pairwise inequivalent consider the family f) in Theorem \ref{lowdimpsrclassTHM} which corresponds to hyperbolic positive level sets of Weierstra{\ss} polynomials with positive discriminant \cite[Lem.\,1]{CDL}. In Example \ref{example_PSRsurfaces_standard_form}, equation \eqref{eqn_CDL_f_equiv_h_form}, we calculated the standard form for these polynomials and have seen that this family interpolates between the CCPSR surfaces Thm. \ref{lowdimpsrclassTHM} d) and e).
	We see that after setting $b=1-2t^2$ and swapping the coordinates $y$ and $z$,
	the family Thm. \ref{lowdimpsrclassTHM} f) is realized as a curve of CCPSR surfaces with
		\begin{equation*}
			\gamma(t)=x^3-x(y^2+z^2)+\frac{2t}{3\sqrt{3}}y^3,\quad t\in(0,1),
		\end{equation*}
	connecting the CCPSR surface Thm. \ref{lowdimpsrclassTHM} d) corresponding to $P_3\left(\left(\begin{smallmatrix}y\\ z\end{smallmatrix}\right)\right)=0$ respectively $t=0$, cf. equation \eqref{eqn_CDL_d_equiv_h_form}, and the CCPSR surface Thm. \ref{lowdimpsrclassTHM} e) corresponding to $P_3\left(\left(\begin{smallmatrix}y\\ z\end{smallmatrix}\right)\right)=\frac{2}{3\sqrt{3}}y^3$ respectively $t=1$, cf. equation \eqref{eqn_b_alt_limit}.
\item
	Next, consider $P_3\left(\left(\begin{smallmatrix}y\\ z\end{smallmatrix}\right)\right)=\frac{2}{3\sqrt{3}}y^3+\frac{1}{\sqrt{3}}yz^2$ and $\widetilde{P_3}\left(\left(\begin{smallmatrix}y\\ z\end{smallmatrix}\right)\right)=\frac{2}{3\sqrt{3}}y^3-\frac{2}{\sqrt{3}}yz^2$. The corresponding CCPSR surfaces $\mathcal{H}$ and $\widetilde{\mathcal{H}}$ are equivalent to Thm. \ref{lowdimpsrclassTHM} a) and Thm. \ref{lowdimpsrclassTHM} b), respectively, cf. Example \ref{example_PSRsurfaces_standard_form} equations \eqref{eqn_CDL_a_equiv_h_form} and \eqref{eqn_CDL_b_equiv_h_form}. One can further show that $\mathcal{H}$ is a homogeneous space with vanishing scalar curvature and $\widetilde{\mathcal{H}}$ is a homogeneous space with scalar curvature given by $-9/4$. The curve $\gamma$ as in Corollary \ref{cor_curve_between_CCPSR} is given by
		\begin{equation*}
			\gamma(t)=x^3-x(y^2+z^2)+\frac{2}{3\sqrt{3}}y^3+\frac{1-3t}{\sqrt{3}}yz^2.
		\end{equation*}
	The associated curve of CCPSR surfaces $\mathcal{H}_t$ contains precisely $3$ inequivalent CCPSR surfaces. To show this, it suffices to show that for each $t\in(0,1)$, $\mathcal{H}_t$ is equivalent to the CCPSR surface in Thm. \ref{lowdimpsrclassTHM} e) which coincides by Example \ref{example_PSRsurfaces_standard_form} equation \eqref{eqn_CDL_e_equiv_h_form} with $\mathcal{H}_{1/2}$. Consider
		\begin{equation*}
			\mathcal{H}_{1/3}\subset\left\{h_{1/3}:=x^3-x(y^2+z^2)+\frac{2}{3\sqrt{3}}y^3=1\right\}.
		\end{equation*}
	We see that $(t,0)^T\in\mathrm{dom}(\mathcal{H}_{1/3})$ for all $r\in\left(-\frac{\sqrt{3}}{2},\sqrt{3}\right)$. Next, we construct a linear transformation $A(p)$ of the form \eqref{p_moving_A_matrix} for $p=p(r)=\frac{1}{\sqrt[3]{h_{1/3}\left((1,r,0)^T\right)}}\left(\begin{smallmatrix}1\\ r\\ 0\end{smallmatrix}\right)$. We leave it as an exercise for the reader to show that with appropriately chosen $E(p)$ in \eqref{p_moving_A_matrix},
		\begin{equation*}
			h_{1/3}\left(A(p)\cdot\left(\begin{smallmatrix}x\\ y\\ z\end{smallmatrix}\right)\right)=x^3-x(y^2+z^2)+\frac{2}{3\sqrt{3}}y^3 -\frac{2r}{3}yz^2.
		\end{equation*}
	Since $r\in(-\sqrt{3}/2,\sqrt{3})$, this implies that $\mathcal{H}_{1/3}$ is indeed equivalent to $\mathcal{H}_t$ for all $t\in(0,1)$, in particular to $\mathcal{H}_{1/2}$.
\end{enumerate}
\end{Ex}
Note that Example \ref{ex_curves_CCPSRs} shows that the moduli set of CCPSR surfaces when equipped with the topology induced by the $\mathrm{GL}(3)$-action is not Hausdorff. We expect this to be true in all dimensions, and also for CCGPSR manifolds of degree $\tau\geq 4$.

\section{Further applications}

We will now demonstrate how to use Theorem \ref{thm_Cn} to obtain global properties of CCPSR manifolds. The first application is the existence of curvature bounds for the scalar curvature and the sectional curvature:
\begin{Cor}\label{cor_general_CCPSR_S_bounds}
For any fixed dimension $n\geq 2$, there exist $l(n),\, u(n),\, \widetilde{l}(n),\, \widetilde{u}(n)\in\mathbb{R}$, such that $l(n)\leq S_\mathcal{H}\leq u(n)$ and $\widetilde{l}(n)\leq K_\mathcal{H}\leq \widetilde{u}(n)$ for any $n$-dimensional CCPSR manifold $\mathcal{H}$. Here $S_\mathcal{H}$ denotes the scalar curvatures of $\mathcal{H}$ and $K_\mathcal{H}$ denotes all possible sectional curvatures of $\mathcal{H}$.
\begin{proof}
Fix $n\geq 2$ and let $\mathcal{H}$ be a CCPSR manifold in standard form of dimension $n$. By Proposition \ref{prop_scal_GPSR}, equation \eqref{eqn_GPSR_scal_formula}, we know that $S_\mathcal{H}\left(\left(\begin{smallmatrix}1\\ 0\end{smallmatrix}\right)\right)$ depends continuously on $P_3$ and, hence, together with Theorem \ref{thm_Cn} we find that the following expressions are well defined real numbers:
	\begin{align*}
		l(n)&:=\min\limits_{P_3\in\mathrm{pr}_{\mathrm{Sym}(\mathbb{R}^n)^*}\mathcal{C}_n}\left(n(1-n)+\frac{27}{8}\sum\limits_{a,i,\ell}\left(-P_3(\partial_a,\partial_a,\partial_\ell)P_3(\partial_i,\partial_i,\partial_\ell)+P_3(\partial_a,\partial_i,\partial_\ell)^2\right)\right),\\
		u(n)&:=\max\limits_{P_3\in\mathrm{pr}_{\mathrm{Sym}(\mathbb{R}^n)^*}\mathcal{C}_n}\left(n(1-n)+\frac{27}{8}\sum\limits_{a,i,\ell}\left(-P_3(\partial_a,\partial_a,\partial_\ell)P_3(\partial_i,\partial_i,\partial_\ell)+P_3(\partial_a,\partial_i,\partial_\ell)^2\right)\right).
	\end{align*}
Proposition \ref{prop_std_form_h} now implies by the property that we can find a standard form of $\mathcal{H}$ with respect to \textit{any} point $p\in \mathcal{H}$ that the so defined $l(n),\ u(n)\in\mathbb{R}$ fulfil $l(n)\leq S_\mathcal{H}(p)\leq u(n)$ for all $p\in\mathcal{H}$.

For the sectional curvatures, the proof proceeds analogously by using equation \eqref{eqn_sectional_curvature_GPSR} in Lemma \ref{lemma_R_Ric_K_GPSR}.
\end{proof}
\end{Cor}
The study of curvature bounds will be the main topic of an upcoming paper, in which we will expand our studies to manifolds in the image of the supergravity r- and q-map. Recall that, as mentioned in the introduction, sectional curvature bounds are the subject of active study in questions regarding the geometry of K\"ahler cones of Calabi-Yau threefolds \cite{W,TW,KW}.

Next we will give a proof that all closed PSR manifolds $\mathcal{H}$ equipped with their centro-affine fundamental form $g_\mathcal{H}$ are geodesically complete. This was first shown in \cite[Thm.\,2.5]{CNS}. In the corresponding proof it was used that the moduli set of closed PSR curves under the action of $\mathrm{GL}(2)$, which consists precisely of two elements, is compact \cite[Cor.\,4]{CHM}.
Note that geodesically complete PSR manifolds are necessarily closed, since otherwise we could always continuously extend $g_\mathcal{H}$ to each boundary point and construct a geodesic in $\mathcal{H}$ which reaches said point in finite time, cf. \cite[Prop.\,2.4]{CNS}.
We will use the compactness property of $\mathcal{C}_n$ in Theorem \ref{thm_Cn} and the following lemma.

\begin{Lem}
\label{lemma_geodballsfamilymetricscompactembedding}
Let $M$ be manifold of dimension $n\geq 1$ with a locally finite atlas, $C\subset\mathbb{R}^N$ be a compact subset for some $N\in\mathbb{N}$, and $g(\cdot):C\to\Gamma\left(\mathrm{Sym}^2\left(T^*M\right)\right)$, $c\mapsto g(c)$, be a family of Riemannian metrics depending continuously on $c\in C$ in the sense that the map
\begin{equation*}
g:C\times M\to\mathrm{Sym}^2(T^*M),\quad (c,q)\mapsto g(c)_q,
\end{equation*}
is continuous.
Let $p\in M$ be arbitrary and fixed. We denote by $B^{g(c)}_r(p)\subset M$ the geodesic ball of radius $r>0$ around $p\in M$ with respect to the Levi-Civita connection of $g(c)$. Then the following is true:
\begin{equation}\label{eqn_infsupfamilymetrics}
\inf\limits_{c\in C}\left(\sup\limits_{\overline{B^{g(c)}_r(p)}\subset M\mathrm{\ compactly\ embedded}}r\right)>0.
\end{equation}
\begin{proof}
Suppose \eqref{eqn_infsupfamilymetrics} is false. Then there exists a sequence $\{c_i,i\in\mathbb{N}\}\subset C$, such that
\begin{equation*}
\lim\limits_{i\to\infty} \underbrace{\sup\limits_{\overline{B^{g(c_i)}_r(p)}\subset M\text{ compactly embedded}}r}_{=:r_i}=0.
\end{equation*}
Since $C\subset \mathbb{R}^N$ is compact, we can restrict to a subsequence if necessary and assume without loss of generality that $\{c_i,i\in\mathbb{N}\}$ converges to a point $\overline{c}:=\lim\limits_{i\to\infty} c_i$ in $C$. Then, by assumption, 
\begin{equation*}
\sup\limits_{\overline{B^{g\left(\overline{c}\right)}_r(p)}\subset M\text{ compactly embedded}}r=\lim\limits_{i\to\infty}r_i=0.
\end{equation*}
But this is a contradiction to the fact that $g\left(\overline{c}\right)$ is a Riemannian metric and, hence, around every $p\in M$ there exists a positive maximal radius $r>0$, such that $\overline{B^{g\left(\overline{c}\right)}_r(p)}\subset M$ is compactly embedded (recall that independent of the considered Riemannian metric on $M$, the induced metric topology coincides with the given topology on $M$). Hence, \eqref{eqn_infsupfamilymetrics} holds true.
\end{proof}
\end{Lem}

\begin{Prop}
\label{prop_alternative_CCPSR_completeness_proof_numero2}
Closed PSR manifolds are geodesically complete.
\begin{proof}
Let $\mathcal{H}\subset\mathbb{R}^{n+1}$ be an $n$-dimensional PSR manifold that is closed in its ambient space. Assume without loss of generality that $\mathcal{H}$ is connected, i.e. that $\mathcal{H}$ is a CCPSR manifold. Using Theorem \ref{thm_convex_compact_PSR_generating_set}, we can further without loss of generality assume that $\mathcal{H}=\mathcal{H}_{P_3}\subset \left\{h_{P_3}=x^3-x\langle y,y\rangle +P_3(y)=1\right\}$ is the connected component that contains the point $\left(\begin{smallmatrix}
x\\ y
\end{smallmatrix}
\right)=\left(\begin{smallmatrix}
1\\ 0
\end{smallmatrix}
\right)\in \{h=1\}$ and that $P_3\in\left\{ \max\limits_{\|z\|=1}P_3(z)\leq \frac{2}{3\sqrt{3}}\right\}\subset \mathrm{Sym}^3\left(\mathbb{R}^n\right)^*$, where we view the set $\left\{ \max\limits_{\|z\|=1}P_3(z)\leq \frac{2}{3\sqrt{3}}\right\}$ as a compact subset of $\mathbb{R}^N$ with $N=\dim\mathrm{Sym}^3\left(\mathbb{R}^n\right)^*=\frac{n^3+3n^2+2n}{6}$. Consider the set
\begin{equation*}
M:=\bigcap\limits_{P_3\in\left\{ \max\limits_{\|z\|=1}P_3(z)\leq \frac{2}{3\sqrt{3}}\right\}}\mathrm{dom}\left(\mathcal{H}_{P_3}\right).
\end{equation*}
Lemma \ref{lemma_roots_estimates_PSR} implies that $M=\left\{\|z\|<\frac{\sqrt{3}}{2}\right\}\subset\mathbb{R}^n$, in particular $M$ is a smooth submanifold of $\mathbb{R}^n$. Recall that with $\beta_{P_3}(z):=h_{P_3}\left(\left(\begin{smallmatrix}
1\\ z
\end{smallmatrix}\right)\right)$, $\left(\mathcal{H}_{P_3},g_{\mathcal{H}_{P_3}}\right)$ is isometric to
\begin{equation*}
\left(\mathrm{dom}\left(\mathcal{H}_{P_3}\right),-\frac{\partial^2 \beta_{P_3}}{3\beta_{P_3}}+ \frac{2d\beta_{P_3}^2}{9\beta_{P_3}^2}\right)
\end{equation*}
for all $P_3\in\left\{ \max\limits_{\|z\|=1}P_3(z)\leq \frac{2}{3\sqrt{3}}\right\}$, cf. \eqref{Phi_pullback_g_eqn}. Since $M\subset\mathrm{dom}\left(\mathcal{H}_{P_3}\right)$ independent of $P_3$,
we can consider the family of Riemannian metrics on $M$
\begin{equation*}
g(\cdot):P_3\mapsto -\frac{\partial^2 \beta_{P_3}}{3\beta_{P_3}}+ \frac{2d\beta_{P_3}^2}{9\beta_{P_3}^2}.
\end{equation*}
Since $g(\cdot)$ depends continuously on the compact subset $\left\{ \max\limits_{\|z\|=1}P_3(z)\leq \frac{2}{3\sqrt{3}}\right\}\subset \mathrm{Sym}^3\left(\mathbb{R}^n\right)^*$
in the sense of Lemma \ref{lemma_geodballsfamilymetricscompactembedding} (where we identify $\mathrm{Sym}^3\left(\mathbb{R}^n\right)^*$ with $\mathbb{R}^N$ as above and note that $M$ as an open submanifold of $\mathbb{R}^n$ is in particular equipped with a finite atlas consisting of a single chart), we can use Lemma \ref{lemma_geodballsfamilymetricscompactembedding} and obtain that there exists $r>0$, such that $\overline{B^{g\left(P_3\right)}_r(0)}\subset M$
is compactly embedded for all $P_3\in \left\{ \max\limits_{\|z\|=1}P_3(z)\leq \frac{2}{3\sqrt{3}}\right\}$. Together with Proposition \ref{prop_std_form_h} this implies that for all $P_3\in \left\{ \max\limits_{\|z\|=1}P_3(z)\leq \frac{2}{3\sqrt{3}}\right\}$ and all $p\in\mathcal{H}_{P_3}$, $\overline{B^{g_{\mathcal{H}_{P_3}}}_r(p)}\subset \mathcal{H}_{P_3}$ is compactly embedded. Now we use the fact that a Riemannian manifold $(M,g_M)$ is complete if and only if we can find $r>0$, such that $\overline{B_r^{g_M}(p)}$ is compact in $M$ for all $p\in M$, for a proof of that statement see \cite[Lem.\,2.21]{Li}.
We conclude that $\left(\mathcal{H}_{P_3},g_{\mathcal{H}_{P_3}}\right)$ is complete for all $P_3\in \left\{ \max\limits_{\|z\|=1}P_3(z)\leq \frac{2}{3\sqrt{3}}\right\}$.
\end{proof}
\end{Prop}

\end{document}